\documentclass[reqno]{amsart}

\usepackage{graphicx}
\usepackage[top=.75in, bottom=.75in, left=1in, right=1in]{geometry}
\usepackage{ytableau}
\usepackage{mathrsfs}
\usepackage{enumerate}
\usepackage{latexsym}
\usepackage{wrapfig}
\usepackage{float}
\usepackage{verbatim}
\usepackage{placeins,latexsym,amscd,float, amsmath,longtable, amssymb,epsfig,tikz,breqn, array, multirow, mathtools,thmtools,thm-restate,multicol
} 
\usepackage[hidelinks]{hyperref}

\usepackage{url}

\usepackage{bbold}

\usepackage{color}

\usepackage{subcaption}

\usepackage{pgf,tikz,pgfplots, tikz-cd}
\usetikzlibrary{arrows}
\pgfplotsset{compat=1.15}
\usepackage{mathrsfs}
\usetikzlibrary{arrows}
\usepackage[export]{adjustbox}

\definecolor{grey}{rgb}{0.8,0.8,0.8}
\definecolor{purple}{rgb}{.5,0,.6}
\definecolor{forest}{rgb}{0.5,0.8,0.5}

\newcommand{\Z}{\mathbb{Z}}

\def\lT{{\mathscr{T}}}
\def\T{{\mathscr{T}^*}}
\def\Gr{{\textrm{Gr}}}



\newtheorem{theorem}{Theorem}[section]
\newtheorem*{theorem*}{Theorem}
\newtheorem{lemma}[theorem]{Lemma}
\newtheorem{proposition}[theorem]{Proposition}

\theoremstyle{definition}
\newtheorem{definition}[theorem]{Definition}
\newtheorem{example}[theorem]{Example}
\newtheorem{conjecture}[theorem]{Conjecture}
\newtheorem*{conjecture*}{Conjecture}

\theoremstyle{remark}
\newtheorem{remark}[theorem]{Remark}

\numberwithin{equation}{section}

\usepackage{pgf,tikz,pgfplots}
\pgfplotsset{compat=1.15}
\usepackage{mathrsfs}
\usetikzlibrary{arrows}

\begin{document}

\title[Double and Triple Dimer Partition Functions in $\Gr(3,n)$]{Twists of $\Gr(3,n)$ Cluster Variables as \\ Double and Triple Dimer Partition Functions}

\author[M. Elkin]{Moriah Elkin}
\address{University of Minnesota, Twin Cities, Minneapolis, MN}
\email{elkin048@umn.edu}

\author[G. Musiker]{Gregg Musiker}
\address{University of Minnesota, Twin Cities, Minneapolis, MN}
\email{musiker@umn.edu}

\author[K. Wright]{Kayla Wright}
\address{University of Minnesota, Twin Cities, Minneapolis, MN}
\email{kaylaw@umn.edu}

\thanks{}

\begin{abstract}
We give a combinatorial interpretation for certain cluster variables in Grassmannian cluster algebras in terms of double and triple dimer configurations. More specifically, we examine several $\Gr(3,n)$ cluster variables that may be written as degree two or degree three polynomials in terms of Pl\"ucker coordinates, and give generating functions for their images under the twist map - a cluster algebra automorphism introduced in \cite{bfz96}. The generating functions range over certain double or triple dimer configurations on an associated plabic graph, which we describe using particular non-crossing matchings or webs (as in \cite{kup96}), respectively. These connections shed light on a conjecture appearing in \cite{clustering}, extend the concept of web duality introduced in \cite{fll2019}, and more broadly make headway on understanding Grassmannian cluster algebras for $\Gr(3,n)$.
\end{abstract}


\maketitle

\tableofcontents


\section{Introduction}\label{section::intro}
Cluster algebras are a well-loved object of study in algebraic combinatorics because of their deep connection to a myriad of mathematical fields. Many mathematicians are interested in finding combinatorial models for the generators of these algebras, which are otherwise only recursively defined. This paper addresses that question for certain cluster algebras coming from the Grassmannian, which is denoted $\Gr(k,n)$ and refers to the set of $k$-dimensional subspaces of $\mathbb{C}^n$. 
Focusing on the case of $k=3$, we discuss a connection between Grassmannian cluster algebras, $m$-fold dimer configurations and non-elliptic webs. In particular, we establish a dimer-theoretic model for certain generators of Grassmannian cluster algebras.

Scott in \cite{s06} was first to describe the cluster structure on $\Gr(k,n)$, and Postnikov pioneered the study of the combinatorics of these cluster algebras in \cite{p06}. In particular, Postnikov introduced \textit{plabic graphs}, which became the main combinatorial tool for studying these cluster structures. In the same paper, he defined a \textit{boundary measurement map} that linked \textit{Pl\"ucker coordinates}, coordinates that parameterize the Grassmannian as a projective variety and are always generators for the associated cluster algebras, to \textit{dimers} (almost perfect matchings) on plabic graphs. This map was later made more explicit by Talaska in \cite{Tal}. Recently, \cite{marshscott2016} and \cite{mullerspeyer}
used the boundary measurement map to give Laurent expansions for the images of Pl\"ucker coordinates under a certain famous automorphism called the \textit{twist map}, defined originally in \cite{bfz96}.

As shown in \cite[Prop. 8.10]{marshscott2016}, up to multiplication by frozen variables, the twist map on the Grassmannian sends cluster variables to cluster variables.
For $k=2$, every Grassmannian cluster variable is a Pl\"ucker coordinate, and thus every Pl\"ucker coordinate is the twist (up to frozens) of another Pl\"ucker coordinate. However, in Grassmannian cluster algebras with $k \geq 3$ and $n \geq 6$, some cluster variables are more complicated polynomials in Pl\"ucker coordinates, and some Pl\"ucker coordinates only appear as factors in the twists of these cluster variables.
This paper will focus on certain quadratic and cubic polynomials that first appear as cluster variables in $\Gr(3,6)$, $\Gr(3,8)$, and $\Gr(3,9)$, and will describe their twists combinatorially via the boundary measurement map, thus providing Laurent expansions for a larger set of cluster variables than simply twists of Pl\"ucker coordinates.

To do so, we examine the connection between products of $m$ Pl\"ucker coordinates and \textit{$m$-fold dimer configurations} (i.e. superimpositions of $m$ single dimer configurations) elucidated in \cite{lam2015} for $m=2$ and $3$. 2-fold or double dimer configurations may be described as \textit{non-crossing matchings}, and Theorem \ref{theorem:doubledimersXY} describes the matchings associated to the twists of the quadratic cluster variables $X$ and $Y$ defined in Scott. 
To describe the appropriate 3-fold or triple dimer configurations for cubic cluster variables, we require another combinatorial object called a \textit{web}, introduced in \cite{kup96}. Synthesizing novel graph-theoretic reasoning with versions of the results of \cite{marshscott2016} and \cite{mullerspeyer} yields our main theorems, Theorems \ref{theorem:connectivityforA}, \ref{theorem:connectivityforB}, \ref{theorem:connectivityforC}, and \ref{theorem:connectivityforZ}, which describe the twists of several cubic expressions in Pl\"ucker coordinates as corresponding to certain basis webs. We note that the matchings referenced in Theorem \ref{theorem:doubledimersXY} and the webs referenced in Theorems \ref{theorem:connectivityforC} and \ref{theorem:connectivityforZ} first appeared as motivational examples in \cite{fll2019} without the application to Laurent expansions provided here.

This paper is organized as follows. In Section \ref{section::prelims} we give background: we recall the cluster algebra structure of the Grassmannian in Section \ref{subsection: clusteralgbackground} and the terminology of plabic graphs in Section \ref{subsection: plabic}, describe the cluster variables we will model combinatorially in Section \ref{subsection:differences}, define the combinatorial models (dimer configurations) in Section \ref{subsection:dimers}, and define the twist map as a cluster algebra automorphism in Section \ref{subsection:twist}. In Section \ref{section:facewts}, we describe a weighting scheme on dimer configurations, and translate results of \cite{marshscott2016} and \cite{mullerspeyer} into our setting for use in proving our main theorems. Then, in Sections \ref{section:doubledimers} and \ref{section:tripledimers}, we give our double and triple dimer partition functions for twists of quadratic and cubic differences of Pl\"ucker coordinates; Section \ref{section:tripledimers} begins with exposition about webs and their connection to triple dimers, enumerates several relevant classes of webs, and culminates with our main theorems.
In Section \ref{section:duality}, we describe these theorems using the language of web duality introduced in \cite{fll2019}, and provide some explicit computations with standard Young tableaux. 
Finally, in Section \ref{section:constructionofC}, we justify a novel expression for a $\Gr(3,9)$ cluster variable introduced in Section \ref{subsection:differences} and referenced throughout.

We conclude with two appendices. The first, Section \ref{section:appendixlemmas}, contains the lengthy, pictorial proofs of vital lemmas posed in Section \ref{section:tripledimers}. The second, Section \ref{section:appendixAB}, gives explicit computations of Laurent polynomials using our results in Section \ref{section:tripledimers}; it also provides values of the twist map on many cluster variables in $\Gr(3,7)$ and $\Gr(3,8)$.

{\bf Acknowledgements:} We would like to give special thanks to Chris Fraser for assisting with the derivations in Subsection \ref{subsection:twist}, as well as many helpful conversations. We would also like to thank Pavlo Pylyavskyy and David Speyer for sharing their expertise and suggestions, as well as the referee for their careful reading. The first author was initially supported by an Undergraduate Research Opportunities Program award, and subsequently by the University of Minnesota, as her work on this paper partially fulfills the requirements for the degree of Master of Science. The second and third authors were supported by NSF Grants DMS-1745638 and DMS-1854162. 

\section{Preliminaries} \label{section::prelims}
We will use the following notation throughout the paper: let $[n] = \{1,2, \dots, n\}$, and for $1 \leq k \leq n$, let $\binom{[n]}{k}$ denote the set of all $k$-element subsets of $[n]$.

\subsection{The Grassmannian and its Cluster Structure} \label{subsection: clusteralgbackground}
In this section, we briefly review the cluster algebra structure on the Grassmannian; see \cite{fwz-ch6} for a more thorough exposition.

The \textbf{Grassmannian}, denoted $\Gr(k,n)$, is the space of $k$-dimensional linear subspaces of $n$-dimensional complex space. Equivalently, it is the space of full rank $k$-by-$n$ matrices written in row-reduced echelon form, where the corresponding linear subspace is given by the row span of the corresponding matrix. We will consider the \textbf{Pl\"ucker embedding} of $\Gr(k,n)$ into projective space of dimension $\binom{n}{k}-1$, defined as follows:
for any $J \in \binom{[n]}{k}$, we have a corresponding projective \textbf{Pl\"ucker coordinate} $\Delta_J$, and at a given $M \in \text{Mat}_{k \times n}(\mathbb{C})$ we define the value $\Delta_J(M)$ to be the maximal minor of $M$ using the column set $J$. For any $k, n$, the corresponding set of Pl\"ucker coordinates will satisfy certain quadratic \textbf{Pl\"ucker relations}.

\begin{example}\label{example:pluckerrelations}
Consider a generic element $M \in \Gr(2,4)$, represented by the following row-reduced matrix:
$$M = \begin{pmatrix}
1 & 0 & a & b\\
0 & 1 & c & d
\end{pmatrix}~~~~\text{ for } a,b,c,d \in \mathbb{C}$$
The set of Pl\"ucker coordinates is given by:
$$\Delta_{12} = \det \begin{pmatrix}
1 & 0\\
0 & 1
\end{pmatrix} = 1~~~,~~~\Delta_{13} = \det \begin{pmatrix}
1 & a\\
0 & c
\end{pmatrix} = c~~~,~~~\Delta_{14} = \det \begin{pmatrix}
1 & b\\
0 & d
\end{pmatrix} = d$$
$$\Delta_{23} = \det \begin{pmatrix}
0 & a\\
1 & c
\end{pmatrix} = -a~~~,~~~\Delta_{24} = \det \begin{pmatrix}
0 & b\\
1 & d
\end{pmatrix} = -b~~~,~~~\Delta_{34} = \det \begin{pmatrix}
a & b\\
c & d
\end{pmatrix} = ad-bc$$
These coordinates satisfy the algebraic relation
$(-b) \cdot c = 1 \cdot (ad-bc) + (-a) \cdot d,$
i.e.
\[\Delta_{24} \cdot \Delta_{13} = \Delta_{12} \cdot \Delta_{34}+ \Delta_{23} \cdot \Delta_{14}.\]

This equation is in fact the Ptolemy relation from ancient Greek geometry. To see the connection, draw a quadrilateral inscribed in a circle with vertices cyclically labeled $1,2,3,4$, and write $\Delta_{ij}$ for the distance between vertices $i$ and $j$. Then the above relation among the lengths of the sides and diagonals of the quadrilateral will always hold. 
\end{example}

We will consider the homogeneous coordinate ring of $\Gr(k,n)$, denoted $\mathbb{C}[\widehat{\Gr}(k,n)]$,
where $\widehat\Gr(k,n)$ is the affine cone over $\Gr(k,n)$, taken in the Pl\"ucker embedding. Scott showed in \cite{s06} that $\mathbb{C}[\widehat{\Gr}(k,n)]$ is a \textbf{cluster algebra} in the sense of \cite{fz02}.

We refer the reader to \cite{fwz-ch1-3} for an in-depth exposition of cluster algebras, but will briefly summarize here. In essence, a cluster algebra is a commutative ring generated by \textbf{cluster variables}; cluster variables are grouped into families called \textbf{clusters}, and are produced recursively from an initial cluster through a process called \textbf{mutation}. Mutation relations are encoded by a \textbf{quiver}, as follows.
\begin{definition} \label{definition:mutation}
Let $Q$ be a quiver with vertex set $Q_0$ and arrow set $Q_1$. For each mutable vertex $r \in Q_0$, define the \textbf{mutation in direction $\mathbf{r}$} of $Q$, denoted $\mu_r(Q)$, as another quiver on vertices $(Q_0 \setminus \{r\}) \cup \{r'\}$ obtained by the following three-step process:

\begin{itemize}
    \item For any arrow $s \rightarrow r$, draw an arrow $r' \rightarrow s$, and for any arrow $r \rightarrow t$, draw an arrow $t \rightarrow r'$;
    \item For any path $s \to r\to t$, draw an arrow $s \to t$;
    \item Delete any created 2-cycles.
\end{itemize}

We label the new vertex $r'$ according to the following relation:
\[r' r = \prod_{(s \to r) \in Q_1} s + \prod_{(r \to t) \in Q_1} t.\]
\end{definition}
\begin{example}\label{example:initialseed}
The quivers in Figure \ref{fig:quiverandmutation} illustrate the cluster structure in $\mathbb{C}[\widehat{\Gr}(2,5)]$. In the left quiver, the mutable vertices are $\Delta_{13}$ and $\Delta_{35}$; all other vertices are not mutable or \textbf{frozen}, which is indicated by drawing them in boxes or ``ice cubes."  The quiver on the right arises from mutation in direction $\Delta_{13}$, and the new vertex label is
\[\frac{\Delta_{12}\Delta_{35}+\Delta_{23}\Delta_{15}}{\Delta_{13}}=\Delta_{25}.\]

\begin{figure}
\centering
\tikzset{every picture/.style={line width=0.75pt}} 

\begin{tikzpicture}[x=0.75pt,y=0.75pt,yscale=-.9,xscale=.9]

\draw    (131,135) -- (101,135) ;
\draw [shift={(99,135)}, rotate = 360] [color={rgb, 255:red, 0; green, 0; blue, 0 }  ][line width=0.75]    (10.93,-3.29) .. controls (6.95,-1.4) and (3.31,-0.3) .. (0,0) .. controls (3.31,0.3) and (6.95,1.4) .. (10.93,3.29)   ;
\draw    (83,148) -- (95.3,181.13) ;
\draw [shift={(96,183)}, rotate = 249.62] [color={rgb, 255:red, 0; green, 0; blue, 0 }  ][line width=0.75]    (10.93,-3.29) .. controls (6.95,-1.4) and (3.31,-0.3) .. (0,0) .. controls (3.31,0.3) and (6.95,1.4) .. (10.93,3.29)   ;
\draw    (138,184) -- (148.43,148.92) ;
\draw [shift={(149,147)}, rotate = 106.56] [color={rgb, 255:red, 0; green, 0; blue, 0 }  ][line width=0.75]    (10.93,-3.29) .. controls (6.95,-1.4) and (3.31,-0.3) .. (0,0) .. controls (3.31,0.3) and (6.95,1.4) .. (10.93,3.29)   ;
\draw    (82,127) -- (82,104) ;
\draw [shift={(82,102)}, rotate = 90] [color={rgb, 255:red, 0; green, 0; blue, 0 }  ][line width=0.75]    (10.93,-3.29) .. controls (6.95,-1.4) and (3.31,-0.3) .. (0,0) .. controls (3.31,0.3) and (6.95,1.4) .. (10.93,3.29)   ;
\draw    (152,105) -- (152,127) ;
\draw [shift={(152,129)}, rotate = 270] [color={rgb, 255:red, 0; green, 0; blue, 0 }  ][line width=0.75]    (10.93,-3.29) .. controls (6.95,-1.4) and (3.31,-0.3) .. (0,0) .. controls (3.31,0.3) and (6.95,1.4) .. (10.93,3.29)   ;
\draw    (39,135) -- (65,135.93) ;
\draw [shift={(67,136)}, rotate = 182.05] [color={rgb, 255:red, 0; green, 0; blue, 0 }  ][line width=0.75]    (10.93,-3.29) .. controls (6.95,-1.4) and (3.31,-0.3) .. (0,0) .. controls (3.31,0.3) and (6.95,1.4) .. (10.93,3.29)   ;
\draw    (171,136) -- (197,136) ;
\draw [shift={(199,136)}, rotate = 180] [color={rgb, 255:red, 0; green, 0; blue, 0 }  ][line width=0.75]    (10.93,-3.29) .. controls (6.95,-1.4) and (3.31,-0.3) .. (0,0) .. controls (3.31,0.3) and (6.95,1.4) .. (10.93,3.29)   ;
\draw   (71,72) -- (99,72) -- (99,100) -- (71,100) -- cycle ;
\draw   (138,72) -- (166,72) -- (166,100) -- (138,100) -- cycle ;
\draw   (9,120) -- (37,120) -- (37,148) -- (9,148) -- cycle ;
\draw   (103,172) -- (131,172) -- (131,200) -- (103,200) -- cycle ;
\draw   (203,121) -- (231,121) -- (231,149) -- (203,149) -- cycle ;
\draw    (427,135) -- (463,135) ;
\draw [shift={(465,135)}, rotate = 180] [color={rgb, 255:red, 0; green, 0; blue, 0 }  ][line width=0.75]    (10.93,-3.29) .. controls (6.95,-1.4) and (3.31,-0.3) .. (0,0) .. controls (3.31,0.3) and (6.95,1.4) .. (10.93,3.29)   ;
\draw    (422,179) -- (408.74,145.86) ;
\draw [shift={(408,144)}, rotate = 68.2] [color={rgb, 255:red, 0; green, 0; blue, 0 }  ][line width=0.75]    (10.93,-3.29) .. controls (6.95,-1.4) and (3.31,-0.3) .. (0,0) .. controls (3.31,0.3) and (6.95,1.4) .. (10.93,3.29)   ;
\draw    (407,104) -- (407.91,125) ;
\draw [shift={(408,127)}, rotate = 267.51] [color={rgb, 255:red, 0; green, 0; blue, 0 }  ][line width=0.75]    (10.93,-3.29) .. controls (6.95,-1.4) and (3.31,-0.3) .. (0,0) .. controls (3.31,0.3) and (6.95,1.4) .. (10.93,3.29)   ;
\draw    (479,103) -- (479,125) ;
\draw [shift={(479,127)}, rotate = 270] [color={rgb, 255:red, 0; green, 0; blue, 0 }  ][line width=0.75]    (10.93,-3.29) .. controls (6.95,-1.4) and (3.31,-0.3) .. (0,0) .. controls (3.31,0.3) and (6.95,1.4) .. (10.93,3.29)   ;
\draw    (392,133) -- (369,133) ;
\draw [shift={(367,133)}, rotate = 360] [color={rgb, 255:red, 0; green, 0; blue, 0 }  ][line width=0.75]    (10.93,-3.29) .. controls (6.95,-1.4) and (3.31,-0.3) .. (0,0) .. controls (3.31,0.3) and (6.95,1.4) .. (10.93,3.29)   ;
\draw    (498,134) -- (524,134) ;
\draw [shift={(526,134)}, rotate = 180] [color={rgb, 255:red, 0; green, 0; blue, 0 }  ][line width=0.75]    (10.93,-3.29) .. controls (6.95,-1.4) and (3.31,-0.3) .. (0,0) .. controls (3.31,0.3) and (6.95,1.4) .. (10.93,3.29)   ;
\draw   (398,70) -- (426,70) -- (426,98) -- (398,98) -- cycle ;
\draw   (465,70) -- (493,70) -- (493,98) -- (465,98) -- cycle ;
\draw   (336,118) -- (364,118) -- (364,146) -- (336,146) -- cycle ;
\draw   (430,170) -- (458,170) -- (458,198) -- (430,198) -- cycle ;
\draw   (530,119) -- (558,119) -- (558,147) -- (530,147) -- cycle ;
\draw    (466,122) -- (433.54,95.27) ;
\draw [shift={(432,94)}, rotate = 39.47] [color={rgb, 255:red, 0; green, 0; blue, 0 }  ][line width=0.75]    (10.93,-3.29) .. controls (6.95,-1.4) and (3.31,-0.3) .. (0,0) .. controls (3.31,0.3) and (6.95,1.4) .. (10.93,3.29)   ;
\draw    (243,136) .. controls (286.56,95.41) and (287,169.49) .. (321.93,132.17) ;
\draw [shift={(323,131)}, rotate = 131.99] [color={rgb, 255:red, 0; green, 0; blue, 0 }  ][line width=0.75]    (10.93,-3.29) .. controls (6.95,-1.4) and (3.31,-0.3) .. (0,0) .. controls (3.31,0.3) and (6.95,1.4) .. (10.93,3.29)   ;

\draw (68,125.4) node [anchor=north west][inner sep=0.75pt]    {$\Delta _{1}{}_{3}$};
\draw (139,125.4) node [anchor=north west][inner sep=0.75pt]    {$\Delta _{35}$};
\draw (70,75.4) node [anchor=north west][inner sep=0.75pt]    {$\Delta _{2}{}_{3}$};
\draw (139,75.4) node [anchor=north west][inner sep=0.75pt]    {$\Delta _{3}{}_{4}$};
\draw (104,175.4) node [anchor=north west][inner sep=0.75pt]    {$\Delta _{1}{}_{5}$};
\draw (9,125.4) node [anchor=north west][inner sep=0.75pt]    {$\Delta _{1}{}_{2}$};
\draw (203,125.4) node [anchor=north west][inner sep=0.75pt]    {$\Delta _{4}{}_{5}$};
\draw (396,123.4) node [anchor=north west][inner sep=0.75pt]    {$\Delta _{2}{}_{5}$};
\draw (466,123.4) node [anchor=north west][inner sep=0.75pt]    {$\Delta _{35}$};
\draw (397,73.4) node [anchor=north west][inner sep=0.75pt]    {$\Delta _{2}{}_{3}$};
\draw (466,73.4) node [anchor=north west][inner sep=0.75pt]    {$\Delta _{3}{}_{4}$};
\draw (431,173.4) node [anchor=north west][inner sep=0.75pt]    {$\Delta _{1}{}_{5}$};
\draw (336,123.4) node [anchor=north west][inner sep=0.75pt]    {$\Delta _{1}{}_{2}$};
\draw (530,123.4) node [anchor=north west][inner sep=0.75pt]    {$\Delta _{4}{}_{5}$};
\draw (264,162.4) node [anchor=north west][inner sep=0.75pt]    {$\mu_{\Delta_{13}}$};

\end{tikzpicture}
\caption{An initial seed for $\mathbb{C}[\widehat{\Gr(2,5)]}$ and an example of quiver mutation.}
\label{fig:quiverandmutation}
\end{figure}
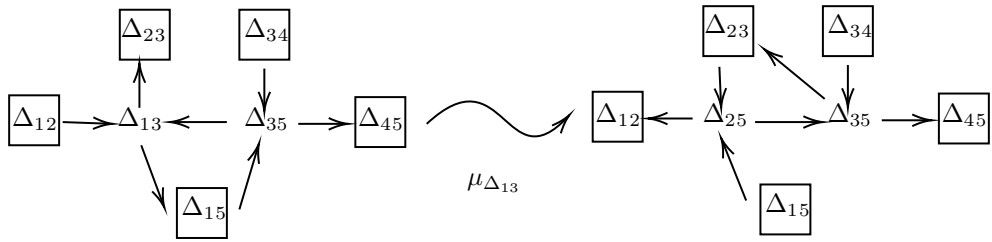

\end{example}

For any $k<n$, the ``rectangles seed" defined in \cite{fwz-ch6} provides an initial quiver and set of cluster variables that generate $\mathbb{C}[\widehat{\Gr}(k,n)]$ as a cluster algebra. 
These initial cluster variables are a subset of the Pl{\"u}cker coordinates, and all Pl{\"u}cker coordinates appear as cluster variables. However, for $k \geq 3$ and $n \geq 6$, certain cluster variables are homogeneous polynomial functions of degree greater than 1 in the Pl{\"u}cker coordinates.
We study several types of non-Pl\"ucker cluster variables throughout this paper.

While lower-dimensional cells of the Grassmannian also admit cluster algebra structures, see for instance \cite{mullerspeyer} which describes the twist map for cases of general positroid varieties, in this paper we focus on seeds associated to the \textit{top cell} of the Grassmannian. In particular, this restriction causes our cluster algebra structure to come equipped with frozen cluster variables, which correspond to determinants of circularly consecutive subsets of columns, i.e. Pl\"uckers of the form $\Delta_{i,i+1,i+2,\dots, i+k}$ where indices are taken modulo $n$.

\subsection{Plabic Graphs}
\label{subsection: plabic}
In this section, we introduce plabic graphs and their connection to the cluster algebra structure on $\mathbb{C}[\widehat{\Gr}(k,n)]$; this study was pioneered by Postnikov in \cite{p06}. See \cite{fwz-ch7} for a more detailed exposition.

\begin{definition}
A \textbf{plabic graph} $G$ is a \textbf{pla}nar \textbf{bic}olored graph embedded in a disk, with $n$ boundary vertices of degree 1, labeled $1,\dots, n$ clockwise. The embedding must be proper, i.e. the edges of $G$ must not cross, and each internal vertex of $G$ must be connected by a path to some boundary vertex of $G$.

All of our plabic graphs will in fact be bipartite, and all of our boundary vertices will be colored black. We will refer to the set of vertices of $G$ as $V(G)$, the set of edges of $G$ as $E(G)$, and the set of faces of $G$ as $F(G)$.
\end{definition}

We next define an important method of labeling the faces of a \textit{reduced} plabic graphs.
\begin{definition} \label{def:trip}
A \textbf{trip} (also called a \textbf{Postnikov strand} or \textbf{zigzag path}) $i \to j$ in a plabic graph is a directed path in $G$ that
\begin{enumerate}
    \item either connects two boundary vertices or is a closed cycle containing no boundary vertices
    \item obeys the \textit{rules of the road} by turning maximally right at black internal vertices and maximally left at white internal vertices.
\end{enumerate}
We may label any face $f \in F(G)$ with the set
\[I_f:=\{i~|~f \text{ lies to the left of the trip ending at vertex }i\}.\]
\end{definition}
As shown in \cite{p06}, the above definition will produce a labeling on the faces of the plabic graph such that each face is labeled by the same number of indices and no two faces have the same label. Note that these conventions for trips and face labels coincide with the conventions in Muller and Speyer's work \cite{mullerspeyer}.

One may construct a reduced plabic graph from certain initial seeds for $\Gr(k,n)$ (such as the rectangles seed of \cite{fwz-ch6}) by reversing the operation described in \cite[Definition 7.1.4]{fwz-ch7}: in particular, by taking the planar dual of the quiver, identifying frozen vertices with boundary faces, and checking that the face labels defined above agree with the vertex labels of the original quiver. Figure \ref{fig:plabic} depicts two examples of the result of such a construction, arising from the quivers in Figure 1.
We note that because we are working in the top cell of $\Gr(k,n)$, all trips as in Definition \ref{def:trip} are of the form $i\to i+k \mod n$, and boundary faces are labeled by circularly consecutive Pl\"ucker coordinates.

\begin{figure}[htb]
\centering

\tikzset{every picture/.style={line width=0.75pt}} 

\begin{tikzpicture}[x=0.75pt,y=0.75pt,yscale=-.75,xscale=.75]

\draw   (22,157) .. controls (22,87.14) and (78.64,30.5) .. (148.5,30.5) .. controls (218.36,30.5) and (275,87.14) .. (275,157) .. controls (275,226.86) and (218.36,283.5) .. (148.5,283.5) .. controls (78.64,283.5) and (22,226.86) .. (22,157) -- cycle ;
\draw  [fill={rgb, 255:red, 0; green, 0; blue, 0 }  ,fill opacity=1 ] (44,240) .. controls (44,236.96) and (46.46,234.5) .. (49.5,234.5) .. controls (52.54,234.5) and (55,236.96) .. (55,240) .. controls (55,243.04) and (52.54,245.5) .. (49.5,245.5) .. controls (46.46,245.5) and (44,243.04) .. (44,240) -- cycle ;
\draw  [fill={rgb, 255:red, 0; green, 0; blue, 0 }  ,fill opacity=1 ] (161,33) .. controls (161,29.96) and (163.46,27.5) .. (166.5,27.5) .. controls (169.54,27.5) and (172,29.96) .. (172,33) .. controls (172,36.04) and (169.54,38.5) .. (166.5,38.5) .. controls (163.46,38.5) and (161,36.04) .. (161,33) -- cycle ;
\draw  [fill={rgb, 255:red, 0; green, 0; blue, 0 }  ,fill opacity=1 ] (268,136) .. controls (268,132.96) and (270.46,130.5) .. (273.5,130.5) .. controls (276.54,130.5) and (279,132.96) .. (279,136) .. controls (279,139.04) and (276.54,141.5) .. (273.5,141.5) .. controls (270.46,141.5) and (268,139.04) .. (268,136) -- cycle ;
\draw   (98.5,117) -- (198.5,117) -- (198.5,197) -- (98.5,197) -- cycle ;
\draw    (148.5,116.5) -- (148.5,197.5) ;
\draw    (49.5,240) -- (98.5,197) ;
\draw    (148.5,116.5) -- (166.5,33) ;
\draw    (198.5,117) -- (273.5,136) ;
\draw  [fill={rgb, 255:red, 0; green, 0; blue, 0 }  ,fill opacity=1 ] (32,96) .. controls (32,92.96) and (34.46,90.5) .. (37.5,90.5) .. controls (40.54,90.5) and (43,92.96) .. (43,96) .. controls (43,99.04) and (40.54,101.5) .. (37.5,101.5) .. controls (34.46,101.5) and (32,99.04) .. (32,96) -- cycle ;
\draw  [fill={rgb, 255:red, 0; green, 0; blue, 0 }  ,fill opacity=1 ] (211,264) .. controls (211,260.96) and (213.46,258.5) .. (216.5,258.5) .. controls (219.54,258.5) and (222,260.96) .. (222,264) .. controls (222,267.04) and (219.54,269.5) .. (216.5,269.5) .. controls (213.46,269.5) and (211,267.04) .. (211,264) -- cycle ;
\draw  [fill={rgb, 255:red, 255; green, 255; blue, 255 }  ,fill opacity=1 ] (93,197) .. controls (93,193.96) and (95.46,191.5) .. (98.5,191.5) .. controls (101.54,191.5) and (104,193.96) .. (104,197) .. controls (104,200.04) and (101.54,202.5) .. (98.5,202.5) .. controls (95.46,202.5) and (93,200.04) .. (93,197) -- cycle ;
\draw  [fill={rgb, 255:red, 0; green, 0; blue, 0 }  ,fill opacity=1 ] (93,117) .. controls (93,113.96) and (95.46,111.5) .. (98.5,111.5) .. controls (101.54,111.5) and (104,113.96) .. (104,117) .. controls (104,120.04) and (101.54,122.5) .. (98.5,122.5) .. controls (95.46,122.5) and (93,120.04) .. (93,117) -- cycle ;
\draw  [fill={rgb, 255:red, 255; green, 255; blue, 255 }  ,fill opacity=1 ] (230.5,126.5) .. controls (230.5,123.46) and (232.96,121) .. (236,121) .. controls (239.04,121) and (241.5,123.46) .. (241.5,126.5) .. controls (241.5,129.54) and (239.04,132) .. (236,132) .. controls (232.96,132) and (230.5,129.54) .. (230.5,126.5) -- cycle ;
\draw    (37.5,96) -- (98.5,117) ;
\draw  [fill={rgb, 255:red, 255; green, 255; blue, 255 }  ,fill opacity=1 ] (62.5,106.5) .. controls (62.5,103.46) and (64.96,101) .. (68,101) .. controls (71.04,101) and (73.5,103.46) .. (73.5,106.5) .. controls (73.5,109.54) and (71.04,112) .. (68,112) .. controls (64.96,112) and (62.5,109.54) .. (62.5,106.5) -- cycle ;
\draw  [fill={rgb, 255:red, 255; green, 255; blue, 255 }  ,fill opacity=1 ] (143,116.5) .. controls (143,113.46) and (145.46,111) .. (148.5,111) .. controls (151.54,111) and (154,113.46) .. (154,116.5) .. controls (154,119.54) and (151.54,122) .. (148.5,122) .. controls (145.46,122) and (143,119.54) .. (143,116.5) -- cycle ;
\draw  [fill={rgb, 255:red, 0; green, 0; blue, 0 }  ,fill opacity=1 ] (193,117) .. controls (193,113.96) and (195.46,111.5) .. (198.5,111.5) .. controls (201.54,111.5) and (204,113.96) .. (204,117) .. controls (204,120.04) and (201.54,122.5) .. (198.5,122.5) .. controls (195.46,122.5) and (193,120.04) .. (193,117) -- cycle ;
\draw    (198.5,197) -- (216.5,264) ;
\draw  [fill={rgb, 255:red, 255; green, 255; blue, 255 }  ,fill opacity=1 ] (193,197) .. controls (193,193.96) and (195.46,191.5) .. (198.5,191.5) .. controls (201.54,191.5) and (204,193.96) .. (204,197) .. controls (204,200.04) and (201.54,202.5) .. (198.5,202.5) .. controls (195.46,202.5) and (193,200.04) .. (193,197) -- cycle ;
\draw  [fill={rgb, 255:red, 0; green, 0; blue, 0 }  ,fill opacity=1 ] (143,197.5) .. controls (143,194.46) and (145.46,192) .. (148.5,192) .. controls (151.54,192) and (154,194.46) .. (154,197.5) .. controls (154,200.54) and (151.54,203) .. (148.5,203) .. controls (145.46,203) and (143,200.54) .. (143,197.5) -- cycle ;
\draw   (380,153) .. controls (380,83.14) and (436.64,26.5) .. (506.5,26.5) .. controls (576.36,26.5) and (633,83.14) .. (633,153) .. controls (633,222.86) and (576.36,279.5) .. (506.5,279.5) .. controls (436.64,279.5) and (380,222.86) .. (380,153) -- cycle ;
\draw  [fill={rgb, 255:red, 0; green, 0; blue, 0 }  ,fill opacity=1 ] (402,236) .. controls (402,232.96) and (404.46,230.5) .. (407.5,230.5) .. controls (410.54,230.5) and (413,232.96) .. (413,236) .. controls (413,239.04) and (410.54,241.5) .. (407.5,241.5) .. controls (404.46,241.5) and (402,239.04) .. (402,236) -- cycle ;
\draw  [fill={rgb, 255:red, 0; green, 0; blue, 0 }  ,fill opacity=1 ] (519,29) .. controls (519,25.96) and (521.46,23.5) .. (524.5,23.5) .. controls (527.54,23.5) and (530,25.96) .. (530,29) .. controls (530,32.04) and (527.54,34.5) .. (524.5,34.5) .. controls (521.46,34.5) and (519,32.04) .. (519,29) -- cycle ;
\draw  [fill={rgb, 255:red, 0; green, 0; blue, 0 }  ,fill opacity=1 ] (626,132) .. controls (626,128.96) and (628.46,126.5) .. (631.5,126.5) .. controls (634.54,126.5) and (637,128.96) .. (637,132) .. controls (637,135.04) and (634.54,137.5) .. (631.5,137.5) .. controls (628.46,137.5) and (626,135.04) .. (626,132) -- cycle ;
\draw   (456.5,113) -- (556.5,113) -- (556.5,193) -- (456.5,193) -- cycle ;
\draw    (506.5,112.5) -- (506.5,193.5) ;
\draw    (407.5,236) -- (456.5,193) ;
\draw    (556.5,113) -- (524.5,29) ;
\draw    (556.5,193) -- (631.5,132) ;
\draw  [fill={rgb, 255:red, 0; green, 0; blue, 0 }  ,fill opacity=1 ] (390,92) .. controls (390,88.96) and (392.46,86.5) .. (395.5,86.5) .. controls (398.54,86.5) and (401,88.96) .. (401,92) .. controls (401,95.04) and (398.54,97.5) .. (395.5,97.5) .. controls (392.46,97.5) and (390,95.04) .. (390,92) -- cycle ;
\draw  [fill={rgb, 255:red, 0; green, 0; blue, 0 }  ,fill opacity=1 ] (569,260) .. controls (569,256.96) and (571.46,254.5) .. (574.5,254.5) .. controls (577.54,254.5) and (580,256.96) .. (580,260) .. controls (580,263.04) and (577.54,265.5) .. (574.5,265.5) .. controls (571.46,265.5) and (569,263.04) .. (569,260) -- cycle ;
\draw  [fill={rgb, 255:red, 255; green, 255; blue, 255 }  ,fill opacity=1 ] (426.5,214.5) .. controls (426.5,211.46) and (428.96,209) .. (432,209) .. controls (435.04,209) and (437.5,211.46) .. (437.5,214.5) .. controls (437.5,217.54) and (435.04,220) .. (432,220) .. controls (428.96,220) and (426.5,217.54) .. (426.5,214.5) -- cycle ;
\draw  [fill={rgb, 255:red, 0; green, 0; blue, 0 }  ,fill opacity=1 ] (451,193) .. controls (451,189.96) and (453.46,187.5) .. (456.5,187.5) .. controls (459.54,187.5) and (462,189.96) .. (462,193) .. controls (462,196.04) and (459.54,198.5) .. (456.5,198.5) .. controls (453.46,198.5) and (451,196.04) .. (451,193) -- cycle ;
\draw  [fill={rgb, 255:red, 255; green, 255; blue, 255 }  ,fill opacity=1 ] (551,113) .. controls (551,109.96) and (553.46,107.5) .. (556.5,107.5) .. controls (559.54,107.5) and (562,109.96) .. (562,113) .. controls (562,116.04) and (559.54,118.5) .. (556.5,118.5) .. controls (553.46,118.5) and (551,116.04) .. (551,113) -- cycle ;
\draw  [fill={rgb, 255:red, 0; green, 0; blue, 0 }  ,fill opacity=1 ] (501,112.5) .. controls (501,109.46) and (503.46,107) .. (506.5,107) .. controls (509.54,107) and (512,109.46) .. (512,112.5) .. controls (512,115.54) and (509.54,118) .. (506.5,118) .. controls (503.46,118) and (501,115.54) .. (501,112.5) -- cycle ;
\draw  [fill={rgb, 255:red, 255; green, 255; blue, 255 }  ,fill opacity=1 ] (588.5,162.5) .. controls (588.5,159.46) and (590.96,157) .. (594,157) .. controls (597.04,157) and (599.5,159.46) .. (599.5,162.5) .. controls (599.5,165.54) and (597.04,168) .. (594,168) .. controls (590.96,168) and (588.5,165.54) .. (588.5,162.5) -- cycle ;
\draw  [fill={rgb, 255:red, 0; green, 0; blue, 0 }  ,fill opacity=1 ] (551,193) .. controls (551,189.96) and (553.46,187.5) .. (556.5,187.5) .. controls (559.54,187.5) and (562,189.96) .. (562,193) .. controls (562,196.04) and (559.54,198.5) .. (556.5,198.5) .. controls (553.46,198.5) and (551,196.04) .. (551,193) -- cycle ;
\draw    (395.5,92) -- (456.5,113) ;
\draw  [fill={rgb, 255:red, 255; green, 255; blue, 255 }  ,fill opacity=1 ] (451,113) .. controls (451,109.96) and (453.46,107.5) .. (456.5,107.5) .. controls (459.54,107.5) and (462,109.96) .. (462,113) .. controls (462,116.04) and (459.54,118.5) .. (456.5,118.5) .. controls (453.46,118.5) and (451,116.04) .. (451,113) -- cycle ;
\draw    (506.5,193.5) -- (574.5,260) ;
\draw  [fill={rgb, 255:red, 255; green, 255; blue, 255 }  ,fill opacity=1 ] (501,193.5) .. controls (501,190.46) and (503.46,188) .. (506.5,188) .. controls (509.54,188) and (512,190.46) .. (512,193.5) .. controls (512,196.54) and (509.54,199) .. (506.5,199) .. controls (503.46,199) and (501,196.54) .. (501,193.5) -- cycle ;
\draw    (291,161) .. controls (330.6,131.3) and (326.1,181.97) .. (364.81,153.88) ;
\draw [shift={(366,153)}, rotate = 143.13] [color={rgb, 255:red, 0; green, 0; blue, 0 }  ][line width=0.75]    (10.93,-3.29) .. controls (6.95,-1.4) and (3.31,-0.3) .. (0,0) .. controls (3.31,0.3) and (6.95,1.4) .. (10.93,3.29)   ;

\draw (19,78.9) node [anchor=north west][inner sep=0.75pt]    {$1$};
\draw (218.5,267.4) node [anchor=north west][inner sep=0.75pt]    {$4$};
\draw (32,247.9) node [anchor=north west][inner sep=0.75pt]    {$5$};
\draw (284,124.9) node [anchor=north west][inner sep=0.75pt]    {$3$};
\draw (162,7.9) node [anchor=north west][inner sep=0.75pt]    {$2$};
\draw (50,154.4) node [anchor=north west][inner sep=0.75pt]    {$12$};
\draw (104,65.4) node [anchor=north west][inner sep=0.75pt]    {$23$};
\draw (188,75.4) node [anchor=north west][inner sep=0.75pt]    {$34$};
\draw (226,165.4) node [anchor=north west][inner sep=0.75pt]    {$45$};
\draw (135,235.4) node [anchor=north west][inner sep=0.75pt]    {$15$};
\draw (114,149.4) node [anchor=north west][inner sep=0.75pt]    {$13$};
\draw (163,151.4) node [anchor=north west][inner sep=0.75pt]    {$35$};
\draw (377,74.9) node [anchor=north west][inner sep=0.75pt]    {$1$};
\draw (576.5,263.4) node [anchor=north west][inner sep=0.75pt]    {$4$};
\draw (390,243.9) node [anchor=north west][inner sep=0.75pt]    {$5$};
\draw (642,120.9) node [anchor=north west][inner sep=0.75pt]    {$3$};
\draw (520,3.9) node [anchor=north west][inner sep=0.75pt]    {$2$};
\draw (405,145.4) node [anchor=north west][inner sep=0.75pt]    {$12$};
\draw (475,64.4) node [anchor=north west][inner sep=0.75pt]    {$23$};
\draw (576,95.4) node [anchor=north west][inner sep=0.75pt]    {$34$};
\draw (575,195.4) node [anchor=north west][inner sep=0.75pt]    {$45$};
\draw (477,226.4) node [anchor=north west][inner sep=0.75pt]    {$15$};
\draw (473,146.4) node [anchor=north west][inner sep=0.75pt]    {$25$};
\draw (521,146.4) node [anchor=north west][inner sep=0.75pt]    {$35$};
\draw (316,184.4) node [anchor=north west][inner sep=0.75pt]    {$\mu_{\Delta_{13}}$};

\end{tikzpicture}
\caption{Two plabic graphs for $\Gr(2,5)$ (with face labeling), arising from the quivers in Figure \ref{fig:quiverandmutation}.}
\label{fig:plabic}
\end{figure}
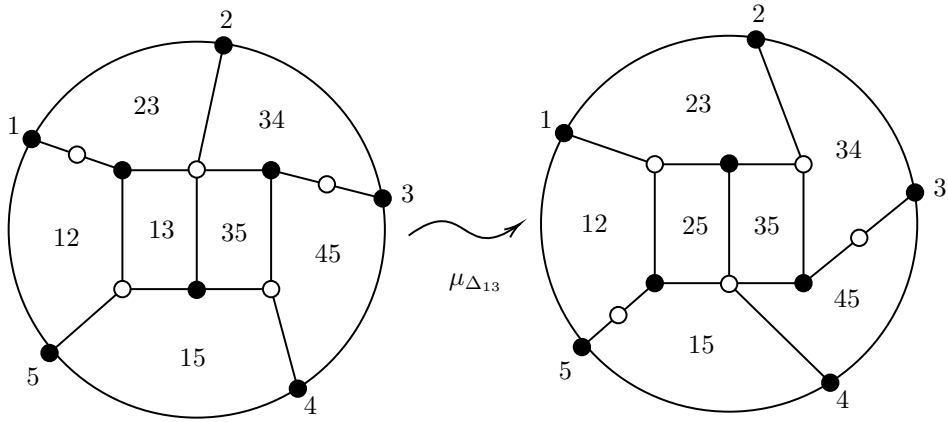

Mutations at a vertex of degree 4 in a quiver correspond to \textbf{square moves} in a plabic graph.
\begin{definition}\label{definition:squaremove}
Let $G$ be a plabic graph, and let $F$ be a square face of $G$ such that each vertex bordering $F$ is trivalent.\footnote{In general, performing a square move also requires that the vertices bordering $F$ alternate in color, but this is immediate from the fact that all plabic graphs we consider are bipartite.} Suppose that $F$ is bounded by strands with sinks $a,b,c,d$ such that $F$ is labeled by the $k$-element subset $Sbd := S \cup \{b,d\}$ for some $S \in \binom{[n]}{k-2}$ and $b,d \in [n] \setminus S$. Define the \textbf{square move} at face $Sbd$ to be the local move on $G$ that swaps the colors of all vertices bordering $Sab$, and updates the label of $Sbd$ to $Sac$. See Figure \ref{fig:squaremove}.
\begin{figure}
\centering
\begin{tikzpicture}[x=0.5pt,y=0.5pt,yscale=-1,xscale=1]

\draw   (70,42) -- (151.5,42) -- (151.5,123.5) -- (70,123.5) -- cycle ;
\draw   (341,40) -- (422.5,40) -- (422.5,121.5) -- (341,121.5) -- cycle ;
\draw    (70,42) -- (40.5,9.97) ;
\draw    (181,155.53) -- (151.5,123.5) ;
\draw    (341,40) -- (311.5,7.97) ;
\draw    (452,153.53) -- (422.5,121.5) ;
\draw    (178.5,9.97) -- (151.5,42) ;
\draw    (70,123.5) -- (43,155.53) ;
\draw    (341,121.5) -- (314,153.53) ;
\draw    (449.5,7.97) -- (422.5,40) ;
\draw  [fill={rgb, 255:red, 255; green, 255; blue, 255 }  ,fill opacity=1 ] (63.75,42) .. controls (63.75,38.55) and (66.55,35.75) .. (70,35.75) .. controls (73.45,35.75) and (76.25,38.55) .. (76.25,42) .. controls (76.25,45.45) and (73.45,48.25) .. (70,48.25) .. controls (66.55,48.25) and (63.75,45.45) .. (63.75,42) -- cycle ;
\draw  [fill={rgb, 255:red, 255; green, 255; blue, 255 }  ,fill opacity=1 ] (145.25,123.5) .. controls (145.25,120.05) and (148.05,117.25) .. (151.5,117.25) .. controls (154.95,117.25) and (157.75,120.05) .. (157.75,123.5) .. controls (157.75,126.95) and (154.95,129.75) .. (151.5,129.75) .. controls (148.05,129.75) and (145.25,126.95) .. (145.25,123.5) -- cycle ;
\draw  [fill={rgb, 255:red, 255; green, 255; blue, 255 }  ,fill opacity=1 ] (416.25,40) .. controls (416.25,36.55) and (419.05,33.75) .. (422.5,33.75) .. controls (425.95,33.75) and (428.75,36.55) .. (428.75,40) .. controls (428.75,43.45) and (425.95,46.25) .. (422.5,46.25) .. controls (419.05,46.25) and (416.25,43.45) .. (416.25,40) -- cycle ;
\draw  [fill={rgb, 255:red, 255; green, 255; blue, 255 }  ,fill opacity=1 ] (334.75,121.5) .. controls (334.75,118.05) and (337.55,115.25) .. (341,115.25) .. controls (344.45,115.25) and (347.25,118.05) .. (347.25,121.5) .. controls (347.25,124.95) and (344.45,127.75) .. (341,127.75) .. controls (337.55,127.75) and (334.75,124.95) .. (334.75,121.5) -- cycle ;
\draw  [fill={rgb, 255:red, 0; green, 0; blue, 0 }  ,fill opacity=1 ] (145.25,42) .. controls (145.25,38.55) and (148.05,35.75) .. (151.5,35.75) .. controls (154.95,35.75) and (157.75,38.55) .. (157.75,42) .. controls (157.75,45.45) and (154.95,48.25) .. (151.5,48.25) .. controls (148.05,48.25) and (145.25,45.45) .. (145.25,42) -- cycle ;
\draw  [fill={rgb, 255:red, 0; green, 0; blue, 0 }  ,fill opacity=1 ] (63.75,123.5) .. controls (63.75,120.05) and (66.55,117.25) .. (70,117.25) .. controls (73.45,117.25) and (76.25,120.05) .. (76.25,123.5) .. controls (76.25,126.95) and (73.45,129.75) .. (70,129.75) .. controls (66.55,129.75) and (63.75,126.95) .. (63.75,123.5) -- cycle ;
\draw  [fill={rgb, 255:red, 0; green, 0; blue, 0 }  ,fill opacity=1 ] (334.75,40) .. controls (334.75,36.55) and (337.55,33.75) .. (341,33.75) .. controls (344.45,33.75) and (347.25,36.55) .. (347.25,40) .. controls (347.25,43.45) and (344.45,46.25) .. (341,46.25) .. controls (337.55,46.25) and (334.75,43.45) .. (334.75,40) -- cycle ;
\draw  [fill={rgb, 255:red, 0; green, 0; blue, 0 }  ,fill opacity=1 ] (416.25,121.5) .. controls (416.25,118.05) and (419.05,115.25) .. (422.5,115.25) .. controls (425.95,115.25) and (428.75,118.05) .. (428.75,121.5) .. controls (428.75,124.95) and (425.95,127.75) .. (422.5,127.75) .. controls (419.05,127.75) and (416.25,124.95) .. (416.25,121.5) -- cycle ;
\draw    (199,81) -- (306.5,80.97) ;
\draw [shift={(308.5,80.97)}, rotate = 179.98] [color={rgb, 255:red, 0; green, 0; blue, 0 }  ][line width=0.75]    (10.93,-3.29) .. controls (6.95,-1.4) and (3.31,-0.3) .. (0,0) .. controls (3.31,0.3) and (6.95,1.4) .. (10.93,3.29)   ;
\draw    (276,81) -- (192.5,80.97) ;
\draw [shift={(190.5,80.97)}, rotate = 0.02] [color={rgb, 255:red, 0; green, 0; blue, 0 }  ][line width=0.75]    (10.93,-3.29) .. controls (6.95,-1.4) and (3.31,-0.3) .. (0,0) .. controls (3.31,0.3) and (6.95,1.4) .. (10.93,3.29)   ;
\draw   (74,247) -- (155.5,247) -- (155.5,328.5) -- (74,328.5) -- cycle ;
\draw   (345,245) -- (426.5,245) -- (426.5,326.5) -- (345,326.5) -- cycle ;
\draw    (74,247) -- (44.5,214.97) ;
\draw    (185,360.53) -- (155.5,328.5) ;
\draw    (345,245) -- (315.5,212.97) ;
\draw    (456,358.53) -- (426.5,326.5) ;
\draw    (182.5,214.97) -- (155.5,247) ;
\draw    (74,328.5) -- (47,360.53) ;
\draw    (345,326.5) -- (318,358.53) ;
\draw    (453.5,212.97) -- (426.5,245) ;
\draw  [fill={rgb, 255:red, 255; green, 255; blue, 255 }  ,fill opacity=1 ] (67.75,247) .. controls (67.75,243.55) and (70.55,240.75) .. (74,240.75) .. controls (77.45,240.75) and (80.25,243.55) .. (80.25,247) .. controls (80.25,250.45) and (77.45,253.25) .. (74,253.25) .. controls (70.55,253.25) and (67.75,250.45) .. (67.75,247) -- cycle ;
\draw  [fill={rgb, 255:red, 255; green, 255; blue, 255 }  ,fill opacity=1 ] (149.25,328.5) .. controls (149.25,325.05) and (152.05,322.25) .. (155.5,322.25) .. controls (158.95,322.25) and (161.75,325.05) .. (161.75,328.5) .. controls (161.75,331.95) and (158.95,334.75) .. (155.5,334.75) .. controls (152.05,334.75) and (149.25,331.95) .. (149.25,328.5) -- cycle ;
\draw  [fill={rgb, 255:red, 255; green, 255; blue, 255 }  ,fill opacity=1 ] (420.25,245) .. controls (420.25,241.55) and (423.05,238.75) .. (426.5,238.75) .. controls (429.95,238.75) and (432.75,241.55) .. (432.75,245) .. controls (432.75,248.45) and (429.95,251.25) .. (426.5,251.25) .. controls (423.05,251.25) and (420.25,248.45) .. (420.25,245) -- cycle ;
\draw  [fill={rgb, 255:red, 255; green, 255; blue, 255 }  ,fill opacity=1 ] (338.75,326.5) .. controls (338.75,323.05) and (341.55,320.25) .. (345,320.25) .. controls (348.45,320.25) and (351.25,323.05) .. (351.25,326.5) .. controls (351.25,329.95) and (348.45,332.75) .. (345,332.75) .. controls (341.55,332.75) and (338.75,329.95) .. (338.75,326.5) -- cycle ;
\draw  [fill={rgb, 255:red, 0; green, 0; blue, 0 }  ,fill opacity=1 ] (149.25,247) .. controls (149.25,243.55) and (152.05,240.75) .. (155.5,240.75) .. controls (158.95,240.75) and (161.75,243.55) .. (161.75,247) .. controls (161.75,250.45) and (158.95,253.25) .. (155.5,253.25) .. controls (152.05,253.25) and (149.25,250.45) .. (149.25,247) -- cycle ;
\draw  [fill={rgb, 255:red, 0; green, 0; blue, 0 }  ,fill opacity=1 ] (67.75,328.5) .. controls (67.75,325.05) and (70.55,322.25) .. (74,322.25) .. controls (77.45,322.25) and (80.25,325.05) .. (80.25,328.5) .. controls (80.25,331.95) and (77.45,334.75) .. (74,334.75) .. controls (70.55,334.75) and (67.75,331.95) .. (67.75,328.5) -- cycle ;
\draw  [fill={rgb, 255:red, 0; green, 0; blue, 0 }  ,fill opacity=1 ] (338.75,245) .. controls (338.75,241.55) and (341.55,238.75) .. (345,238.75) .. controls (348.45,238.75) and (351.25,241.55) .. (351.25,245) .. controls (351.25,248.45) and (348.45,251.25) .. (345,251.25) .. controls (341.55,251.25) and (338.75,248.45) .. (338.75,245) -- cycle ;
\draw  [fill={rgb, 255:red, 0; green, 0; blue, 0 }  ,fill opacity=1 ] (420.25,326.5) .. controls (420.25,323.05) and (423.05,320.25) .. (426.5,320.25) .. controls (429.95,320.25) and (432.75,323.05) .. (432.75,326.5) .. controls (432.75,329.95) and (429.95,332.75) .. (426.5,332.75) .. controls (423.05,332.75) and (420.25,329.95) .. (420.25,326.5) -- cycle ;
\draw    (224,285.97) -- (289,285.97) ;
\draw [shift={(291,285.97)}, rotate = 180] [color={rgb, 255:red, 0; green, 0; blue, 0 }  ][line width=0.75]    (10.93,-3.29) .. controls (6.95,-1.4) and (3.31,-0.3) .. (0,0) .. controls (3.31,0.3) and (6.95,1.4) .. (10.93,3.29)   ;
\draw [color={rgb, 255:red, 208; green, 2; blue, 27 }  ,draw opacity=1 ]   (199.75,359.75) .. controls (160.75,323.75) and (78,355.97) .. (80,304.97) ;
\draw [color={rgb, 255:red, 208; green, 2; blue, 27 }  ,draw opacity=1 ]   (80,304.97) .. controls (80.99,267.35) and (52.58,264.03) .. (57.83,207.69) ;
\draw [shift={(58,205.97)}, rotate = 95.91] [color={rgb, 255:red, 208; green, 2; blue, 27 }  ,draw opacity=1 ][line width=0.75]    (10.93,-3.29) .. controls (6.95,-1.4) and (3.31,-0.3) .. (0,0) .. controls (3.31,0.3) and (6.95,1.4) .. (10.93,3.29)   ;
\draw [color={rgb, 255:red, 126; green, 211; blue, 33 }  ,draw opacity=1 ]   (38,356.97) .. controls (66,319.97) and (145,366.97) .. (159,275.97) ;
\draw [color={rgb, 255:red, 126; green, 211; blue, 33 }  ,draw opacity=1 ]   (159,275.97) .. controls (176.73,242.48) and (147.89,226.45) .. (173.78,205.91) ;
\draw [shift={(175,204.97)}, rotate = 143.13] [color={rgb, 255:red, 126; green, 211; blue, 33 }  ,draw opacity=1 ][line width=0.75]    (10.93,-3.29) .. controls (6.95,-1.4) and (3.31,-0.3) .. (0,0) .. controls (3.31,0.3) and (6.95,1.4) .. (10.93,3.29)   ;
\draw [color={rgb, 255:red, 189; green, 16; blue, 224 }  ,draw opacity=1 ]   (30,223.97) .. controls (92,249.97) and (144,228.97) .. (161,284.97) ;
\draw [color={rgb, 255:red, 189; green, 16; blue, 224 }  ,draw opacity=1 ]   (161,284.97) .. controls (183.66,332.25) and (141.3,352.36) .. (171.57,370.16) ;
\draw [shift={(173,370.97)}, rotate = 208.61] [color={rgb, 255:red, 189; green, 16; blue, 224 }  ,draw opacity=1 ][line width=0.75]    (10.93,-3.29) .. controls (6.95,-1.4) and (3.31,-0.3) .. (0,0) .. controls (3.31,0.3) and (6.95,1.4) .. (10.93,3.29)   ;
\draw [color={rgb, 255:red, 74; green, 144; blue, 226 }  ,draw opacity=1 ]   (60,284.97) .. controls (98,201.97) and (185,250.97) .. (200,224.97) ;
\draw [color={rgb, 255:red, 74; green, 144; blue, 226 }  ,draw opacity=1 ]   (61,284.97) .. controls (44.59,316.81) and (73.82,348.66) .. (68.7,365.23) ;
\draw [shift={(68,366.97)}, rotate = 296.57] [color={rgb, 255:red, 74; green, 144; blue, 226 }  ,draw opacity=1 ][line width=0.75]    (10.93,-3.29) .. controls (6.95,-1.4) and (3.31,-0.3) .. (0,0) .. controls (3.31,0.3) and (6.95,1.4) .. (10.93,3.29)   ;
\draw (99,277.97) node [anchor=north west][inner sep=0.75pt]   [align=left] {Sbd};
\draw (96,212.97) node [anchor=north west][inner sep=0.75pt]   [align=left] {Sbc};
\draw (163,276.97) node [anchor=north west][inner sep=0.75pt]   [align=left] {Scd};
\draw (29,277.97) node [anchor=north west][inner sep=0.75pt]   [align=left] {Sab};
\draw (100,340.97) node [anchor=north west][inner sep=0.75pt]   [align=left] {Sad};
\draw (369,212.97) node [anchor=north west][inner sep=0.75pt]   [align=left] {Sbc};
\draw (306,276.97) node [anchor=north west][inner sep=0.75pt]   [align=left] {Sab};
\draw (372,277.97) node [anchor=north west][inner sep=0.75pt]   [align=left] {Sac};
\draw (373,338.97) node [anchor=north west][inner sep=0.75pt]   [align=left] {Sad};
\draw (437,277.97) node [anchor=north west][inner sep=0.75pt]   [align=left] {Scd};
\draw (44,187.97) node [anchor=north west][inner sep=0.75pt]   [align=left] {\textcolor[rgb]{0.82,0.01,0.11}{a}};
\draw (179,187.97) node [anchor=north west][inner sep=0.75pt]   [align=left] {\textcolor[rgb]{0.49,0.83,0.13}{b}};
\draw (175,373.97) node [anchor=north west][inner sep=0.75pt]   [align=left] {\textcolor[rgb]{0.74,0.06,0.88}{c}};
\draw (55,371.97) node [anchor=north west][inner sep=0.75pt]   [align=left] {\textcolor[rgb]{0.29,0.56,0.89}{d}};
\end{tikzpicture}
\caption{A square move, the plabic graph analogue of quiver mutation at a degree 4 vertex. Note that the new face label satisfies the Pl{\"u}cker relation $\Delta_{Sac} = \dfrac{\Delta_{Sab}\Delta_{Scd} + \Delta_{Sad}\Delta_{Sbc}}{\Delta_{Sbd}}$.}
\label{fig:squaremove}
\end{figure}
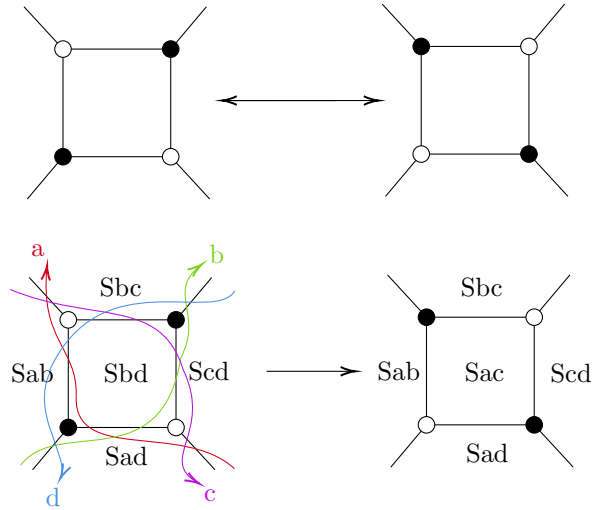
\end{definition}

We note that if one begins with a reduced plabic graph associated to a Pl\"ucker seed and applies a sequence of square moves, the resulting plabic graph will again correspond to a Pl\"ucker seed. Our work concerns non-Pl\"ucker cluster variables, which arise from quiver mutations at vertices with valence greater than 4.

\subsection{Quadratic and Cubic Differences of Pl\"ucker Coordinates}\label{subsection:differences}
This paper will focus on Grassmannian cluster algebras in the case $k=3$. In this section, we introduce notation for several classes of degree 2 and 3 cluster variables that appear in this setting, as well as another distinguished degree 3 polynomial in Pl{\"u}cker coordinates.

We begin by considering the case $n=6$; the corresponding cluster algebra is $\mathbb{C}[\widehat{\Gr}(3,6)]$, of finite type $D_4$. We adopt the conventions of \cite{s06}, and write $X$ to refer to the compound determinant 
\[\det \bigg( v_1 \times v_2 \quad v_3 \times v_4 \quad v_5 \times v_6\bigg).\]
Here $v_i$ denotes the $i$th column of $M \in \Gr(3,6)$, treated as a vector in $\mathbb{R}^3$, and $\times$ denotes the usual three-dimensional cross-product (taking two $\mathbb{R}^3$-vectors as input and outputting an $\mathbb{R}^3$-vector).  Using cross-product identities (2) and (3) from \cite{s06}, i.e. 
$$ u \cdot (v \times w) = (u \times v) \cdot w = \det ( u \quad u \quad w) \mathrm{~and~}
(u \times v) \cdot (w \times z) = \bigg( \begin{matrix} u\cdot w & u\cdot z \\ v\cdot w & v\cdot z\end{matrix}\bigg),$$
we can re-express $X$ as the quadratic difference
$\Delta_{134}\Delta_{256}-\Delta_{156}\Delta_{234},$
which may also be written as $\Delta_{124}\Delta_{356}-\Delta_{123}\Delta_{456}$ or $\Delta_{125}\Delta_{346}-\Delta_{126}\Delta_{345}.$
Analogously, we write
$Y$ to refer to the compound determinant
$$\det \bigg( v_6 \times v_1 \quad v_2 \times v_3 \quad v_4 \times v_5\bigg),$$ which can be re-expressed as a quadratic difference as
$\Delta_{145}\Delta_{236}-\Delta_{123}\Delta_{456}$,
$\Delta_{146}\Delta_{235}-\Delta_{156}\Delta_{234}$, or $\Delta_{136}\Delta_{245}-\Delta_{126}\Delta_{345}.$
Scott observes in \cite{s06} that $X$ and $Y$ appear as cluster variables for $\mathbb{C}[\widehat{\Gr}(3,6)]$.

When $n=8$, we have the cluster algebra $\mathbb{C}[\widehat{\Gr}(3,8)]$ of finite type $E_8$. Scott shows in this case that all dihedral translates of the following cubics appear as cluster variables:
\[A=\Delta_{134}\Delta_{258}\Delta_{167}-\Delta_{134}\Delta_{678}\Delta_{125}-\Delta_{158}\Delta_{234}\Delta_{167}\]
and
\[B=\Delta_{258}\Delta_{134}\Delta_{267}-\Delta_{234}\Delta_{128}\Delta_{567}-\Delta_{234}\Delta_{258}\Delta_{167}.\]
We note that when the dihedral group $D_8$ acts on the indices of the cubic function $A$ in $\Gr(3,8)$, the image is only of size $8$, i.e. only the cyclic translates.  For example, if we apply the reflection $\rho : i \to 9-i$ to each entry, we get $\rho(A) = \Delta_{568}\Delta_{147}\Delta_{238}-\Delta_{568}\Delta_{123}\Delta_{478}-\Delta_{148}\Delta_{567}\Delta_{238} = \sigma^7(A)$ where $\sigma : i \to i+1 \mod 8$.  On the other hand, the image of $D_8$'s action on the cubic function $B$ is indeed of size $16$. 

When $n=9$, the corresponding cluster algebra $\mathbb{C}[\widehat{\Gr}(3,9)]$ is of affine type. We will prove in Section \ref{section:constructionofC} that all dihedral translates of the following expression appear as cluster variables:
\[C = \Delta_{124}\Delta_{357}\Delta_{689}+\Delta_{123}\Delta_{456}\Delta_{789}-\Delta_{124}\Delta_{356}\Delta_{789}-\Delta_{123}\Delta_{457}\Delta_{689}.\]
Note that the image of the action of the dihedral group $D_9$ on $C$ has size $9$, since $C$ is invariant under the reflection $\rho: i \to 10-i$, i.e. $\rho(C)=C$.

We will also consider the expression
\[Z= \Delta_{145}\Delta_{278}\Delta_{369}-\Delta_{245}\Delta_{178}\Delta_{369}-\Delta_{123}\Delta_{456}\Delta_{789}-\Delta_{129}\Delta_{345}\Delta_{678} \in \mathbb{C}[\widehat{\Gr}(3,9)],\]
which appears in Example 8.1 of \cite{quantum20}. Its significance arises from the fact that every cluster monomial arises from a standard Young tableau, but not every standard Young tableau arises from a cluster monomial. The ones that do arise from cluster monomials are called {\bf real} tableaux, due to their manifestation in quantum affine algebras. In finite type, every Young tableau is real, but since $\Gr(3,9)$ is of affine type, some non-real tableaux appear. $Z$ is the lowest-degree element of $\Gr(3,9)$ that arises from a non-real tableau. We note that $\sigma^3(Z)=Z$ and $\rho(Z) = \sigma^8(Z)$, so the image of $D_9$ acting on $Z$ is simply of size $3$.

To extend these classes of cluster variables, we increase $n$. Given $n \geq n' \geq 3$ and a subset $I \subseteq [n]$ with $|I|=n'$, we define a projection $\pi_{n,I}:\Gr(3,n) \rightarrow \Gr(3,n')$ that retains exactly the columns indexed by $I$ of any matrix $M \in \Gr(3,n)$. We have from \cite{s06} that in the case of $n'=6$, for any $n\geq 6$ and $I \subseteq [n]$ with $|I|=6$, the expressions $X^I:=X \circ \pi_{6,I}$ and $Y^I:=Y \circ \pi_{6,I}$ appear as cluster variables of $\Gr(3,n)$. Theorem 8.8 of \cite{fominpasha2016} implies the analogous statement for expressions that project to $A$, $B$, and $C$.
The following widely expected conjecture, a reformulation of Conjecture 3.2 of \cite{clustering} in the case of $\Gr(3,n)$, generalizes this fact.
\begin{conjecture}\label{conj:projections}
Given $n \geq n' \geq 3$ and a subset $I \subseteq [n]$ with $|I|=n'$, we have that $x$ is a cluster variable in $\mathbb{C}[\widehat{\Gr}(3,n')]$ if and only if $\pi_{n,I} \circ x$ is a cluster variable in $\mathbb{C}[\widehat{\Gr}(3,n)]$.
\end{conjecture}

We note that the above expressions for $A$, $B$, and $C$ shed light on Conjecture 3.1 of \cite{clustering}, which in part posits that there are $24\binom{n}{8}+9\binom{n}{9}$ degree 3 cluster variables in $\Gr(3,n)$. Indeed, for any $n$ we have described exactly $24\binom{n}{8}+9\binom{n}{9}$ such cluster variables: $8\binom{n}{8}$ dihedral translates and projections of $A$, $16\binom{n}{8}$ dihedral translates and projections of $B$, and $9\binom{n}{9}$ dihedral translates and projections of $C$. It remains to show that there are no other degree 3 cluster variables.

\subsection{Dimer Configurations}\label{subsection:dimers}
We discuss the combinatorial model that dimers on plabic graphs provide for Pl\"ucker coordinates. This connection was discovered by Postnikov in \cite{p06} and developed by Talaska in \cite{Tal}. We also define $m$-fold dimers; in Sections \ref{section:doubledimers} and \ref{section:tripledimers}, we will describe how these objects extend the model to products of $m$ Pl\"ucker coordinates.

\begin{definition}\label{definition:dimer}
A \textbf{dimer configuration} (also called an \textbf{almost perfect matching}) $D$ on a plabic graph $G$ is a subset of edges of $G$ such that
\begin{enumerate}
    \item each interior vertex of $G$ is adjacent to exactly one edge in $D$, and
    \item each boundary vertex of $G$ is adjacent to either no edges or one edge in $D$.
\end{enumerate}
Let $\partial(D)$ be the set of boundary vertices adjacent to one edge in $D$; we call $\partial(D)$ the \textbf{boundary condition} for $D$.
Also let $\mathcal{D}(G)$ denote the set of dimer configurations on $G$, and define \[\mathcal{D}_J(G):=\{D \in \mathcal{D}(G)~|~\partial(D)=J\}.\]
\end{definition}

Given a plabic graph $G$ for $\Gr(k,n)$, we may assign
nonnegative real weights to its edges, and define the \textbf{edge weight} of any dimer $D$ to be the product of the weights of the edges it contains:
\[\text{wt}_e(D) = \prod_{e \in D} \text{wt}(e).\]
The following theorem, stated concisely in \cite{lam2015} with references to other work, relates dimer configurations to points in the Grassmannian.
\begin{theorem}[\cite{kuo04},\cite{p06},\cite{Tal},\cite{PSW},\cite{lam2015}]\label{theorem:edgewts}
Let $G$ be a plabic graph with black boundary vertices $1,2, \dots, n$, and let $k$ be the number of internal white vertices minus the number of internal black vertices in $G$. Also let $\text{wt}: E(G) \longrightarrow \mathbb{R}_{\geq 0}$ be any weight function on the edges of $G$. Then there exists some $\tilde{M}$ in the affine cone over the Grassmannian $\widetilde{\Gr}(k,n)$ such that for all $J \in \binom{[n]}{k}$,
$$\Delta_J(\tilde{M}) = \sum_{D \in \mathcal{D}_J(G)} \text{wt}_e(D).$$
\end{theorem}
\noindent Here the affine cone arises since Pl\"ucker coordinates embed the Grassmannian into projective space, so the value of an individual $\Delta_J(M)$ is not well-defined in $\mathbb{R}$. To instead arrive at a point in the Grassmannian, we may identify edge weight functions that yield the same sets of Pl\"uckers up to scaling, or alternatively take the equivalence class of any $\tilde{M}$.

\begin{example}\label{example:almostperfmatching}
Consider the weighted plabic graph in Figure \ref{fig:almostperfmatching}. The edges highlighted in red form a dimer $D$ with $\partial(D) = \{1,4\}$ and edge weight $\text{wt}_e(D)=achnk$.
\begin{figure}[htb] 
\centering

\tikzset{every picture/.style={line width=0.75pt}} 

\begin{tikzpicture}[x=0.75pt,y=0.75pt,yscale=-1,xscale=1]

\draw   (187,153.5) .. controls (187,83.64) and (243.64,27) .. (313.5,27) .. controls (383.36,27) and (440,83.64) .. (440,153.5) .. controls (440,223.36) and (383.36,280) .. (313.5,280) .. controls (243.64,280) and (187,223.36) .. (187,153.5) -- cycle ;
\draw  [fill={rgb, 255:red, 0; green, 0; blue, 0 }  ,fill opacity=1 ] (209,236.5) .. controls (209,233.46) and (211.46,231) .. (214.5,231) .. controls (217.54,231) and (220,233.46) .. (220,236.5) .. controls (220,239.54) and (217.54,242) .. (214.5,242) .. controls (211.46,242) and (209,239.54) .. (209,236.5) -- cycle ;
\draw  [fill={rgb, 255:red, 0; green, 0; blue, 0 }  ,fill opacity=1 ] (326,29.5) .. controls (326,26.46) and (328.46,24) .. (331.5,24) .. controls (334.54,24) and (337,26.46) .. (337,29.5) .. controls (337,32.54) and (334.54,35) .. (331.5,35) .. controls (328.46,35) and (326,32.54) .. (326,29.5) -- cycle ;
\draw  [fill={rgb, 255:red, 0; green, 0; blue, 0 }  ,fill opacity=1 ] (433,132.5) .. controls (433,129.46) and (435.46,127) .. (438.5,127) .. controls (441.54,127) and (444,129.46) .. (444,132.5) .. controls (444,135.54) and (441.54,138) .. (438.5,138) .. controls (435.46,138) and (433,135.54) .. (433,132.5) -- cycle ;
\draw   (263.5,113.5) -- (363.5,113.5) -- (363.5,193.5) -- (263.5,193.5) -- cycle ;
\draw    (313.5,113) -- (313.5,194) ;
\draw [color={rgb, 255:red, 208; green, 2; blue, 27 }  ,draw opacity=1 ][line width=4.5]    (202.5,92.5) -- (263.5,113.5) ;
\draw    (214.5,236.5) -- (263.5,193.5) ;
\draw    (363.5,113.5) -- (331.5,29.5) ;
\draw    (363.5,193.5) -- (438.5,132.5) ;
\draw [color={rgb, 255:red, 208; green, 2; blue, 27 }  ,draw opacity=1 ][line width=3.75]    (313.5,194) -- (381.5,260.5) ;
\draw  [fill={rgb, 255:red, 255; green, 255; blue, 255 }  ,fill opacity=1 ] (258,113.5) .. controls (258,110.46) and (260.46,108) .. (263.5,108) .. controls (266.54,108) and (269,110.46) .. (269,113.5) .. controls (269,116.54) and (266.54,119) .. (263.5,119) .. controls (260.46,119) and (258,116.54) .. (258,113.5) -- cycle ;
\draw  [fill={rgb, 255:red, 255; green, 255; blue, 255 }  ,fill opacity=1 ] (308,194) .. controls (308,190.96) and (310.46,188.5) .. (313.5,188.5) .. controls (316.54,188.5) and (319,190.96) .. (319,194) .. controls (319,197.04) and (316.54,199.5) .. (313.5,199.5) .. controls (310.46,199.5) and (308,197.04) .. (308,194) -- cycle ;
\draw  [fill={rgb, 255:red, 0; green, 0; blue, 0 }  ,fill opacity=1 ] (197,92.5) .. controls (197,89.46) and (199.46,87) .. (202.5,87) .. controls (205.54,87) and (208,89.46) .. (208,92.5) .. controls (208,95.54) and (205.54,98) .. (202.5,98) .. controls (199.46,98) and (197,95.54) .. (197,92.5) -- cycle ;
\draw  [fill={rgb, 255:red, 0; green, 0; blue, 0 }  ,fill opacity=1 ] (376,260.5) .. controls (376,257.46) and (378.46,255) .. (381.5,255) .. controls (384.54,255) and (387,257.46) .. (387,260.5) .. controls (387,263.54) and (384.54,266) .. (381.5,266) .. controls (378.46,266) and (376,263.54) .. (376,260.5) -- cycle ;
\draw [color={rgb, 255:red, 208; green, 2; blue, 27 }  ,draw opacity=1 ][fill={rgb, 255:red, 208; green, 2; blue, 27 }  ,fill opacity=1 ][line width=3.75]    (239,215) -- (263.5,193.5) ;
\draw  [fill={rgb, 255:red, 255; green, 255; blue, 255 }  ,fill opacity=1 ] (233.5,215) .. controls (233.5,211.96) and (235.96,209.5) .. (239,209.5) .. controls (242.04,209.5) and (244.5,211.96) .. (244.5,215) .. controls (244.5,218.04) and (242.04,220.5) .. (239,220.5) .. controls (235.96,220.5) and (233.5,218.04) .. (233.5,215) -- cycle ;
\draw  [fill={rgb, 255:red, 0; green, 0; blue, 0 }  ,fill opacity=1 ] (258,193.5) .. controls (258,190.46) and (260.46,188) .. (263.5,188) .. controls (266.54,188) and (269,190.46) .. (269,193.5) .. controls (269,196.54) and (266.54,199) .. (263.5,199) .. controls (260.46,199) and (258,196.54) .. (258,193.5) -- cycle ;
\draw [color={rgb, 255:red, 208; green, 2; blue, 27 }  ,draw opacity=1 ][fill={rgb, 255:red, 208; green, 2; blue, 27 }  ,fill opacity=1 ][line width=3.75]    (313.5,113) -- (363.5,113.5) ;
\draw  [fill={rgb, 255:red, 255; green, 255; blue, 255 }  ,fill opacity=1 ] (358,113.5) .. controls (358,110.46) and (360.46,108) .. (363.5,108) .. controls (366.54,108) and (369,110.46) .. (369,113.5) .. controls (369,116.54) and (366.54,119) .. (363.5,119) .. controls (360.46,119) and (358,116.54) .. (358,113.5) -- cycle ;
\draw  [fill={rgb, 255:red, 0; green, 0; blue, 0 }  ,fill opacity=1 ] (308,113) .. controls (308,109.96) and (310.46,107.5) .. (313.5,107.5) .. controls (316.54,107.5) and (319,109.96) .. (319,113) .. controls (319,116.04) and (316.54,118.5) .. (313.5,118.5) .. controls (310.46,118.5) and (308,116.04) .. (308,113) -- cycle ;
\draw [color={rgb, 255:red, 208; green, 2; blue, 27 }  ,draw opacity=1 ][fill={rgb, 255:red, 208; green, 2; blue, 27 }  ,fill opacity=1 ][line width=3.75]    (363.5,193.5) -- (401,163) ;
\draw  [fill={rgb, 255:red, 255; green, 255; blue, 255 }  ,fill opacity=1 ] (395.5,163) .. controls (395.5,159.96) and (397.96,157.5) .. (401,157.5) .. controls (404.04,157.5) and (406.5,159.96) .. (406.5,163) .. controls (406.5,166.04) and (404.04,168.5) .. (401,168.5) .. controls (397.96,168.5) and (395.5,166.04) .. (395.5,163) -- cycle ;
\draw  [fill={rgb, 255:red, 0; green, 0; blue, 0 }  ,fill opacity=1 ] (358,193.5) .. controls (358,190.46) and (360.46,188) .. (363.5,188) .. controls (366.54,188) and (369,190.46) .. (369,193.5) .. controls (369,196.54) and (366.54,199) .. (363.5,199) .. controls (360.46,199) and (358,196.54) .. (358,193.5) -- cycle ;

\draw (184,75.4) node [anchor=north west][inner sep=0.75pt]    {$1$};
\draw (383.5,263.9) node [anchor=north west][inner sep=0.75pt]    {$4$};
\draw (197,244.4) node [anchor=north west][inner sep=0.75pt]    {$5$};
\draw (449,121.4) node [anchor=north west][inner sep=0.75pt]    {$3$};
\draw (327,4.4) node [anchor=north west][inner sep=0.75pt]    {$2$};
\draw (230,84.4) node [anchor=north west][inner sep=0.75pt]    {$a$};
\draw (282,96.4) node [anchor=north west][inner sep=0.75pt]    {$b$};
\draw (333,95.4) node [anchor=north west][inner sep=0.75pt]    {$c$};
\draw (351,62.4) node [anchor=north west][inner sep=0.75pt]    {$d$};
\draw (251,143.4) node [anchor=north west][inner sep=0.75pt]    {$e$};
\draw (316,144.4) node [anchor=north west][inner sep=0.75pt]    {$f$};
\draw (366,142.4) node [anchor=north west][inner sep=0.75pt]    {$g$};
\draw (240,187.4) node [anchor=north west][inner sep=0.75pt]    {$h$};
\draw (282,194.4) node [anchor=north west][inner sep=0.75pt]    {$i$};
\draw (335,194.4) node [anchor=north west][inner sep=0.75pt]    {$j$};
\draw (383,177.4) node [anchor=north west][inner sep=0.75pt]    {$k$};
\draw (418,149.4) node [anchor=north west][inner sep=0.75pt]    {$\ell $};
\draw (214,208.4) node [anchor=north west][inner sep=0.75pt]    {$m$};
\draw (334,227.4) node [anchor=north west][inner sep=0.75pt]    {$n$};

\end{tikzpicture}
\caption{An example of a dimer $D$ with $\partial D = \{1,4\}$ on the plabic graph at the right of Figure \ref{fig:plabic}.}
\label{fig:almostperfmatching}
\end{figure}
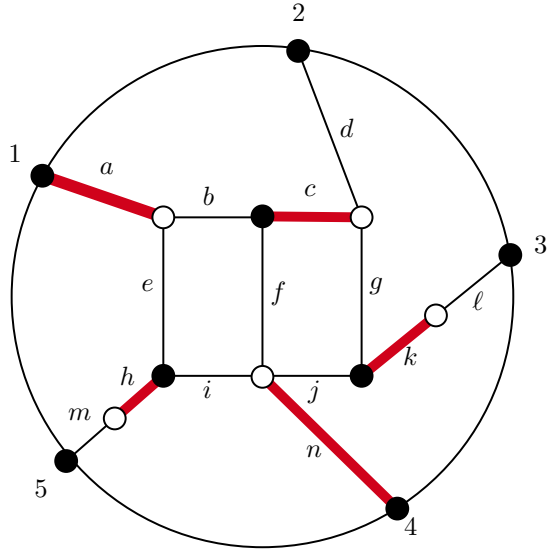
\end{example}

To model products of Pl\"ucker coordinates, we define the notion of higher dimer configurations.
\begin{definition}\label{definition:doubletripledimerconfiguration}
An \textbf{$m$-fold dimer configuration} $D$ of a plabic graph $G$ is a multiset of the edges of $G$ such that each vertex is contained in exactly $m$ edges in $D$. In other words, it is a superimposition of $m$ single dimer configurations of $G$. When $m=2$, we call these \textbf{double dimer configurations}, and when $m=3$ we call them \textbf{triple dimer configurations}. We refer to the set of $m$-fold dimer configurations of $G$ as $\mathcal{D}^m(G)$.
\end{definition}

\begin{example}\label{example:tripledimerB}
In Figure \ref{fig:tripledimerB}, the \textcolor{red}{red} edges, \textcolor{blue}{blue} edges, and \textcolor{forest}{green} edges are individually single dimer configurations of the the plabic graph with boundary conditions $\{4,5,6\}$, $\{2,3,4\}$ and $\{1,7,8\}$ respectively. Forgetting the distinctions between these colors yields a corresponding triple dimer configuration.

\begin{figure}[H]
    \centering
    \includegraphics[scale=.15]{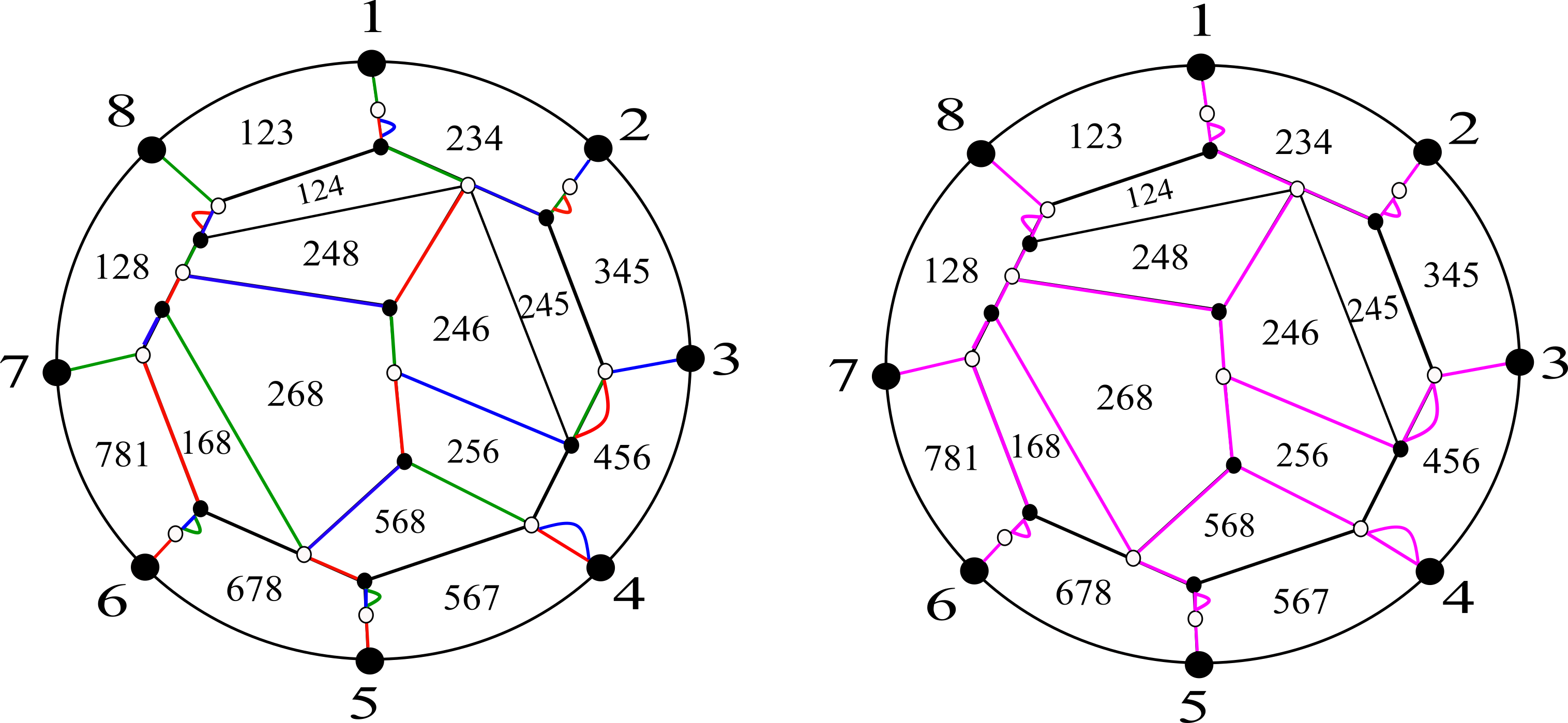}
    \caption{Three overlaid single dimer configurations (left) and the associated triple dimer (right).}
    \label{fig:tripledimerB}
\end{figure}
\end{example}

\subsection{The Twist Map}\label{subsection:twist}
In this section, we define an important cluster algebra automorphism called the \textit{twist map}. This map was first introduced in \cite{bfz96}; Marsh and Scott linked it to dimer partition functions in \cite{marshscott2016}, and Muller and Speyer showed in \cite{mullerspeyer} that it provides an inverse to the famous boundary measurement map introduced by Postnikov \cite{p06}. Each paper uses a slightly different set of conventions, and we will use yet another, but we will clarify the relationships between our twist and those in \cite{marshscott2016} and \cite{mullerspeyer}.

We first give exposition following \cite[Sec. 2]{marshscott2016}.
Given a matrix $M$ representing an element of $\Gr(k,n)$, with column vectors $v_1, v_2, \dots v_n \in \mathbb{R}^k$ (in order), we define the {\bf generalized cross-product} $v = v_1 \times v_2 \times \dots \times v_{k-1}$ to be the unique vector in $\mathbb{R}^k$ satisfying $v \cdot w = \det(v_1 \quad v_2 \dots v_{k-1} \quad w)$ for all $w \in \mathbb{R}^k$.
Then, the
{\bf (left) twist} of $M$, denoted as $\lT(M)$, is defined to be the $k$-by-$n$ matrix whose $i$th column vector is given by 
$\lT(M)_i =\varepsilon_i \cdot v_{i-k+1} \times v_{i-k+2} \times \dots \times v_{i-1}$, where
\[\varepsilon_i = \begin{cases}
(-1)^{i(k-i)} & i \leq (k-1)\\
1 & i \geq k
\end{cases}\] and
the subscripts are taken modulo $n$ with signs introduced when wrapping around.
Explicitly, 
$$\lT(M)_i = \begin{cases} 
(-1)^{k-i} v_1 \times v_2 \times \dots \times v_{i-1} \times v_{i-k+1+n} \times \dots \times v_n & \mathrm{~if~} i\leq k-1 \\
v_{i-k+1} \times v_{i-k+2} \times \dots \times v_{i-1} & \mathrm{~if~} i\geq k \end{cases}.$$
In the special case of this paper where $k=3$, by construction we obtain
$$\lT(M) = [v_{n-1} \times v_n \quad v_n \times v_1 \quad v_1 \times v_2 \dots v_{n-2} \times v_{n-1}]
= [v_{n-1} \times v_n \quad -v_1 \times v_n \quad v_1 \times v_2 \quad \dots \quad v_{n-2} \times v_{n-1}]$$
where only the usual cross-product is required.
To compare the value of Pl{\"u}cker coordinates before and after the twist, let $\Delta_J$ denote the determinant of the submatrix $M_J$ (given by the columns indexed by set $J$), and let $\lT(\Delta_J)$ denote the determinant of the corresponding submatrix of the twisted matrix, i.e. $\det \left(\lT(M)_{J}\right)$. When $J = \{a,b,c\}$, we have
$$\lT(\Delta_{abc}) = \det \bigg( v_{a-2} \times v_{a-1} \quad v_{b-2} \times v_{b-1} \quad v_{c-2} \times v_{c-1}\bigg),$$
where indices are taken modulo $n$.

In \cite{mullerspeyer}, Muller and Speyer define a different version of the left twist, which we denote $\lT_{MuSp}$. They also define a \textbf{right twist} analogously, which is its inverse. We define a right twist $\T$ analogously to the left twist of \cite{marshscott2016} above, via $\T(M)_i= \varepsilon'_i \cdot v_{i+1} \times \dots \times v_{i+k-1}$, where
\[\varepsilon'_i = \begin{cases}
(-1)^{(k-1)(n-i+1)} & i \geq (n-k+2)\\
1 & i \leq (n-k+1)
\end{cases}\]
again with reduction modulo $n$ and appropriate signs. Explicitly,
$$\T(M)_i = \begin{cases} 
(-1)^{k-n+i-1} v_1 \times v_2 \times \dots \times v_{i-n+k-1} \times v_{i+1} \times \dots \times v_n & \mathrm{~if~} i\geq n-k+2 \\
v_{i+1} \times v_{i+2} \times \dots \times v_{i+k-1} & \mathrm{~if~} i\leq n-k+1 \end{cases}.$$
When $k=3$ and $J=\{a,b,c\}$, we have
$$\T(\Delta_{abc}) = \det \bigg(v_{a+1} \times v_{a+2} \quad v_{b+1} \times v_{b+2} \quad 
v_{c+1} \times v_{c+2}\bigg),$$ 
where indices are taken modulo $n$.
This $\T$ is the twist we will use for the rest of the paper.

We now recover a version of \cite[Proposition 3.5]{marshscott2016} for our right twist in the special case of $k=3$, using the cross-product identities of \cite{s06} mentioned in Section \ref{subsection:differences}. If $J =\{ a, a+1, a+2\}$, i.e. $\Delta_J$ is a frozen variable, then 
\begin{eqnarray*}\T(\Delta_J) &= \det \bigg( v_{a+1} \times v_{a+2} \quad v_{a+2} \times v_{a+3} \quad v_{a+3} \times v_{a+4}\bigg) \\ &= \det(v_{a+1} \hspace{3pt} v_{a+2} \hspace{3pt} v_{a+3}) \det (v_{a+2}\hspace{3pt}  v_{a+3}\hspace{3pt} v_{a+4}) - \det (v_{a+1}\hspace{3pt} v_{a+3}\hspace{3pt} v_{a+4}) \det(v_{a+2}\hspace{3pt} v_{a+2}\hspace{3pt} v_{a+3}) \\ &= \Delta_{a+1,a+2,a+3} \Delta_{a+2,a+3,a+4},\end{eqnarray*} 
since $\det(v_{a-1}\hspace{3pt} v_{a-1}\hspace{3pt} v_{a})=\det(v_{a+2}\hspace{3pt} v_{a+2}\hspace{3pt} v_{a+3}) = 0$.
Similarly, if $J = \{a,a+1, b\}$ where $b \neq a-1, a+2$, then 
\begin{eqnarray*}\T(\Delta_J) &= \det \bigg( v_{a+1} \times v_{a+2} \quad v_{a+2} \times v_{a+3} \quad v_{b+1} \times v_{b+2}\bigg) \\ &= \det(v_{a+1}\hspace{3pt} v_{a+2}\hspace{3pt} v_{a+3}) \det(v_{a+2}\hspace{3pt} v_{b+1}\hspace{3pt} v_{b+2}) - \det(v_{a+1}\hspace{3pt} v_{b+1}\hspace{3pt} v_{b+2}) \det(v_{a+2}\hspace{3pt} v_{a+2}\hspace{3pt} v_{a+3}) \\ &= \Delta_{a+1,a+2,a+3} \Delta_{a+2,b+1,b+2}.\end{eqnarray*}
Note that since the sign of a 3-cycle is +1, we may reorder the indices of the resulting Pl\"ucker coordinates to be increasing modulo $n$ without concern for signs.

When $J = \{a,b,c\}$ where none of $a,b,c$ are adjacent mod $n$, none of the Pl{\"u}cker coordinates appearing in the quadratic differences vanish, and hence we recover the expressions 
\begin{eqnarray*}\T(\Delta_J) &= \det \bigg( v_{a+1} \times v_{a+2} \quad v_{b+1} \times v_{b+2} \quad v_{c+1} \times v_{c+2}\bigg) \\
&=\begin{cases}
X^{a+1,~a+2,~b+1,~b+2,~c+1,~c+2} & a,b,c \neq n-1\\
Y^{a+1,~a+2,~b+1,~b+2,~c+1,~c+2} & \text{otherwise}
\end{cases}.\end{eqnarray*}
Recall that order is disregarded for the superscripts, so we may always write them in increasing order modulo $n$.

We conclude by noting that, as stated in \cite[Remark 6.3]{mullerspeyer}, our right twist $\T$ agrees with the right twist $\T_{MuSp}$ up to rescaling. In the case of $\T(\Delta_J)$ where $J = \{a_1,a_2,\dots, a_k\}$, we get
\begin{equation}\label{eq:musprmk}
    \T(\Delta_J) = \Delta_{I_{[a_1]}}\Delta_{I_{[a_2]}}\cdots \Delta_{I_{[a_k]}} \cdot \T_{MuSp}(\Delta_J),
\end{equation}
where the notation $\Delta_{I_{[a_j]}}$ is shorthand for the Pl\"ucker coordinate for the cyclically connected subset $\{a_j, a_j+1, a_j+2, \dots, a_j + k-1\}$ where indices are taken modulo $n$.

\section{Dimer Face Weights}\label{section:facewts}
In this section, we describe a method of weighting dimers using the face labels that arise from strands of a plabic graph, rather than the arbitrary real edge weights described in Section \ref{subsection:dimers}. Work of Marsh and Scott \cite{marshscott2016} and later Muller and Speyer \cite{mullerspeyer} connects a given boundary set to a sum of face weights of its corresponding dimers via the twist map described in Section \ref{subsection:twist}, though Marsh and Scott use different conventions than ours. We show that our face weights coincide with face weights in \cite{mullerspeyer} up to scaling, and conclude that they describe our version of the right twist. We also describe a translation between edge and face weights similar to that in \cite{marshscott2016}.

We begin by establishing the notation that we will use to define our version of face weights. 
Given a plabic graph $G$ and a face $f \in F(G)$, labeled using strands as described in Subsection \ref{subsection: clusteralgbackground}, we define the following quantities:
$$I_f := \text{ the face label of } f \hspace{2em} \text{ and } \hspace{2em} W_f := \#\{\text{white vertices bordering } f\}.$$
We say $f$ is an \textbf{inner face} if it is not adjacent to the boundary of the circle, and \textbf{outer} otherwise. Given a dimer $D$, we define the following quantities based on whether a given face is inner or outer. 
$$D_f := \begin{cases}
    \#\{\text{edges of } D \text{ that border } f\} &\text{ if } f \text{ is inner}\\
    \#\{\text{edges of } D \textit{ not adjacent to boundary vertices}\text{ that border } f\} &\text{ if } f \text{ is outer}
\end{cases} .$$

\begin{example}
Figure \ref{fig:346dimers} depicts the four single dimer configurations on a certain plabic graph for $\Gr(3,7)$ that have the boundary condition $\{3,4,6\}$. Let \textcolor{red}{$D_1$}, \textcolor{orange}{$D_2$}, \textcolor{blue}{$D_3$}, and \textcolor{purple}{$D_4$} be the single dimer configurations shown left to right in colors \textcolor{red}{red}, \textcolor{orange}{orange}, \textcolor{blue}{blue} and \textcolor{purple}{purple} respectively.
\begin{itemize}
    \item For the inner face $\Delta_{367}$, $W_f=2$; we have $(\textcolor{red}{D_1})_f=(\textcolor{orange}{D_2})_f=2$ and $(\textcolor{blue}{D_3})_f=(\textcolor{purple}{D_4})_f=1$.
    \item For the outer face $\Delta_{127}$, $W_f=3$; we have $(\textcolor{red}{D_1})_f=(\textcolor{orange}{D_2})_f=(\textcolor{blue}{D_3})_f=2$ and $(\textcolor{purple}{D_4})_f=1$.
\end{itemize}
\end{example}

We are now ready to define dimer face weights.
\begin{definition}\label{definition:faceweights}
Given an $m$-fold dimer configuration $D$ on a plabic graph $G$, we define the \textbf{face weight of $D$} to be
\[\text{wt}_f(D)=\prod_{f \in F(G)} I_f^{mW_f-D_f-m}.\] 
\end{definition}

\begin{example}\label{ex:346wts}
Again, consider the dimers \textcolor{red}{$D_1$}, \textcolor{orange}{$D_2$}, \textcolor{blue}{$D_3$}, and \textcolor{purple}{$D_4$} on the plabic graph for $\Gr(3,7)$ shown in Figure \ref{fig:346dimers}. There are four possible single dimer configurations with respect to the boundary condition $\{3,4,6\}$. The weights of each of these single dimer configurations are as follows:
$$\text{wt}_f(\textcolor{red}{D_1}) = \Delta_{456}\frac{\Delta_{167}\Delta_{237}\Delta_{567}}{\Delta_{267}\Delta_{367}}~~~;~~~\text{wt}_f(\textcolor{orange}{D_2})  = \Delta_{456}\frac{\Delta_{167}\Delta_{347}\Delta_{567}}{\Delta_{367}\Delta_{467}}$$
$$\text{wt}_f(\textcolor{blue}{D_3})  = \Delta_{456}\frac{\Delta_{167}\Delta_{457}}{\Delta_{467}}~~~;~~~\text{wt}_f(\textcolor{purple}{D_4})  = \Delta_{456}\frac{\Delta_{127}\Delta_{567}}{\Delta_{267}} .$$

\begin{figure}
    \centering
    \includegraphics[scale=.12]{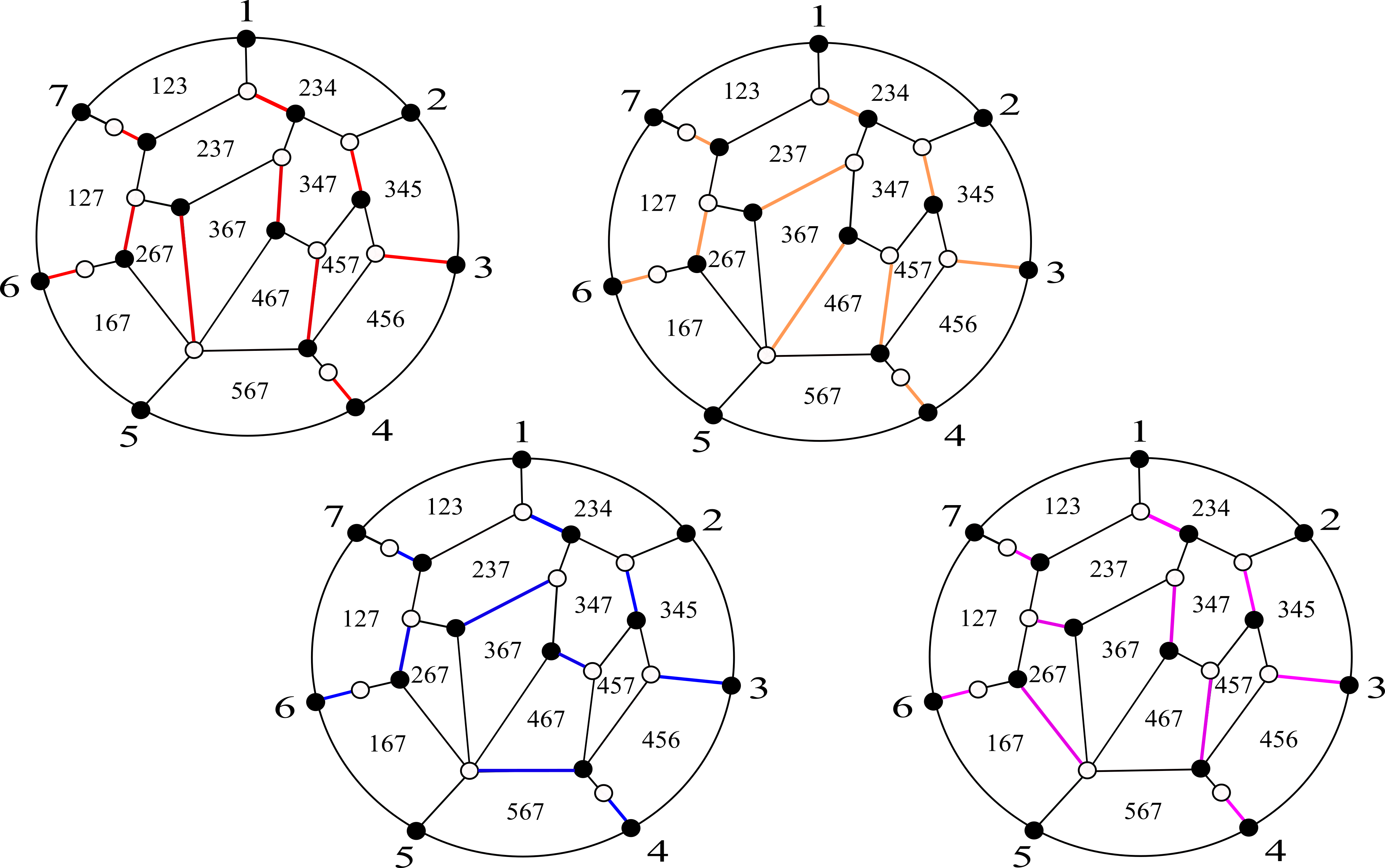}
    \caption{All single dimer configurations with boundary condition $\{3,4,6\}$ on a certain plabic graph for $\Gr(3,7)$.}
    \label{fig:346dimers}
\end{figure}
\end{example}

The following theorem is central to our main results; we will extend it to $m$-fold dimers in future sections.
\begin{theorem}\label{theorem:facewts}
Let $G$ be a plabic graph with black boundary vertices $1,2, \dots, n$, let $k$ be the number of internal white vertices minus the number of internal black vertices in $G$, and let $J$ be a $k$-element subset of $[n]$. Then where $\text{wt}_f(D)=\prod_{f \in F(G)} I_f^{nW_f-D_f-n}$,
\[\T(\Delta_J) = \sum_{D \in \mathcal{D}_J(G)} \text{wt}_f(D).\]
\end{theorem}
\begin{proof}
We will prove this theorem by relating it to Remark 7.11 of \cite{mullerspeyer}, which provides the following formula for their variant $\T_{MuSp}$ of the right twist of a Pl\"ucker coordinate:
\[\T_{MuSp}(\Delta_J) = \sum_{D \in \mathcal{D}_J(G)} \widetilde{\text{wt}_f(D)},\]
where $\widetilde{\text{wt}_f(D)}$ is defined via \[\prod_{f \in F(G)} I_f^{(\tilde{B}_f-1) - \#\{e \in D : \tilde{\partial}_{fe} = 1\}},\] and $\tilde{B}_f$ and $\tilde{\partial}_{fe}$ are defined in terms of the number of edges $e$ such that face $f$ lies \textbf{directly upstream} of $e$.\footnote{As in \cite[Remark 5.8]{mullerspeyer}, such a weighting appeared previously in \cite{octahedron} but via different exposition.} 
We have from Equation (\ref{eq:musprmk}) that for $J = \{a_1,a_2,\dots, a_k\}$,
\[\T_{MuSp}(\Delta_J) =  \frac{1}{\Delta_{I_{[a_1]}}\Delta_{I_{[a_2]}}\cdots \Delta_{I_{[a_k]}}} \T(\Delta_J),\]
where the notation $\Delta_{I_{[a_j]}}$ is shorthand for the Pl\"ucker coordinate for the cyclically connected subset $\{a_j, a_j+1, a_j+2, \dots, a_j + k-1\}$ where indices are taken modulo $n$. It will therefore suffice to show that
\begin{equation}\label{eq:muspfacewts}
\sum_{D \in \mathcal{D}_J(G)} \widetilde{\text{wt}_f(D)} = \frac{1}{\Delta_{I_{[a_1]}}\Delta_{I_{[a_2]}}\cdots \Delta_{I_{[a_k]}}} \sum_{D \in \mathcal{D}_J(G)} \text{wt}_f(D).
\end{equation}

Note that since our plabic graphs are bipartite, all inner faces are bordered by an even number of edges, and similarly for all outer faces since all boundary vertices are black (and we do not count the ``edges" between boundary vertices). Thus $W_f$, the number of white vertices adjacent to a face $f$, is exactly half the number of edges bordering $f$; we now have from \cite[Section 5.1]{mullerspeyer} that $W_f=\tilde{B}_f$.

If a given face $f$ is inner, we immediately have that $\#\{e \in D : \tilde{\partial}_{fe} = 1\} = D_f$. If $f$ is outer, the equality holds unless $f$ lies immediately counter-clockwise to a boundary edge $e$ that is included in the dimer cover $D$, in which case we have $\#\{e \in D : \tilde{\partial}_{fe} = 1\} = D_f+1$. The boundary edges included in $D$ are exactly those adjacent to boundary vertices $j \in J$. Therefore, by the construction of face labels, when $J = \{a_1,a_2,\dots, a_k\}$, the three outer faces where $\#\{e \in D : \tilde{\partial}_{fe} = 1\} = D_f + 1$ are precisely those labeled $I_{[a_1]}$, $I_{[a_2]},$\dots, $I_{[a_k]}$. Equation (\ref{eq:muspfacewts}) now follows from a comparison of definitions.
\end{proof}

\begin{remark}
\label{rem:twistdegrees}
For $k=3$, the twist map $\T$ doubles degree, while the twist map $\T_{MuSp}$ is of degree $-1$. This is consistent with the claims made in our proof, which assert in this case that images of $\T_{MuSp}(\Delta_J)$ and $\T(\Delta_J)$ agree up to a quotient by three frozen variables.
\end{remark}

The following definition and proposition provide a translation between the face weights of Definition \ref{definition:faceweights} and the edge weights used in the statements of Subsection \ref{subsection:dimers}, creating a streamlined comparison between our results and those of other authors that are phrased in terms of edge weights. In particular, in Section \ref{section:duality}, we will use this propostion to describe our main theorems using the language of web duality introduced in \cite{fll2019}.
\begin{definition}\label{def:MaScEdgeWeights}
Given a plabic graph $G$ with face labels $I_f$, we define the weight of an edge $e$, denoted $wt(e)$, via
$$\text{wt}(e) = \frac{\Delta_{I_1}\cdot \Delta_{I_2} \cdots \Delta_{I_d}}{\Delta_{{I_{e(1)}}}\cdot \Delta_{I_{e(2)}}},$$
where $I_1$ through $I_d$ label the $d$ faces bordering the black endpoint of edge $e$, and $I_{e(1)}$ and $I_{e(2)}$ label the two faces that border edge $e$ itself. See Figure \ref{fig:edgetoface}.
\end{definition}
\begin{remark}
This definition is analogous to \cite[Definition 7.1]{marshscott2016}, except that they consider white vertices where we consider black vertices. The significance of this switch is that by our definition, all edges adjacent to boundary vertices will have weight 1, since all boundary vertices of our plabic graphs are black.
\end{remark}

\begin{proposition}\label{prop:edgetoface}
Given an $m$-fold dimer configuration $D$ on a plabic graph $G$, consider the edge-weighting $\text{wt}(e)$ of Definition \ref{def:MaScEdgeWeights}, and let $\text{wt}_e(D) = \prod_{e \in D} \text{wt}(e)$ as in Section \ref{subsection:dimers}.
Also recall the face-weighting $\text{wt}_f(D)$ of Definition \ref{definition:faceweights}.
Then
\[\text{wt}_e(D)\bigg/\left(\prod_{\text{inner }f \in F(G)} \Delta_{I_f}\right)^m=\text{wt}_f(D).\]

\begin{figure}
    \centering
    \includegraphics[scale=.25]{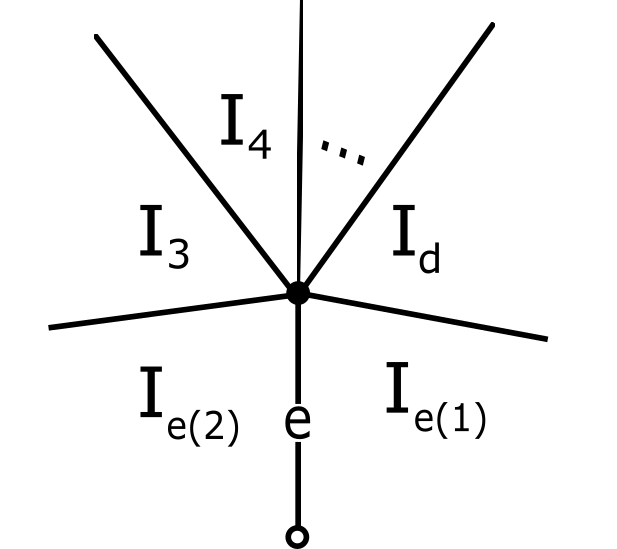}
    \caption{A weight function on the plabic graph $G$; here $\text{wt}(e)=\Delta_{I_3}\cdot\Delta_{I_4} \cdots \Delta_{I_d}$.}
    \label{fig:edgetoface}
\end{figure}
\end{proposition}
\begin{proof}
First, note that any $m$-fold dimer $D$ is an overlay of $m$ single dimers $D_1,\dots,D_m$, and from the definitions we have
\[\frac{\text{wt}_e(D)}{\left(\prod_{\substack{\text{inner }\\f \in F(G)}} \Delta_{I_f}\right)^m}=\prod_{i=1}^m \frac{\text{wt}_e(D_i)}{\prod_{\substack{\text{inner }\\f \in F(G)}} \Delta_{I_f}}\]
and
\[\text{wt}_f(D)=\prod_{i=1}^m \text{wt}_f(D_i).\]
It therefore suffices to prove the case where $m=1$.

For any edge $e \in E(G)$, let $e(1)$ and $e(2)$ be the faces adjacent to $e$, as in Figure \ref{fig:edgetoface}. Also, let $V_D(G)$ be the set of vertices of $G$ that are adjacent to some edge of $D$; note that $V_D(G)$ is exactly the set of interior vertices of $G$ together with $\partial D$. Then by definition, we have
\[\frac{\text{wt}_e(D)}{\prod_{\substack{\text{inner }\\f \in F(G)}} \Delta_{I_f}}=\frac{\prod_{\substack{\text{black}\\v \in V_D(G)}}\prod_{\substack{\text{$f \in F(G)$}\\\text{incident to }v}} \Delta_{I_f}}{\prod_{e \in D}I_{e(1)}I_{e(2)}}\cdot \frac{1}{\prod_{\substack{\text{inner }\\f \in F(G)}} \Delta_{I_f}}\]
\[=\frac{\prod_{\substack{\text{inner }\\f \in F(G)}} \Delta_{I_f}^{\#\left(\substack{\text{black }v \in V_D(G)\\\text{incident to }f}\right)-1} \cdot\prod_{\substack{\text{outer }\\f \in F(G)}} \Delta_{I_f}^{\#\left(\substack{\text{black }v \in V_D(G)\\\text{incident to }f}\right)}}
{\prod_{e \in D} I_{e(1)}I_{e(2)}}.\]
For a given $f \in F(G)$, we extract the power of $\Delta_{I_f}$ in the above quotient. If $f$ is inner, every vertex adjacent to $f$ is included in $D$, and there are the same number of white vertices adjacent to $f$ as black vertices. We therefore get
\[\Delta_{I_f}^{\#\left(\substack{\text{black }v \in V_D(G)\\\text{incident to }f}\right)-\#\left(\substack{\text{edges }e \in D\\\text{incident to }f}\right)-1}=\Delta_{I_f}^{\#\left(\substack{\text{white }v \in V(G)\\\text{incident to }f}\right)-\#\left(\substack{\text{edges }e \in D\\\text{incident to }f}\right)-1}=\Delta_{I_f}^{W_f-D_f-1}.\]
If $f$ is outer, 
we get
\[\Delta_{I_f}^{\#\left(\substack{\text{black }v \in V_D(G)\\\text{incident to }f}\right)-\#\left(\substack{\text{edges }e \in D\\\text{incident to }f}\right)}
\]\[=\Delta_{I_f}^{\left(\#\left(\substack{\text{black interior }\\v \in V(G)\\\text{incident to }f}\right)+\#\left(\substack{\text{black }v \in \partial D\\\text{incident to }f}\right)\right)-\left(\#\left(\substack{\text{edges }e \in D\\\text{incident to }f \\\text{~and~to~}\partial G}\right)+\#\left(\substack{\text{edges }e \in D\\\text{incident to }f\\\text{but not }\partial G}\right)\right)}\]
\[=\Delta_{I_f}^{\#\left(\substack{\text{black interior }\\v \in V(G)\\\text{incident to }f}\right)-\#\left(\substack{\text{edges }e \in D\\\text{incident to }f\\\text{but not }\partial G}\right)}\]
\[=\Delta_{I_f}^{\left(\#\left(\substack{\text{white interior}\\v \in V(G)\\\text{incident to }f}\right)-1\right)-\#\left(\substack{\text{edges }e \in D\\\text{incident to }f\\\text{but not }\partial G}\right)}=\Delta_{I_f}^{W_f-D_f-1},\]
where the second-to-last equality follows from the fact that all boundary vertices of $G$ are black.
The product of these powers of face weights is $\text{wt}_f(D)$ by definition.
\end{proof}

\begin{example}
Since $\{3,4,6\}=\{a,a+1,b\}$, we have from Subsection \ref{subsection:twist} that $\T(\Delta_{346})= \Delta_{456}\Delta_{157}.$ Taking the sum of the weights of the single dimer configurations with boundary condition $\{3,4,6\}$, as computed in Example \ref{ex:346wts}, yields
\[
    \Delta_{456} \cdot \dfrac{\Delta_{167}\Delta_{237}\Delta_{467}\Delta_{567}+\Delta_{167}\Delta_{267}\Delta_{347}\Delta_{567} + \Delta_{167}\Delta_{267}\Delta_{367}\Delta_{457}+\Delta_{127}\Delta_{367}\Delta_{467}\Delta_{567}}{\Delta_{267}\Delta_{367}\Delta_{467}}.\]
This is indeed the Laurent expansion for $\Delta_{456}\Delta_{157}$, which is consistent with Theorem \ref{theorem:facewts}.

We also note using Equation (\ref{eq:musprmk}) that $\T_{MuSp}(\Delta_{346}) = \frac{\Delta_{456}\Delta_{157}}{\Delta_{345}\Delta_{456}\Delta_{167}}
= \frac{\Delta_{157}}{\Delta_{345}\Delta_{167}}$, which is consistent with this map being degree $-1$ as in Remark \ref{rem:twistdegrees}.
\end{example}

\section{Double Dimer Configurations for Quadratic Differences}\label{section:doubledimers}
In this section, we give a combinatorial interpretation via double dimer face weights for the twists of the quadratic cluster variables $X$ and $Y$ in $\Gr(3,6)$, and of analogous expressions $X^S$ and $Y^S$ for any $S \subset [n]$ with $|S|=6 \leq n$.

\begin{theorem}\label{theorem:doubledimersXY}
Given a plabic graph $G$ for $\Gr(3,n)$ and a set $S=\{s_1<\dots<s_6\} \subseteq [n]$, let $\mathcal{D}^2_X(G)$ be the set of double dimers on $G$ with paths connecting vertices in pairs $\{s_1,s_6\},\{s_2,s_3\},$ and $\{s_4,s_5\}$, and let $\mathcal{D}^2_Y(G)$ be the set of double dimers with paths connecting vertices in pairs $\{s_1,s_2\},\{s_3,s_4\},$ and $\{s_5,s_6\}$, where the double dimers in each set possibly include internal doubled edges and cycles, but no additional edges adjacent to boundary vertices. Then
\begin{itemize}
    \item $\mathcal{T}^*(X^S)=\sum_{D \in \mathcal{D}^2_X(G)} 2^{\#( \mathrm{cycles~in~}D)} wt_f(D)$
    \item $\mathcal{T}^*(Y^S)=\sum_{D \in \mathcal{D}^2_Y(G)} 2^{\#( \mathrm{cycles~in~}D)} wt_f(D).$
\end{itemize}
\end{theorem}
\begin{proof}
It follows from Theorem \ref{theorem:facewts} that for any plabic graph $G$ for $\Gr(k,n)$, and for any $I,J \subset [n]$ with $|I|=|J|=k$,
\[\T(\Delta_I\Delta_J) =\T(\Delta_I)\T(\Delta_J)= \left(\sum_{D \in \mathcal{D}_I(G)} \text{wt}_f(D)\right)\left(\sum_{D \in \mathcal{D}_J(G)} \text{wt}_f(D)\right)\]
\[=\sum_{D \in \mathcal{D}_{I,J}^2(G)}M_D \text{wt}_f(D),\]
where $\mathcal{D}_{I,J}^2(G) \subset \mathcal{D}^2(G)$ is the set of double dimer covers $D$ of $G$ formed by overlaying a single dimer with boundary condition $I$ and a single dimer with boundary condition $J$, and the multiplicity $M_D$ is the number of pairs of single dimers that become $D$ when overlaid.

To characterize $\mathcal{D}_{I,J}^2(G)$, we note that the edges contained in any double dimer may be viewed as a union of connected components, each of which is a path between boundary vertices, a doubled edge, or an internal cycle; see \cite{kuo04}. It follows from \cite[Theorem 3.1]{lam2015} that for any $I,J \in \binom{[n]}{k}$, $\mathcal{D}_{I,J}^2(G)$ contains exactly those double dimers consisting of $k$ non-crossing paths, each connecting a vertex with label in $I$ to a vertex with label in $J$, as well as possibly some internal doubled edges and cycles; and that for any double dimer $D$,
\[M_D=2^{\# (\mathrm{cycles~in~}D)}.\]

Now in the case of $X$, where $k=3$, we have
\[\mathscr{T}^*(X^S)=\mathscr{T}^*(\Delta_{134})\mathscr{T}^*(\Delta_{256})-\mathscr{T}^*(\Delta_{156})\mathscr{T}^*(\Delta_{234})\]
\[=\sum_{D \in \mathcal{D}_{134,256}(G)} 2^{\#( \mathrm{cycles~in~}D)}wt_f(D) - \sum_{D \in \mathcal{D}_{156,234}(G)} 2^{\# (\mathrm{cycles~in~}D)}wt_f(D).\]
The first sum contains exactly those double dimers with paths connecting $1$ to $2$, $3$ to $6$, and $4$ to $5$; or $1$ to $6$, $2$ to $3$, and $4$ to $5$. These are the only possible non-crossing matchings. The second, negative sum contains exactly those double dimers with paths connecting $1$ to $2$, $3$ to $6$, and $4$ to $5$. Therefore the terms that remain after cancellation are the weights of all double dimers that connect $1$ to $6$, $2$ to $3$, and $4$ to $5$, which are exactly those included in $\mathcal{D}_{X}^2(G)$.

Similarly, in the case of $Y$, we have
\[\mathscr{T}^*(Y)=\mathscr{T}^*(\Delta_{145})\mathscr{T}^*(\Delta_{236})-\mathscr{T}^*(\Delta_{123})\mathscr{T}^*(\Delta_{456})\]
\[=\sum_{D \in \mathcal{D}_{145,236}(G)} 2^{\#( \mathrm{cycles~in~}D)}wt_f(D) - \sum_{D \in \mathcal{D}_{123,456}(G)} 2^{\# (\mathrm{cycles~in~}D)}wt_f(D).\]
The first sum contains exactly those double dimers with paths connecting $1$ to $6$, $2$ to $5$, and $3$ to $4$; or $1$ to $2$, $3$ to $4$, and $5$ to $6$. The second, negative sum contains exactly those double dimers with paths connecting $1$ to $6$, $2$ to $5$, and $3$ to $4$. Therefore the terms that remain after cancellation are the weights of all double dimers that connect $1$ to $2$, $3$ to $4$, and $5$ to $6$, which are exactly those included in $\mathcal{D}_{Y}^2(G)$.

This completes the proof for $S=\{1,2,3,4,5,6\}$; the argument is identical for general $S$.
\end{proof}

\begin{example}
\begin{figure}[H]
    \centering
    \includegraphics[width=\textwidth]{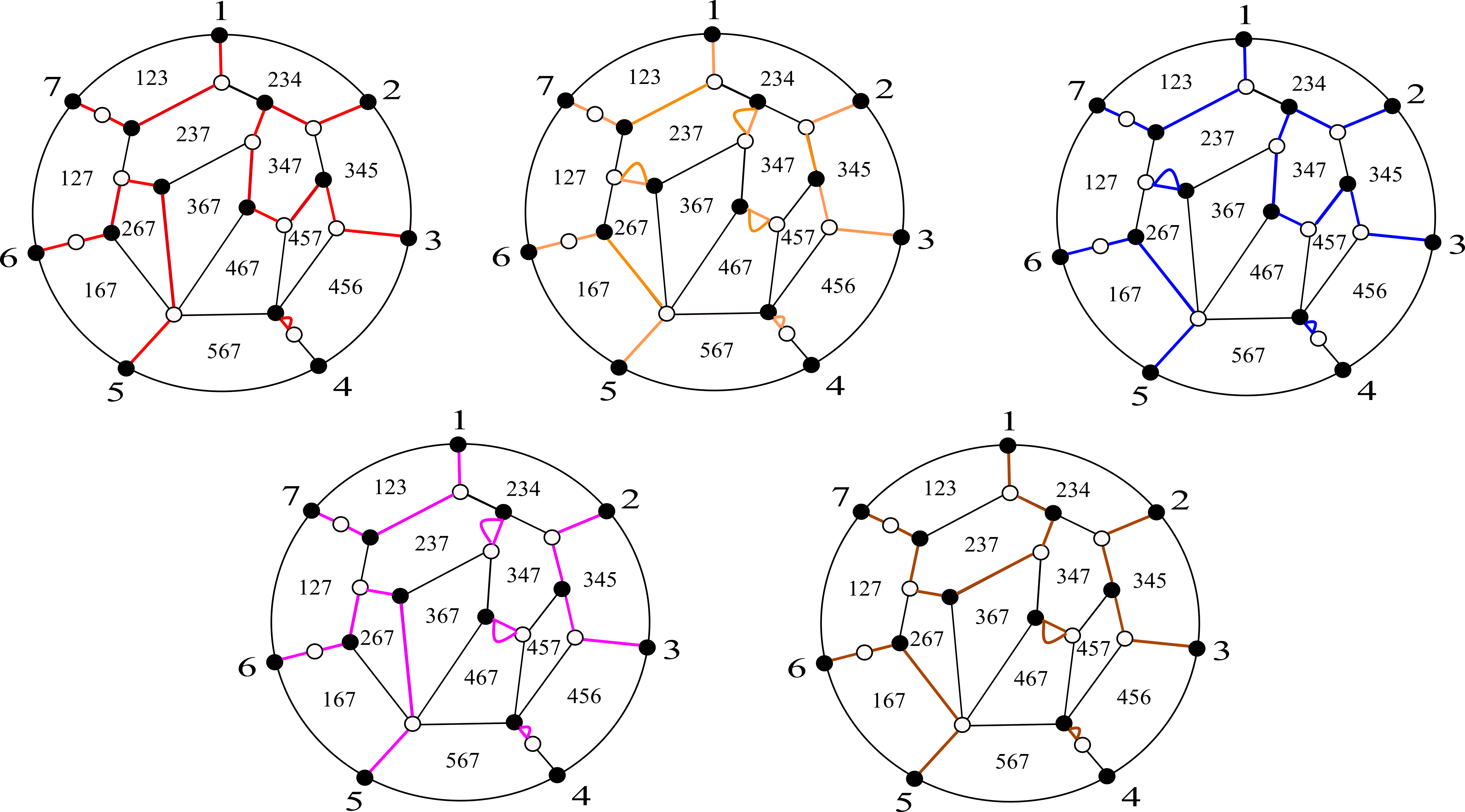}
    \caption{Double dimer configurations satisfying the connectivity pattern for $X^{123567}$.}
    \label{fig:X123567}
\end{figure}
Figure \ref{fig:X123567} depicts all double dimer configurations in $\mathcal{D}_{X^{123567}}^2(G)$, which is defined in Theorem \ref{theorem:doubledimersXY} to be the set of double dimers with paths connecting vertex 1 to 7, 2 to 3, and 5 to 6. One may streamline the construction of these dimers by first finding ``forced edges" that necessarily must be included a certain number of times in any dimer with the desired connectivity. For example, there must be a single dimer edge adjacent to every boundary vertex except vertex 4, which cannot be adjacent to any dimer edges. 

Computing the twist of $X^{123567}$ using the methods described in Section \ref{subsection:twist} yields $\T(X^{123567}) = (127)(234)(X^{134567})$. Theorem \ref{theorem:doubledimersXY} asserts that the Laurent expansion for $\T(X^{123567})$ should be the sum of the face weights of the double dimers in Figure \ref{fig:X123567}, and indeed we have
\begin{align*}
    (127)(234)(X^{134567}) & =\textcolor{red}{(127)(234)\frac{(167)(237)(345)(467)}{(267)(347)}} + \textcolor{orange}{(127)(234)\frac{(127)(234)(367)^2(457)}{(237)(267)(347)}}\\
    &+\textcolor{blue}{(127)(234)\frac{(127)(345)(367)(467)}{(267)(347)}} + \textcolor{purple}{(127)(234)\frac{(167)(234)(367)(457)}{(267)(347)}}\\
    &+ \textcolor{brown}{(127)(234)\frac{(123)(367)(457)}{(237)}}\\&= wt_f(\textcolor{red}{D_1}) + wt_f(\textcolor{orange}{D_2}) + wt_f(\textcolor{blue}{D_3}) + wt_f(\textcolor{purple}{D_4}) + wt_f(\textcolor{brown}{D_5})
\end{align*}
where $D_1, D_2, D_3, D_4, D_5$ are the red, orange, blue, purple and brown double dimer configurations in Figure \ref{fig:X123567}, respectively.
\end{example}

We end this section by answering a question posed to us by David Speyer regarding cluster variables associated to (untwisted) Pl\"ucker coordinates.

\begin{proposition}
\label{prop:untwisted}
For any $n\geq 4$, in the coordinate ring $\mathbb{C}[\widehat{\Gr}(3,n)]$, the Laurent expansion for the cluster variable given by a Pl\"ucker coordinate $\Delta_J$ may be expressed combinatorially as a partition function as given by either Theorem \ref{theorem:facewts} or Theorem \ref{theorem:doubledimersXY}, up to multiplication by frozen variables.
\end{proposition}
\begin{proof}
It is sufficient to demonstrate in $\mathbb{C}[\widehat{\Gr}(3,n)]$  that up to multiplication by frozen variables, any Pl\"ucker coordinate may be written as the image of the twist map applied either to a Pl\"ucker coordinate $\Delta_I$ or to a quadratic difference $X^S$. In what follows, we consider indices mod $n$.

If $J=\{a,a+1,a+2\}$, the fact that $\T(\Delta_{a-1,a,a+1})=\Delta_{a+1,a+2,a+3}\Delta_J$ follows directly from a computation at the end of Section \ref{subsection:twist}.
The subsequent computation implies that if $J=\{a,a+1,b\}$ where $b \neq a-1,a+2$, we have $\T(\Delta_{a-1,b-2,b-1})=\Delta_{b-1,b,b+1}\Delta_J$.

Finally, if $J=\{a,b,c\}$ where none of $a,b,c$ are adjacent, we consider the right twist map applied to the quadratic difference 
$X^{a-2,a-1,b-2,b-1,c-2,c-1} = \Delta_{a-2,b-2,b-1} \Delta_{a-1,c-2,c-1} - \Delta_{a-2,c-2,c-1} \Delta_{a-1,b-2,b-1}$.  Since none of $a,b,c$ are adjacent to one another, all six of these superscripts are distinct numbers.  Using the identities of Section \ref{subsection:twist}, we compute
\begin{eqnarray*} 
\T(X^{a-2,a-1,b-2,b-1,c-2,c-1}) &=& [\Delta_{b-1,b,b+1} \Delta_{a-1,a,b}] [\Delta_{c-1,c,c+1} \Delta_{a,a+1,c}] \\
&~~& - ~~ [\Delta_{c-1,c,c+1} \Delta_{a-1,a,c}] [\Delta_{b-1,b,b+1} \Delta_{a,a+1,b}] \\
&=& \Delta_{b-1,b,b+1}\Delta_{c-1,c,c+1} \bigg( \Delta_{a-1,a,b}\Delta_{a,a+1,c} - \Delta_{a-1,a,c}\Delta_{a,a+1,b}\bigg) \\
&=& \Delta_{b-1,b,b+1}\Delta_{c-1,c,c+1} \bigg( \Delta_{a-1,a,a+1} \Delta_{a,b,c}\bigg)
\end{eqnarray*}
where the last equality follows from a three-term Pl{\"u}cker relation. The final expression is the Pl\"ucker variable $\Delta_J$ up to multiplication by the frozen variables
$\Delta_{a-1,a,a+1} \Delta_{b-1,b,b+1}\Delta_{c-1,c,c+1}$.
\end{proof}

\section{Triple Dimer Configurations for Cubic Differences}\label{section:tripledimers}

\begin{subsection}{Webs}\label{subsection:webs}
In order to classify the triple dimer configurations that give expressions for $A$ and $B$, we associate each triple dimer configuration to a \textbf{web}. Webs were first introduced by Kuperberg in \cite{kup96}. 

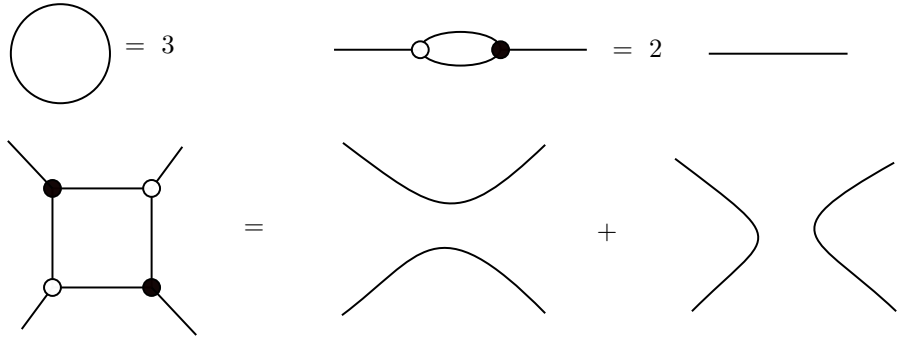
\begin{figure}
\centering

\tikzset{every picture/.style={line width=0.75pt}} 
\begin{tikzpicture}[x=0.75pt,y=0.75pt,yscale=-1,xscale=1]

\draw   (107,34) .. controls (107,20.19) and (118.19,9) .. (132,9) .. controls (145.81,9) and (157,20.19) .. (157,34) .. controls (157,47.81) and (145.81,59) .. (132,59) .. controls (118.19,59) and (107,47.81) .. (107,34) -- cycle ;
\draw    (270,32) -- (313.5,32) ;
\draw    (354,32) -- (397.5,32) ;
\draw   (313.5,31.5) .. controls (313.5,26.81) and (322.57,23) .. (333.75,23) .. controls (344.93,23) and (354,26.81) .. (354,31.5) .. controls (354,36.19) and (344.93,40) .. (333.75,40) .. controls (322.57,40) and (313.5,36.19) .. (313.5,31.5) -- cycle ;
\draw  [fill={rgb, 255:red, 255; green, 255; blue, 255 }  ,fill opacity=0.98 ] (309.25,32) .. controls (309.25,29.65) and (311.15,27.75) .. (313.5,27.75) .. controls (315.85,27.75) and (317.75,29.65) .. (317.75,32) .. controls (317.75,34.35) and (315.85,36.25) .. (313.5,36.25) .. controls (311.15,36.25) and (309.25,34.35) .. (309.25,32) -- cycle ;
\draw  [fill={rgb, 255:red, 13; green, 0; blue, 0 }  ,fill opacity=0.98 ] (349.75,32) .. controls (349.75,29.65) and (351.65,27.75) .. (354,27.75) .. controls (356.35,27.75) and (358.25,29.65) .. (358.25,32) .. controls (358.25,34.35) and (356.35,36.25) .. (354,36.25) .. controls (351.65,36.25) and (349.75,34.35) .. (349.75,32) -- cycle ;
\draw   (128,102) -- (178,102) -- (178,152) -- (128,152) -- cycle ;
\draw  [fill={rgb, 255:red, 13; green, 0; blue, 0 }  ,fill opacity=0.98 ] (123.75,102) .. controls (123.75,99.65) and (125.65,97.75) .. (128,97.75) .. controls (130.35,97.75) and (132.25,99.65) .. (132.25,102) .. controls (132.25,104.35) and (130.35,106.25) .. (128,106.25) .. controls (125.65,106.25) and (123.75,104.35) .. (123.75,102) -- cycle ;
\draw  [fill={rgb, 255:red, 13; green, 0; blue, 0 }  ,fill opacity=0.98 ] (173.75,152) .. controls (173.75,149.65) and (175.65,147.75) .. (178,147.75) .. controls (180.35,147.75) and (182.25,149.65) .. (182.25,152) .. controls (182.25,154.35) and (180.35,156.25) .. (178,156.25) .. controls (175.65,156.25) and (173.75,154.35) .. (173.75,152) -- cycle ;
\draw    (105.5,78) -- (128,102) ;
\draw    (178,152) -- (200.5,176) ;
\draw    (112.5,173) -- (128,152) ;
\draw    (178,102) -- (193.5,81) ;
\draw  [fill={rgb, 255:red, 255; green, 255; blue, 255 }  ,fill opacity=0.98 ] (123.75,152) .. controls (123.75,149.65) and (125.65,147.75) .. (128,147.75) .. controls (130.35,147.75) and (132.25,149.65) .. (132.25,152) .. controls (132.25,154.35) and (130.35,156.25) .. (128,156.25) .. controls (125.65,156.25) and (123.75,154.35) .. (123.75,152) -- cycle ;
\draw  [fill={rgb, 255:red, 255; green, 255; blue, 255 }  ,fill opacity=0.98 ] (173.75,102) .. controls (173.75,99.65) and (175.65,97.75) .. (178,97.75) .. controls (180.35,97.75) and (182.25,99.65) .. (182.25,102) .. controls (182.25,104.35) and (180.35,106.25) .. (178,106.25) .. controls (175.65,106.25) and (173.75,104.35) .. (173.75,102) -- cycle ;
\draw    (274,166) .. controls (314,136) and (320.5,108) .. (376.5,165) ;
\draw    (274.5,79) .. controls (323.5,116) and (331.5,123) .. (376.5,80) ;
\draw    (442,87) .. controls (501.5,132) and (491.5,123) .. (450.5,164) ;
\draw    (552.5,89) .. controls (486.5,126) and (512.5,125) .. (553.5,164) ;
\draw    (459,34) -- (529,34) ;

\draw (163,23.4) node [anchor=north west][inner sep=0.75pt]    {$=\ 3$};
\draw (409,25.4) node [anchor=north west][inner sep=0.75pt]    {$=\ 2$};
\draw (223,118.4) node [anchor=north west][inner sep=0.75pt]    {$=$};
\draw (401,118.4) node [anchor=north west][inner sep=0.75pt]    {$+$};

\end{tikzpicture}
\caption{Planar skein relations for web reduction.}
\label{fig:skeinrelations}
\end{figure}

\begin{definition}\label{definition:web}
A \textbf{web} $W$ is a planar bipartite graph embedded in the disc such that all internal vertices are trivalent. Within the disc, $W$ may also include some vertex-less directed cycles, as well as some directed edges from a black boundary vertex to a white boundary vertex, which we call ``paths."

We will require all boundary vertices of our webs to be either univalent or isolated (0-valent); by an abuse of notation, we will treat this condition as intrinsic to the definition for the remainder of the paper.
\end{definition}

Given a web $W$, we may consider its \textbf{web interior} $\mathring{W}$, which only consists of the internal vertices and edges of $W$. A \textbf{connected component} of $W$ is a connected component of $\mathring{W}$ along with the boundary vertices it attaches to; namely, we do not consider the boundary of the disc as an edge that connects all boundary vertices.

A \textbf{non-elliptic} web $W$ is a web containing no interior faces bounded by four or fewer edges; i.e., it contains no contractible cycles, bigons, or squares. Every web may be expressed as a sum of nonelliptic webs via the reduction moves in Figure \ref{fig:skeinrelations}, see \cite{kup96} and \cite{fominpasha2016}.

Lastly, we introduce terminology for some common web components. We will call an internal white vertex incident to three black boundary vertices a \textbf{tripod}; and we will call a component with one black internal vertex adjacent to three white internal vertices, each of which are adjacent to two black boundary vertices, a \textbf{hexapod}. See Figure \ref{fig:pod}.

\begin{figure}
    \centering
    \includegraphics[scale=.22]{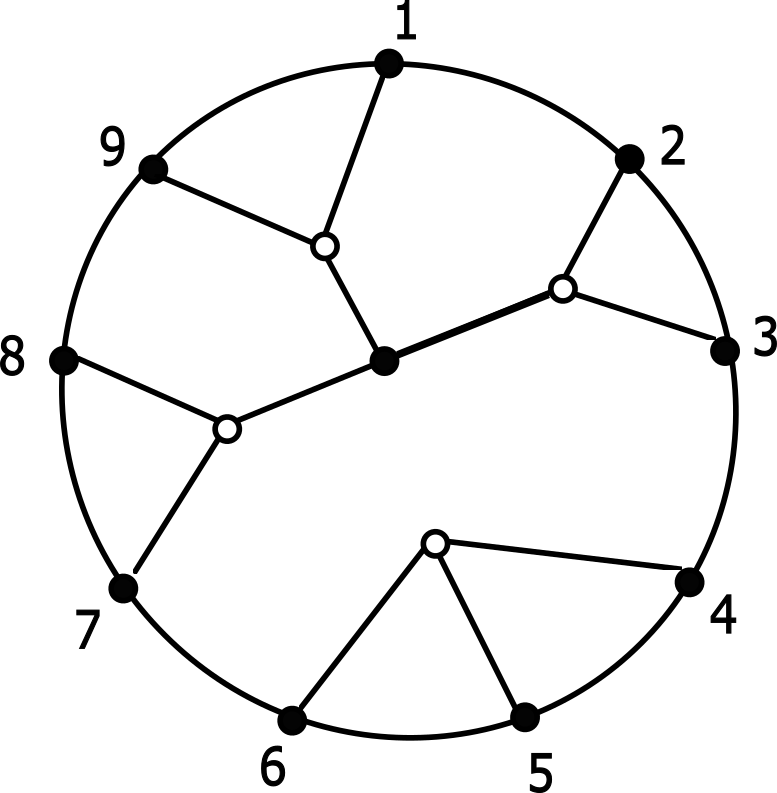}
    \caption{A web composed of a hexapod (on vertices 1, 2, 3, 7, 8, and 9) and a tripod (on vertices 4, 5, and 6)}
    \label{fig:pod}
\end{figure}
\end{subsection}

\begin{subsection}{Triple Dimer Configurations As Webs}\label{subsection:tripledimers}
Given a triple dimer configuration $D$ on a plabic graph $G$, let $G_D$ be the subgraph of $G$ containing all edges included at least once in $D$. We create a web $W(D)$ corresponding to $D$ as follows:
\begin{enumerate}
    \item For each boundary vertex in $G$, create a corresponding boundary vertex in $W(D)$. Color each boundary vertex white if it is adjacent to a doubled edge in $D$, and black otherwise.
    \item For each interior cycle in $G_D$ consisting entirely of bivalent vertices, add a vertexless loop to $W(D)$, oriented arbitrarily.
    \item For each chain of bivalent vertices connecting two boundary vertices $v,v' \in G_D$, construct a path between the corresponding boundary vertices in $W(D)$. To orient this path, note that the chain must correspond to a path alternating between singled and doubled edges in $D$, and since all boundary vertices of $G$ are black, that path must have even length. Thus exactly one of $v$ and $v'$ must be adjacent to a doubled edge in $D$, and therefore colored white in $W(D)$. Orient the path in $W(D)$ towards the white vertex.
    \item For each connected component of $G_D$ containing at least one trivalent vertex, include a corresponding component in $W(D)$ with all bivalent vertices removed, merging each pair of edges that was adjacent to a deleted vertex. Retain the color of all interior trivalent vertices. (Note that all trivalent vertices in $G_D$ must be adjacent to single edges in $D$, so the graph will remain bipartite, again because chains of bivalent vertices in $G_D$ correspond to paths alternating between singled and doubled edges in $D$.)
\end{enumerate}
Note that this construction of $W(D)$ is equivalent to the construction of a web from a triple dimer via ``weblike subgraphs" given in \cite{lam2015}.

\begin{example}\label{example:dimertoweb}
\begin{figure}
    \centering
    \includegraphics[width=\textwidth]{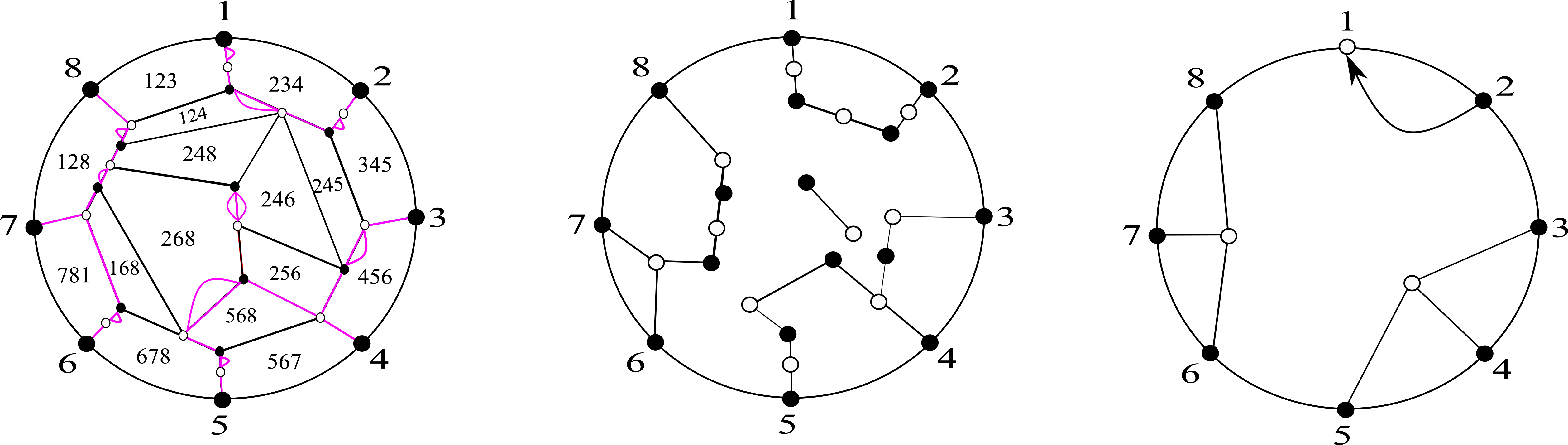}
    \caption{From a triple dimer configuration $D$ (left), to the corresponding graph $G_D$ (middle), to a web $W(D)$ (right).}
    \label{fig:dimertoweb}
\end{figure}

Consider the triple dimer configuration $D$ on the plabic graph $G$ in Figure \ref{fig:dimertoweb}. We create the corresponding subgraph $G_D$ by removing duplicate edges of $D$. In $W(D)$, we include an oriented edge from vertex 2 to vertex 1, since vertex 1 is adjacent to a doubled edge in $D$, and color vertex 1 white, while all other boundary vertices remain black. We also include two connected components corresponding to the components of $G_D$ that contain trivalent vertices, with bivalent vertices removed. Note that we ignore isolated edges in $G_D$, which come from tripled edges in $D$.
\end{example}
\end{subsection}

\begin{subsection}{Enumeration of Non-Elliptic Webs}
The proofs of our main theorems will rely on the enumeration of several classes of non-elliptic webs in Lemmas \ref{lemma:pathlesswebs}, \ref{lemma:webswithpathsadj}, \ref{lemma:webswithpathsnonadj}, and \ref{lemma:pathlesswebs9}. Propositions \ref{prop:eulerchar} and \ref{prop:minvertices} place significant bounds on this enumeration, in particular confirming that the classes are finite. Lemma \ref{lemma:2deg2verts} is used to prove Proposition \ref{prop:minvertices}.

\begin{proposition}\label{prop:eulerchar}
Let $W$ be a nonelliptic web with $n$ boundary vertices. Let $c$ be the number of cycles in $W$, let $k$ be the number of connected components of $W$, and let $|V_{int}|$ be the number of internal vertices in $W$. Then $|V_{int}| = n+2c-2k$.
\end{proposition}

\begin{proof}
By planarity of $W$ and the Euler characteristic, we have that $|E| = |V|-k+c$ where $E$ is the set of edges of $W$ and $V$ is the set of all vertices in $W$. Moreover, we have $\sum_{v \in V} \text{deg}(v) = 2|E|$. Since all internal vertices in the web are trivalent and we have $n$ boundary vertices of degree $1$, we see
\[n + 3(|V|-n) = \sum_{v \in V} \text{deg}(v) = 2|E| = 2|V| - 2k + 2c,\]
which implies that $|V| - 2n = -2k+2c$. Since $|V| = |V_{int}|+n$, it follows that $|V_{int}| = n+2c-2k$. 
\end{proof}

\begin{lemma}\label{lemma:2deg2verts}
Let $\mathring{W}$ be the interior of a nonelliptic web $W$, and assume that $\mathring{W}$ is connected and consists only of cycles. Then there exist two adjacent vertices in $\mathring{W}$ that are bivalent in $\mathring{W}$.
\end{lemma}
\begin{proof}

Given a connected nonelliptic web interior $\mathring{W}$ that consists entirely of cycles, we may construct $\mathring{W}$ by beginning with a central $2m$-gon and adding $2m$-gons exterior
of it one by one, so that every intermediate step remains a valid web interior $\mathring{V}$ of a different web $V$. Since webs are trivalent, all vertices with degree less than three in $\mathring{V}$ must be adjacent to the boundary in $V$; and since webs are planar, only vertices not enclosed by other edges in $V$ (``exterior vertices") may possibly be adjacent to the boundary in $V$. Therefore, at any step of the process of building $\mathring{W}$, the additional $2m$-gon cannot share more than one edge with any one old $2m$-gon (see the left of Figure \ref{fig:bivalent}), and it also must only share one set of adjacent edges along the boundary of the web interior (see the right of Figure \ref{fig:bivalent}): otherwise, in both cases, we would have interior bivalent vertices in $\mathring{V}$.

\begin{figure}
    \centering
    \includegraphics[scale=.14]{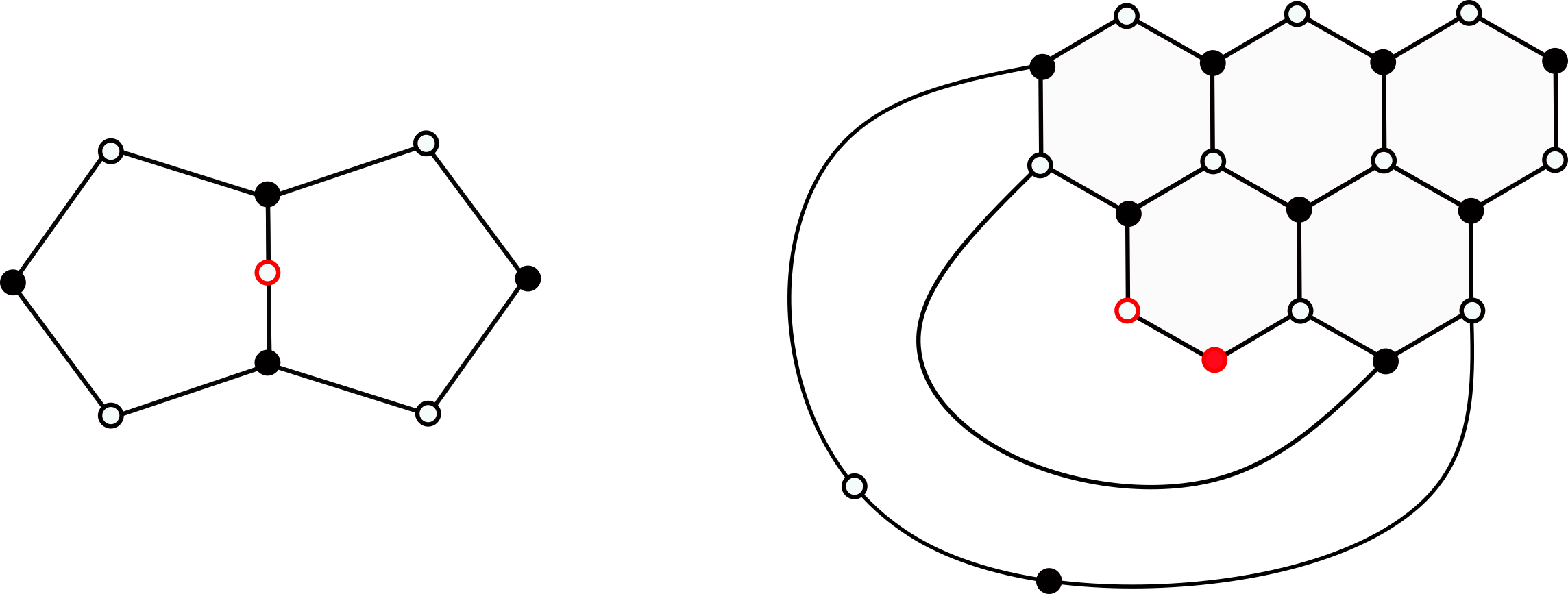}
    \caption{Ways to add a hexagon to a web interior that do not produce another valid web interior.}
    \label{fig:bivalent}
\end{figure}

\begin{figure}
    \centering
    \includegraphics[scale=.15]{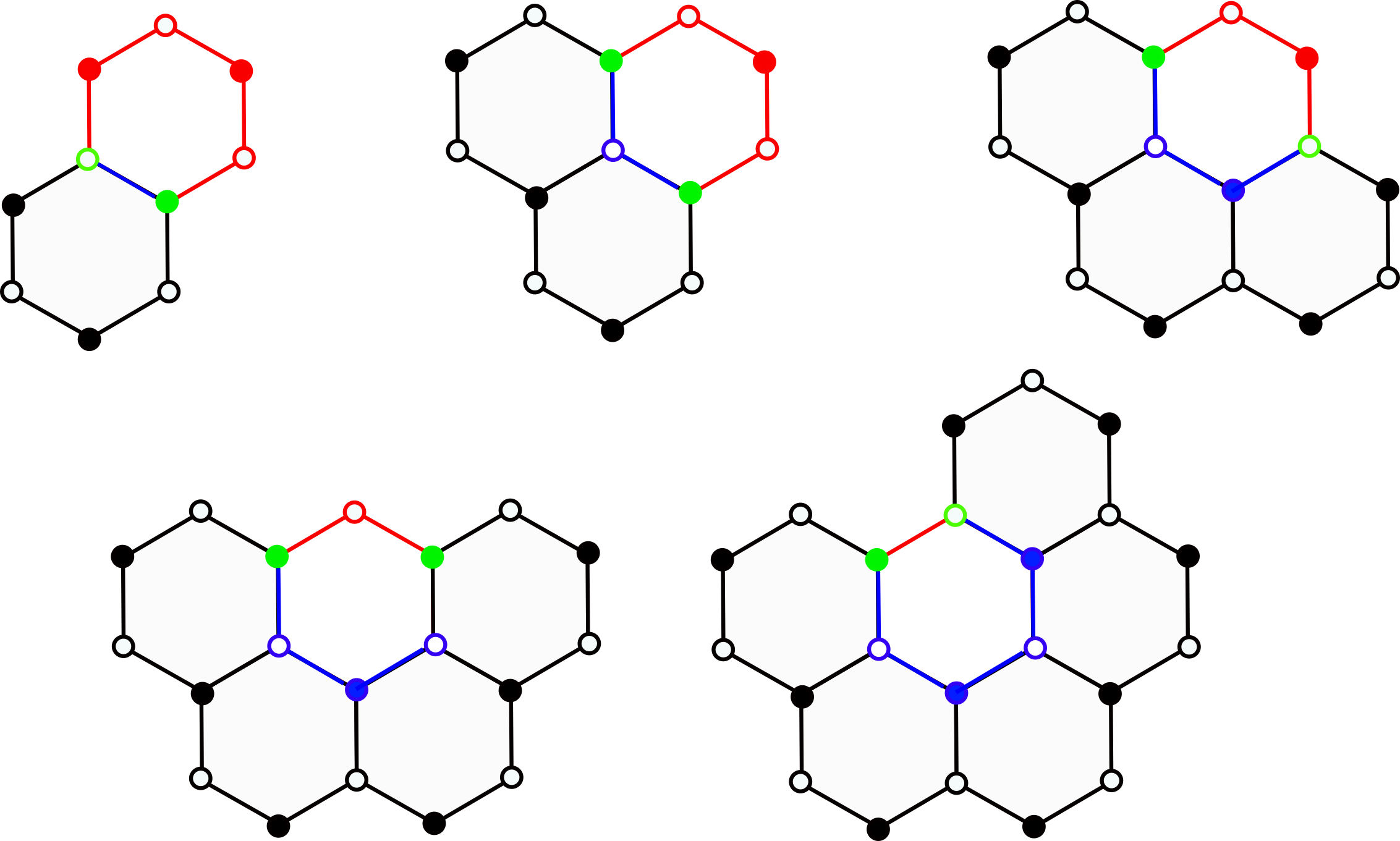}
    \caption{The five ways to add a hexagon to a web interior: the boundary initially consists of the blue and green vertices (where the green vertices are bivalent), and after the hexagon is added, it consists of the red and green vertices (where the green vertices are trivalent). Note that in each case, adding the hexagon does not change $\mathcal{BT}$. For instance, adding a hexagon as in the top leftmost figure replaces two bivalent vertices with four bivalent vertices and two trivalent vertices.}
    \label{fig:addcycles}
\end{figure}

Additionally, note that at any step, every exterior vertex of $\mathring{V}$ must be either bivalent or trivalent, since $\mathring{V}$ consists only of cycles. Let $\mathcal{BT}_{\mathring{V}}$ be the number of bivalent exterior vertices minus the number of trivalent exterior vertices of $\mathring{V}$. When we build $\mathring{W}$ starting from a central $2m$-gon $\mathring{W_0}$, this central
$2m$-gon has at least 6 bivalent vertices and 0 trivalent vertices, so $\mathcal{BT}_{\mathring{W_0}}\geq 6$. Figure \ref{fig:addcycles} shows the possible effects of adding a hexagon; none of these change $\mathcal{BT}$. The effects of adding a larger $2m$-gon would be analogous, with possibly more trivalent vertices being replaced by possibly more bivalent vertices, so adding a larger $2m$-gon would only increase $\mathcal{BT}$. Therefore $\mathring{W}$ has $\mathcal{BT}_{\mathring{W}}\geq 6$ also, i.e. $\mathring{W}$ has more exterior bivalent vertices than exterior trivalent vertices. The exterior vertices of $\mathring{W}$ (each of which is either bivalent or trivalent) form a cycle, and since the discrepancy $\mathcal{BT}_{\mathring{V}}$ is positive for each such intermediate web interior $\mathring{V}$, at least two exterior bivalent vertices must be adjacent.
\end{proof}

\begin{proposition}\label{prop:minvertices}
Let $W$ be a nonelliptic web, let $c$ be the number of cycles in $W$, and let $|V_{int}|$ be the number of internal vertices in $W$ (equivalently, $|V_{int}|$ is the number of vertices in $\mathring{W}$). Then
\begin{itemize}
    \item if $c \geq 1$, $|V_{int}| \geq 2c+4$,
    \item if $c \geq 2$, $|V_{int}| \geq 2c+6$,
    \item if $c \geq 3$, $|V_{int}| \geq 2c+7$,
    \item and if $c \geq 4$, $|V_{int}| \geq 2c + 8$. 
\end{itemize}
\end{proposition}
\begin{proof}
Let $a_c$ be the minimal number of vertices in a nonelliptic web interior with $c$ cycles, so $|V_{int}| \geq a_c$. We first show that $a_c$ is also the minimum number of vertices in a connected nonelliptic web interior with $c$ cycles that is composed entirely of cycles. To accomplish this, given a nonelliptic web interior $\mathring{W}$ with $c$ cycles and $v$ vertices, we demonstrate that if $\mathring{W}$ is disconnected or contains a vertex or edge that is not part of a cycle, then $v>a_c$.

If $\mathring{W}$ contains a vertex that is not part of a cycle, then we may delete it to arrive at a web interior with $c$ cycles and $v-1$ vertices, so by minimality, $v>a_c$. If $\mathring{W}$ contains edges that are not part of cycles, we may delete them to create a disconnected web interior that has $v$ vertices and is composed entirely of cycles, so the only remaining case is when $\mathring{W}$ is disconnected and composed entirely of cycles. In this case, by Lemma \ref{lemma:2deg2verts}, there exist two adjacent bivalent vertices in each connected component; identifying two such pairs of vertices from separate components yields a web interior with $c$ cycles and $v-2$ vertices, so $v>a_c$ in this case also.

Therefore $a_c$ is the minimum number of vertices in a connected nonelliptic web interior with $c$ cycles that is composed entirely of cycles. Let $\mathring{W}$ be such a web interior with $c$ cycles and $a_c$ vertices. Then from Lemma \ref{lemma:2deg2verts}, there exist two adjacent bivalent vertices in $\mathring{W}$; these vertices must both be part of only one cycle. Therefore removing both vertices creates a web interior with $c-1$ cycles and $a_c-2$ vertices, so $a_{c-1} \leq a_c-2$, so $a_c \geq a_{c-1}+2$. It follows that, for any $i<c$, \[|V_{int}| \geq a_c \geq a_i + 2(c-i) = 2c+ (a_i-2i).\]

Clearly $a_1=6$, the minimal number of vertices in a nonelliptic web interior with one cycle. The only way to construct a connected nonelliptic web interior consisting entirely of two cycles is to overlap the cycles at one edge; the number of vertices is minimized when both cycles are hexagons, in which case there are ten vertices, so $a_2=10$. Next, we must be able to construct any connected nonelliptic web interior consisting entirely of three cycles by adding a cycle to a connected nonelliptic web interior consisting entirely of two cycles. There is one place to add the new cycle that only creates three new vertices, and no way to create fewer; therefore $a_3=13$. The same is true to add a fourth cycle, so $a_4=16$. Now, applying the above formula for $i=1,2,3,4$ yields the three desired statements. See {\tt http://oeis.org/A121149} for more details.
\end{proof}

\begin{figure}
    \centering
    \includegraphics[width=\textwidth]{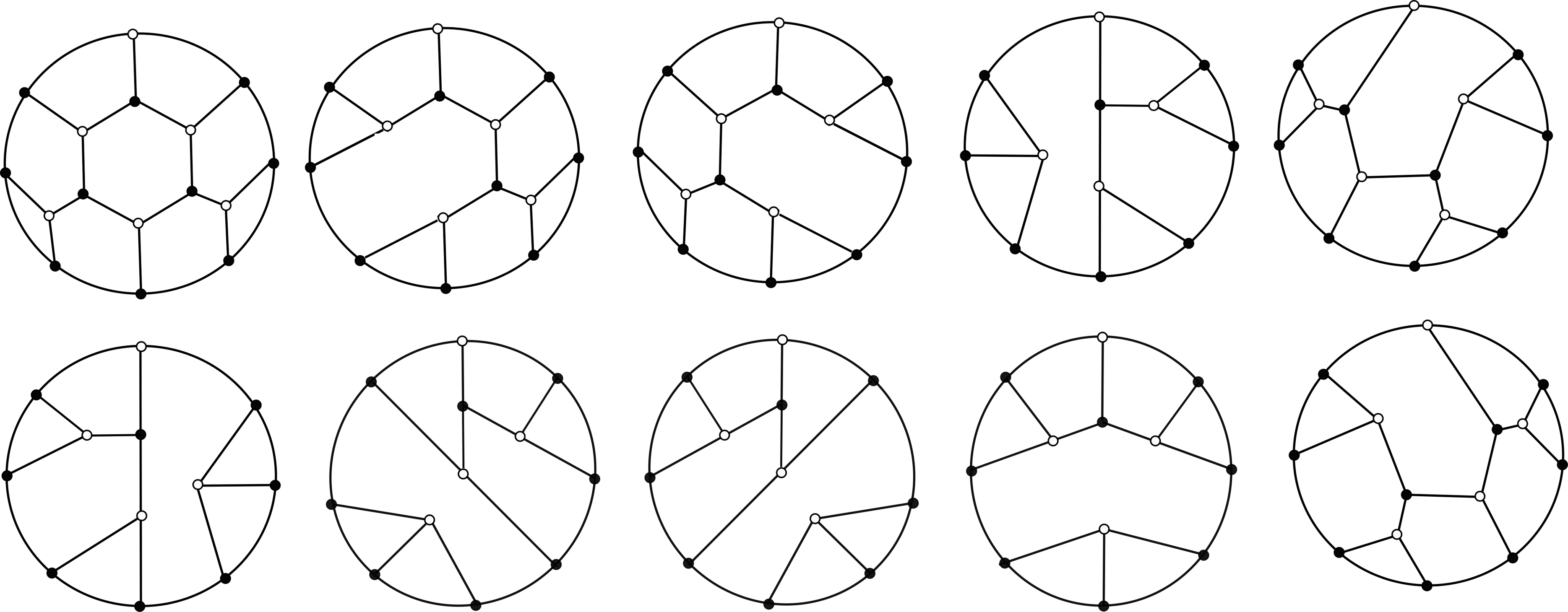}
    \caption{The only non-elliptic webs with seven black boundary vertices and one white boundary vertex that do not contain paths between boundary vertices.}
    \label{fig:pathlesswebs}
\end{figure}

We now begin our enumeration of particular non-elliptic webs with 8 boundary vertices, which will assist in the proofs of Theorems \ref{theorem:connectivityforA} and \ref{theorem:connectivityforB}.

\begin{lemma}\label{lemma:pathlesswebs}
All non-elliptic webs with seven black boundary vertices and one white boundary vertex that do not have have paths between boundary vertices are listed in Figure \ref{fig:pathlesswebs}.
\end{lemma}

\begin{proof}
Proposition \ref{prop:eulerchar} gives $|V_{int}|=8+2c-2k$ for the number of internal vertices in such a web $W$ that has $k$ connected components and $c$ cycles.

If $k=1$, $|V_{int}|=2c+6$, so from Proposition \ref{prop:minvertices}, $c<3$ in this case. If $c=2$, $|V_{int}|=10$; all of these vertices are required to construct the two cycles, but due to the colors of the boundary vertices, it is impossible to connect all boundary vertices to the interior hexagons without adding more vertices. Therefore $c$ cannot be 2. If $c=1$, $|V_{int}|=8$. Six of these internal vertices must be used to create the hexagon; and since there is only one white boundary vertex, the other two internal vertices must be white and adjacent to two of the black vertices in the hexagon. The resulting web is possible to complete, as shown in the top left web in Figure \ref{fig:pathlesswebs}. Finally, if $c=0$, $|V_{int}|=6$. The only bipartite tree webs with 7 black leaves, 1 white leaf, and 6 internal trivalent vertices are shown in Figure \ref{fig:pathlesswebs}; there are four.

If $k=2$, $|V_{int}|=2c+4$, so from Proposition \ref{prop:minvertices}, $c<2$ in this case. If $c=1$, $|V_{int}|=6$; all of these vertices are required to construct the cycle, and again due to the colors of the boundary vertices, it is impossible to complete the web without adding more vertices. Therefore $c=0$. The only bipartite forest webs with 7 black leaves, 1 white leaf, 4 internal trivalent vertices, and 2 connected components are shown in Figure \ref{fig:pathlesswebs}; there are five.

If $k\geq 3$, we have from Proposition \ref{prop:minvertices} that $c=0$, implying that $|V_{int}|\leq 2$. It is impossible to connect all boundary vertices with this few internal vertices. Therefore Figure \ref{fig:pathlesswebs} enumerates all pathless non-elliptic webs with seven black boundary vertices and one white boundary vertex.
\end{proof}

\begin{figure}
    \centering
    \includegraphics[scale=.13]{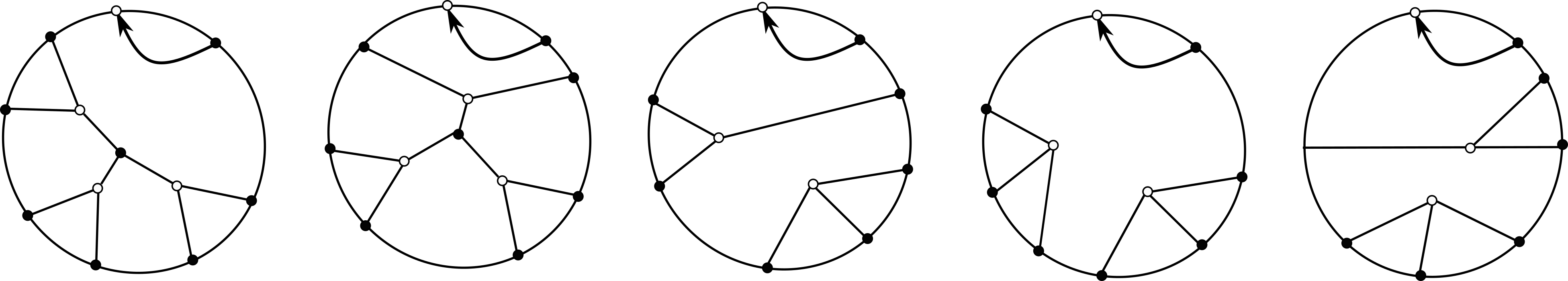}
    \caption{All possible nonelliptic webs with a path between two adjacent boundary vertices and six other black boundary vertices.}
    \label{fig:webswithpathsB}
\end{figure}

\begin{lemma}\label{lemma:webswithpathsadj}
All non-elliptic webs with a path between two adjacent boundary vertices and six other black boundary vertices are listed in Figure \ref{fig:webswithpathsB}, up to reflection.
\end{lemma}

\begin{proof}
Proposition \ref{prop:eulerchar} gives $|V_{int}| = 6+2c-2k$ for the number of internal vertices in such a web $W$ with $k$ connected components and $c$ cycles.

If $k=1$, then $|V_{int}| = 2c+4$, so from Proposition \ref{prop:minvertices}, $c < 2$ in this case. If $c=1$, $|V_{int}| = 6$, and we must use all six internal vertices to make a hexagon. However, the three black internal vertices cannot be connected to the boundary since the web must be bipartite, so it is impossible to complete a valid web where $c=1$. If $c=0$, the web must be composed of a path $2 \to 1$ and a bipartite tree with six black leaves and four trivalent internal vertices. Suppose that we have a white vertex adjacent to a black leaf. We claim that this white vertex must be adjacent to exactly one other leaf. If it were attached to at least two other leaves, that would force us to finish a connected component without including all the vertices. If it was attached to no other leaves, this would force the creation of two other internal black vertices, which would in turn force the creation of four white vertices, two for each new black vertex. This would yield a total of seven internal vertices, which is too many. Therefore, the white vertex must be adjacent to exactly one other black leaf, along with one new black internal vertex adjacent to two more white internal vertices. This tree appears alongside the requisite path in the two webs shown in the top row of Figure \ref{fig:webswithpathsB}.

If $k=2$, then $|V_{int}| = 2c+2$, so from Proposition \ref{prop:minvertices}, $c=0$. There is only one way to construct a bipartite forest with two connected components, six black leaves, and three trivalent internal vertices; the possible configurations of these trees are shown alongside the requisite path in the bottom three webs in Figure \ref{fig:webswithpathsB}.

Finally, if $k \geq 3$, then $|V_{int}| \leq 2c$, but $c=0$ by Proposition \ref{prop:minvertices} and there is no way to complete a web with 0 internal vertices. Therefore Figure \ref{fig:webswithpathsB} enumerates all non-elliptic webs with a path between adjacent boundary vertices and six other black boundary vertices, up to reflection.
\end{proof}

\begin{figure}
    \centering
    \includegraphics[scale=.16]{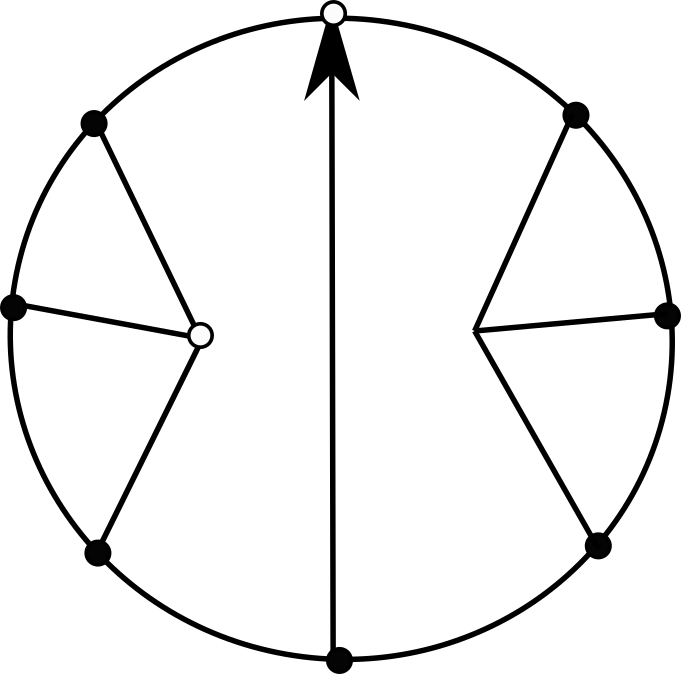}
    \caption{The unique web that connects non-adjacent boundary vertices.}
    \label{fig:5to1}
\end{figure}

\begin{lemma}\label{lemma:webswithpathsnonadj}
The only non-elliptic web with a path between non-adjacent boundary vertices and six other black boundary vertices is depicted in Figure \ref{fig:5to1}.
\end{lemma}

\begin{proof}
For ease of reference, we label the target of the path vertex 1, and continue labeling vertices clockwise. We first show that the only possible path between non-adjacent boundary vertices must be $5 \to 1$. If we had a path $3 \to 1$ or $7 \to 1$, then since webs are planar, one boundary vertex ($2$ or $8$ respectively) would be isolated, unable to be in the same connected component as any other boundary vertex. Then Proposition \ref{prop:eulerchar} gives $|V_{int}| = 1+2c-2k$ for this portion of the web, and by Proposition \ref{prop:minvertices}, it cannot have any cycles. Therefore this portion of the web must have $-1$ internal vertices, which is impossible, so there is no way to complete the web. Similarly, if we had a path $4 \to 1$ or $6 \to 1$, then two vertices ($2$ and $3$ or $7$ and $8$ respectively) would be isolated. Then Proposition \ref{prop:eulerchar} gives $|V_{int}| = 2+2c-2k$ for this portion of the web, and by Proposition \ref{prop:minvertices}, it cannot have any cycles (since it has at least one connected component). Therefore this portion of the web must have zero internal vertices and only one connected component, which is impossible, so there is no way to complete the web.

We now show that Figure \ref{fig:5to1} depicts the only web with a path $5 \to 1$. This path isolates the vertices $2,3,4$ and $6,7,8$, so Proposition \ref{prop:eulerchar} gives $|V_{int}| = 3+2c-2k$ for the portions of the web on either side of the path. Neither side can have any cycles by Lemma \ref{prop:minvertices}, so the web must have one internal vertex adjacent to each of these boundary vertices. This produces the web in Figure \ref{fig:5to1}.
\end{proof}

The following lemma will be of use in the proofs of Theorems \ref{theorem:connectivityforC} and \ref{theorem:connectivityforZ}.
\begin{figure}
    \centering
    \includegraphics[width=\textwidth]{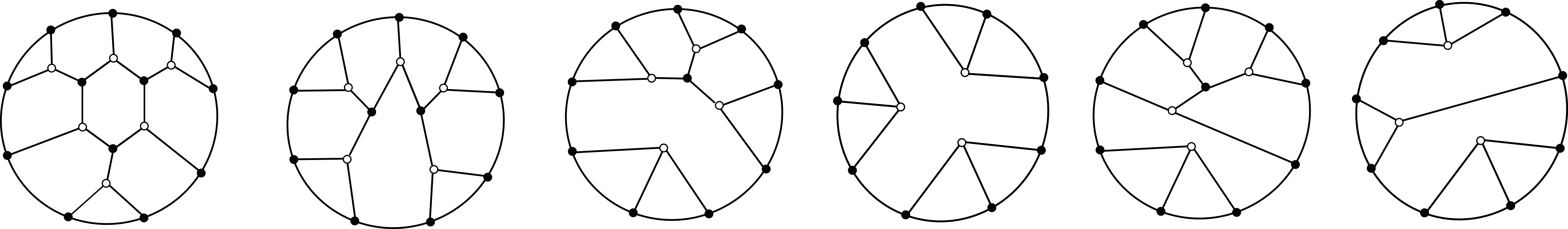}
    \caption{All nonelliptic webs with nine black boundary vertices.}
    \label{fig:pathlesswebs9}
\end{figure}
\begin{lemma}\label{lemma:pathlesswebs9}
All non-elliptic webs with nine black boundary vertices are the dihedral translates of those listed in Figure \ref{fig:pathlesswebs9}. 
\end{lemma}
\begin{proof}
Proposition \ref{prop:eulerchar} gives $|V_{int}|=9+2c-2k$ for the number of internal vertices in such a web $W$ that has $k$ connected components and $c$ cycles.
Note that since there are no white boundary vertices, this web cannot contain any directed paths.

If $k=1$, $|V_{int}|=2c+7$, so from Proposition \ref{prop:minvertices}, $c<4$ in this case. If $c=3$, $|V_{int}|=13$; all of these vertices are required to construct the three cycles, but due to the colors of the boundary vertices, it is impossible to connect all boundary vertices to the interior hexagons without adding more vertices.  Thus, $c$ cannot be $3$.  If $c=2$, $|V_{int}|=11$; all but one of these vertices are required to construct the two cycles, but again due to the colors of the boundary vertices, it is impossible to connect all boundary vertices to the interior hexagons without adding more vertices. Therefore $c$ cannot be 2. If $c=1$, $|V_{int}|=9$. Six of these internal vertices must be used to create the hexagon; and since there are only black boundary vertices, the other three internal vertices must be white and adjacent to two of the black vertices in the hexagon. The resulting web is possible to complete, as shown in Figure \ref{fig:pathlesswebs9}. Finally, if $c=0$, $|V_{int}|=7$. The only bipartite tree with 9 black leaves and 7 internal trivalent vertices is shown.

If $k=2$, then $|V_{int}| = 2c+5$, which implies that $c < 2$.  If $c=1$, then there are seven internal vertices, of which six of these must comprise a cycle (a hexagon).  However, then there are not enough internal vertices left to connect to nine degree one black vertices as two connected components. On the other hand, if $c=0$, there are five internal vertices, and it is possible to construct configurations consisting of a hexapod and a tripod as its two connected components.

If $k=3$, then $|V_{int}| = 2c+3$, which implies that $c=0$, and there are simply three internal vertices.  Consequently, the only allowable configurations in this case are composed of three tripods.
Lastly, we observe that if $k \geq 4$, we again have $c=0$ and hence only one internal vertex; it is impossible to connect all boundary vertices with this few internal vertices, so no configurations exist in this case.
\end{proof}
\end{subsection}

\begin{subsection}{Webs for Cubic Differences}
Our main theorems give non-elliptic webs that describe triple dimers for twists of $A$, $B$, $C$, $Z$, and their dihedral translates. We work through examples of their use in Section \ref{section:appendixAB}.

We begin by establishing notation. Given $n \geq 8$ and a plabic graph $G$ for $\Gr(3,n)$, recall that $\mathcal{D}^3(G)$ is the set of triple dimers on $G$, and let $\mathcal{W}$ denote the set of nonelliptic webs. For any $D \in \mathcal{D}^3(G)$, we may 
apply the reduction rules of Figure \ref{fig:skeinrelations} to write its corresponding web $W(D)$ as a sum of nonelliptic summands:
\[W(D)=\sum_{W \in \mathcal{W}} C_W^D W,\]
where $C_W^D \in \Z_{\geq 0}$ is the coefficient of the nonelliptic web $W$ in $W(D)$.  Additionally, given any web $W$ with $n'$ boundary vertices, $n\geq n'$, and $S=\{s_1<\dots<s_{n'}\} \subseteq [n]$, let $W^S$ denote the web with $n$ boundary vertices and the following properties:
\begin{itemize}
    \item All boundary vertices with labels in $[n] \setminus S$ are isolated, i.e. they are not adjacent to edges.
    \item When all boundary vertices of $W^S$ with labels not in $S$ are removed, and the remaining vertices are relabeled $s_1 \mapsto 1, s_2 \mapsto 2,\dots, s_{n'} \mapsto n'$, the resulting web is $W$.
\end{itemize}

\begin{figure}
    \centering
    \includegraphics[width=\textwidth]{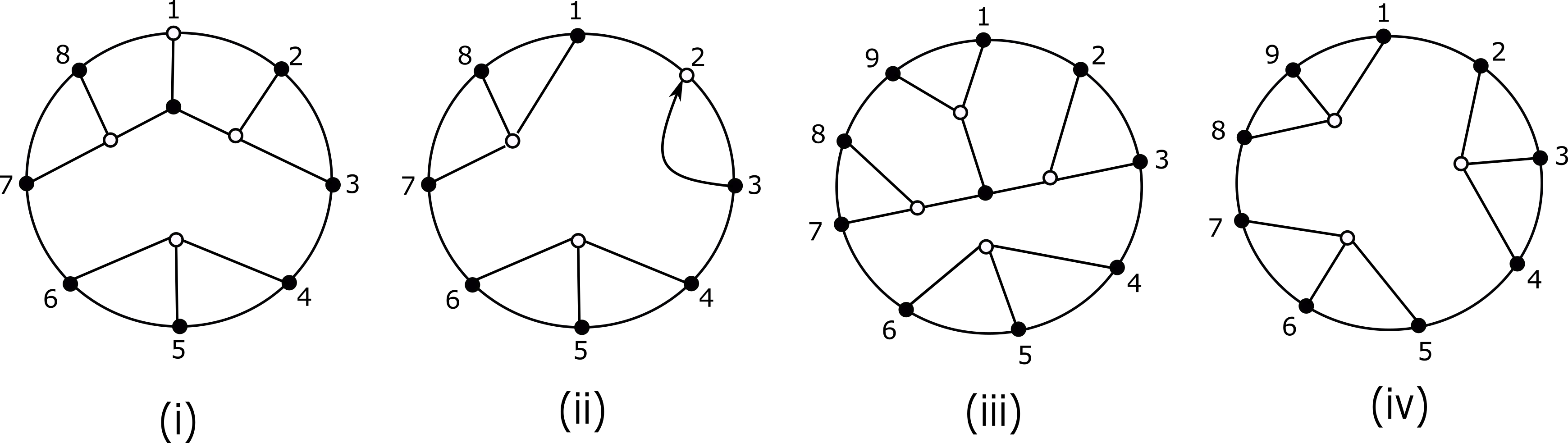}
    \caption{Non-elliptic webs corresponding to cubic differences in $\Gr(3,8)$ and $\Gr(3,9)$. (i) is the batwing corresponding to $A$, (ii) is the octopus corresponding to $B$, (iii) is the hexa-crab corresponding to $C$, and (iv) is the tri-crab corresponding to $Z$.}
    \label{fig:mainresults}
\end{figure}

\begin{theorem} \label{theorem:connectivityforA} In $\Gr(3,8)$, write $A=A_1-A_2-A_3$\[=\Delta_{134}\Delta_{258}\Delta_{167} - \Delta_{134}\Delta_{678}\Delta_{125}-\Delta_{158}\Delta_{234}\Delta_{167}.\]
Then \[\T(A)=\sum_{D \in \mathcal{D}^3(G)} C_{batwing}^D wt_f(D),\]
where the ``batwing" is the nonelliptic web pictured in Figure \ref{fig:mainresults}(i).
Additionally, for any $n \geq 8$, $\sigma \in D_8$, and $S=\{s_1<\dots<s_8\} \subseteq [n]$, we have that $\T(\sigma(A)^S)=\sum_{D \in \mathcal{D}^3} C_{\sigma(batwing)^S}^D wt_f(D)$.
\end{theorem}

\begin{theorem} \label{theorem:connectivityforB} 
In $\Gr(3,8)$, write $B=B_1-B_2-B_3$\[ = \Delta_{258}\Delta_{134}\Delta_{267} - \Delta_{234}\Delta_{128}\Delta_{567}-\Delta_{234}\Delta_{258}\Delta_{167}.\]
Then \[\T(B)=\sum_{D \in \mathcal{D}^3(G)} C_{octopus}^D wt_f(D),\]
where the ``octopus" is the nonelliptic web pictured in Figure \ref{fig:mainresults}(ii).
Additionally, for any $n \geq 8$, $\sigma \in D_8$, and $S=\{s_1<\dots<s_8\} \subseteq [n]$, we have that $\T(\sigma(B)^S)=\sum_{D \in \mathcal{D}^3} C_{\sigma(octopus)^S}^D wt_f(D)$.
\end{theorem}

\begin{theorem} \label{theorem:connectivityforC}
In $\Gr(3,9)$, write $C = C_1+C_2-C_3-C_4$\[= \Delta_{124}\Delta_{357}\Delta_{689} + \Delta_{123}\Delta_{456}\Delta_{789}- \Delta_{124}\Delta_{356}\Delta_{789} - \Delta_{123}\Delta_{457}\Delta_{689}.\]
Then \[\T(C)=\sum_{D \in \mathcal{D}^3(G)} C_{\textrm{hexa-crab}}^D wt_f(D),\]
where the ``hexa-crab" is the nonelliptic web pictured in Figure \ref{fig:mainresults}(iii).
Additionally, for any $n \geq 9$, $\sigma \in D_9$, and $S=\{s_1<\dots<s_9\} \subseteq [n]$, we have that $\T(\sigma(C)^S)=\sum_{D \in \mathcal{D}^3} C_{\sigma(hexa-crab)^S}^D wt_f(D)$.
\end{theorem}

\begin{theorem} \label{theorem:connectivityforZ}
In $\Gr(3,9)$, write $Z = Z_1 - Z_2 - Z_3 - Z_4 $ 
\[=\Delta_{145} \Delta_{278} \Delta_{369} - \Delta_{245} \Delta_{178} \Delta_{369} - \Delta_{123}\Delta_{456}\Delta_{789} - \Delta_{129}\Delta_{345}\Delta_{678}.\]
Then \[\T(Z)=\sum_{D \in \mathcal{D}^3(G)} C_{tri-crab}^D wt_f(D),\]
where the ``tri-crab" is the nonelliptic web pictured in Figure \ref{fig:mainresults}(iv).
\end{theorem}

We will present an example of the application of Theorem \ref{theorem:connectivityforA} in Section \ref{subsection:exampleforA}, and an example of the application of Theorem \ref{theorem:connectivityforB} in Section \ref{subsection:exampleforB}.

In order to prove Theorems \ref{theorem:connectivityforA} through \ref{theorem:connectivityforZ}, we define a notion of \textit{compatibility} between webs and triple products of Pl{\"u}cker coordinates. We enumerate nonelliptic webs that are compatible with $A_1$, and then show that the set of nonelliptic webs compatible with $A_2$ or $A_3$ is exactly the set of nonelliptic webs compatible with $A_1$ that are not the batwing. It follows intuitively that the batwing should be the only web compatible with $A_1-A_2-A_3$.

Rigorously, we show that the twist of $A_1$ is the sum of the face weights of triple dimers that correspond to webs that have as a summand one of the webs with which $A_1$ is compatible. Proving the analogous statement for $A_2$ and $A_3$ allows cancellation to yield the given formula for the twist of $A$. A similar cancellation occurs with Theorem \ref{theorem:connectivityforB} for $B$ and the octopus, Theorem \ref{theorem:connectivityforC} for $C$ and the hexa-crab, and Theorem \ref{theorem:connectivityforZ} for $Z$ and the tri-crab.

\begin{definition}
Let $W$ be a non-elliptic web with $n$ boundary vertices $v_1,\dots, v_n$. Also let $I,J,K$ be subsets of $[n]$ with $|I|=|J|=|K|=3$; we associate them to the Pl{\"u}cker coordinates $\Delta_I$, $\Delta_J$, and $\Delta_K$. By Lemma 4.11 of \cite{lam2015}, $W$ has an edge coloring in the usual sense (that no incident edges have the same color) using the three colors red, blue, and green. We write that $W$ is \textbf{compatible} with the product $\Delta_I\Delta_J\Delta_K$ if there exists such an edge coloring of $W$ that satisfies the following conditions for all $i \in [n]$:
\begin{itemize}
    \item If $i \in I\setminus J \setminus K$ (resp. $J\setminus I \setminus K$, $K\setminus I \setminus J$), then $v_i$ is black and adjacent to a red (resp. blue, green) edge.
    \item If $i \in I\cap J \setminus K$ (resp. $I\cap K \setminus J$, $J\cap K \setminus I$), then $v_i$ is white and adjacent to a green (resp. blue, red) edge.
    \item Otherwise, $v_i$ is adjacent to no edges in $W$.
\end{itemize}
We write that $W$ is \textbf{uniquely compatible} with $\Delta_I\Delta_J\Delta_K$ if there exists a unique such edge coloring. 
\end{definition}
\begin{remark}\label{rmk:aIJK}
This definition is a special case of the notion of \textit{consistent labeling} defined in \cite[Section 4.5]{lam2015}. In Lam's notation, $a(I,J,K;W)$ counts the number of edge colorings that satisfy the above conditions; thus we write that $W$ is compatible with $\Delta_I\Delta_J\Delta_K$ if $a(I,J,K;W) >0$, and uniquely compatible with $\Delta_I\Delta_J\Delta_K$ if $a(I,J,K;W) =1$.
\end{remark}
\begin{remark}
The definition of compatibility is best understood through the lens of dimers. $W$ is compatible with $\Delta_I\Delta_J\Delta_K$ exactly when it is possible to construct a triple dimer (on some plabic graph) that corresponds to $W$ by overlaying a red dimer with boundary condition $I$, a blue dimer with boundary condition $J$, and a green dimer with boundary condition $K$.
\end{remark}

We next present lemmas enumerating the non-elliptic webs corresponding to each term of $A$, $B$, $C$, and $Z$. Each proof involves examining the webs enumerated in Lemmas \ref{lemma:pathlesswebs}, \ref{lemma:webswithpathsadj}, \ref{lemma:webswithpathsnonadj}, and \ref{lemma:pathlesswebs9}, in an attempt to find an edge coloring that satisfies the compatibility conditions. 
Even after some immediate reductions, the lists of webs to be tested are cumulatively quite long, so we defer these proofs to Section \ref{section:appendixlemmas}.

\begin{figure}
    \centering
   \includegraphics[scale=.15]{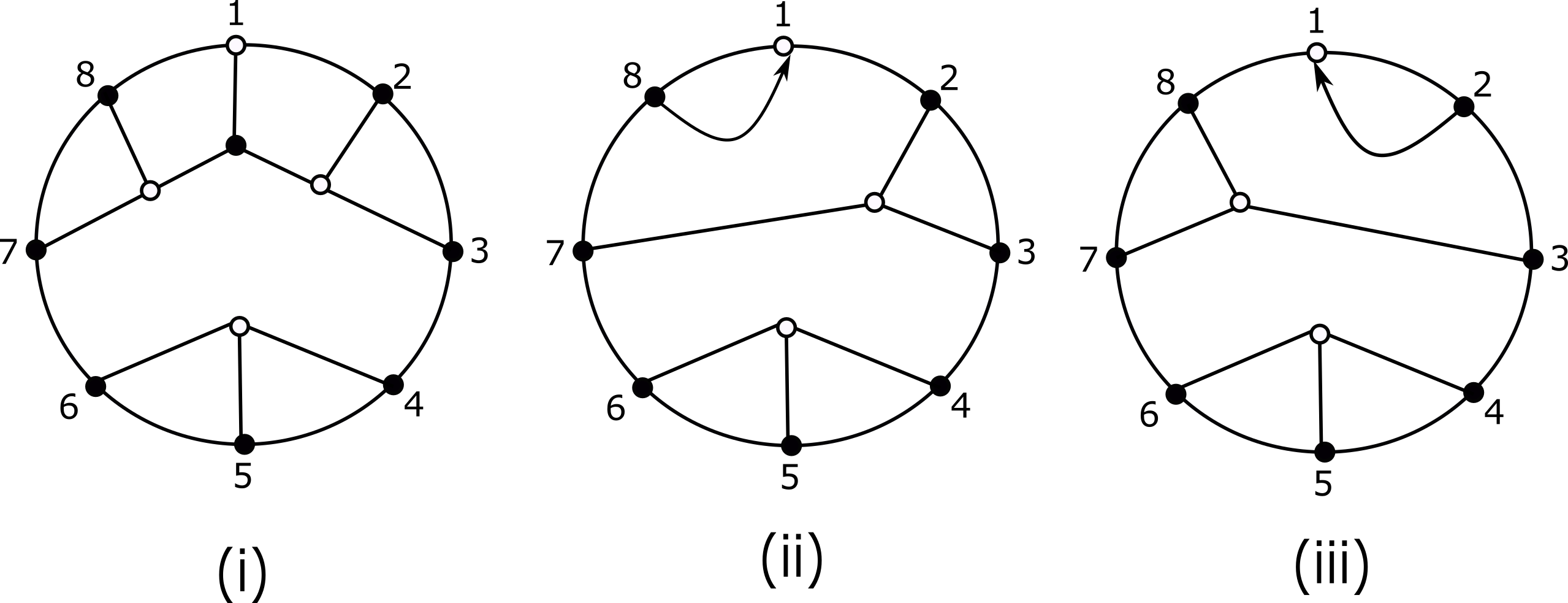}
    \caption{All non-elliptic webs compatible with $A_1$. Web (i) is the batwing, web (ii) is the only non-elliptic web compatible with $A_2$, and web (iii) is the only non-elliptic webs compatible with $A_3$.}
   \label{fig:A123}
   \end{figure}

\begin{restatable}{lemma}{Alemma}\label{lemma:A123}
The non-elliptic webs compatible with $A_1 = \textcolor{red}{\Delta_{134}}\textcolor{blue}{\Delta_{258}}\textcolor{forest}{\Delta_{167}}$ are pictured in Figure \ref{fig:A123}. Web (ii) of Figure \ref{fig:A123} is the only non-elliptic web compatible with $A_2 = \textcolor{red}{\Delta_{134}}\textcolor{blue}{\Delta_{125}}\textcolor{forest}{\Delta_{678}}$, and web (iii) of Figure \ref{fig:A123} is the only non-elliptic web compatible with $A_3 =\textcolor{red}{\Delta_{158}}\textcolor{blue}{\Delta_{234}}\textcolor{forest}{\Delta_{167}}$. Additionally, each compatibility is unique.
\end{restatable}

\begin{figure}
    \centering
    \includegraphics[scale=.15]{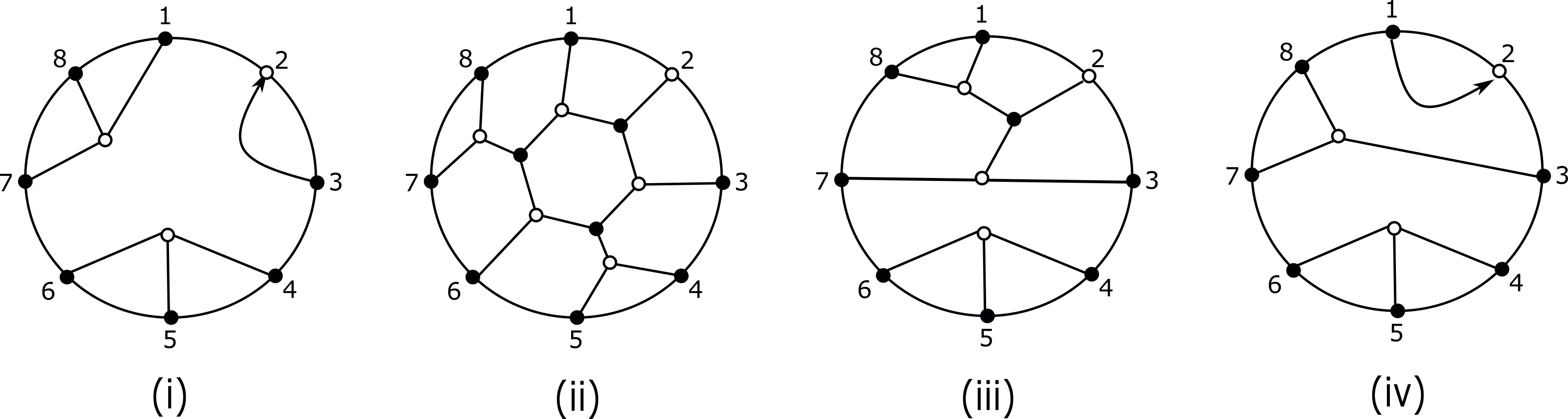}
    \caption{All non-elliptic webs compatible with $B_1$. Web (ii) is the only non-elliptic web compatible with $B_2$, and webs (iii) and (iv) are the only non-elliptic webs compatible with $B_3$.}
    \label{fig:B123}
\end{figure}

\begin{restatable}{lemma}{Blemma}
\label{lemma:B123}
The non-elliptic webs compatible with $B_1 =\textcolor{red}{\Delta_{258}}\textcolor{blue}{\Delta_{134}}\textcolor{forest}{\Delta_{267}}$ are pictured in Figure \ref{fig:B123}. Web (iii) of Figure \ref{fig:B123} is the only non-elliptic web compatible with $B_2 =\textcolor{red}{\Delta_{234}}\textcolor{blue}{\Delta_{128}}\textcolor{forest}{\Delta_{567}}$, and webs (ii) and (iv) of Figure \ref{fig:B123} are the only non-elliptic webs compatible with $B_3 =\textcolor{red}{\Delta_{234}}\textcolor{blue}{\Delta_{258}}\textcolor{forest}{\Delta_{167}}$. Additionally, each compatibility is unique.
\end{restatable}

\begin{figure}
    \centering
    \includegraphics[width=\textwidth]{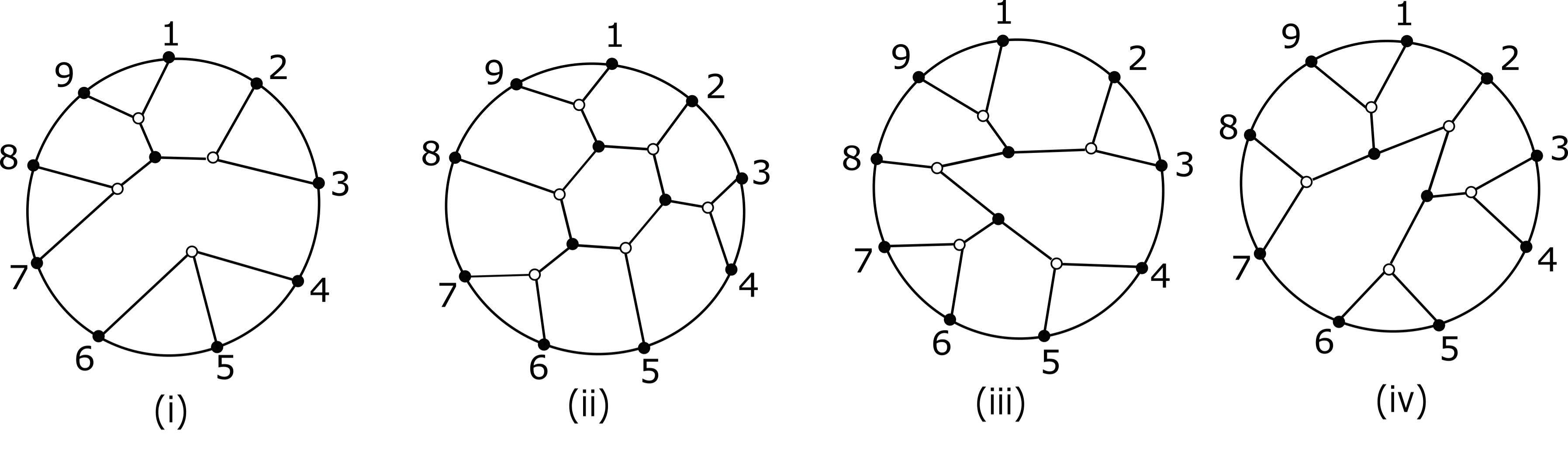}
    \caption{Webs corresponding to the triple products in $C$.}
    \label{fig:C1-6}
\end{figure}
\begin{restatable}{lemma}{Clemma}
\label{lemma:C1-6}
Figure \ref{fig:C1-6} depicts all non-elliptic webs compatible with $C_1 =\textcolor{red}{\Delta_{124}}\textcolor{blue}{\Delta_{357}}\textcolor{forest}{\Delta_{689}}$. Web (ii) is the only non-elliptic web compatible with $C_2 =\textcolor{red}{\Delta_{123}}\textcolor{blue}{\Delta_{456}}\textcolor{forest}{\Delta_{789}}$, webs (ii) and (iii) are the only non-elliptic webs compatible with $C_3 =\textcolor{red}{\Delta_{124}}\textcolor{blue}{\Delta_{356}}\textcolor{forest}{\Delta_{789}}$, and webs (ii) and (iv) are the only non-elliptic webs compatible with $C_4 =\textcolor{red}{\Delta_{123}}\textcolor{blue}{\Delta_{457}}\textcolor{forest}{\Delta_{689}}$.
Additionally, each compatibility is unique.
\end{restatable}

\begin{figure}
    \centering
    \includegraphics[scale=.13]{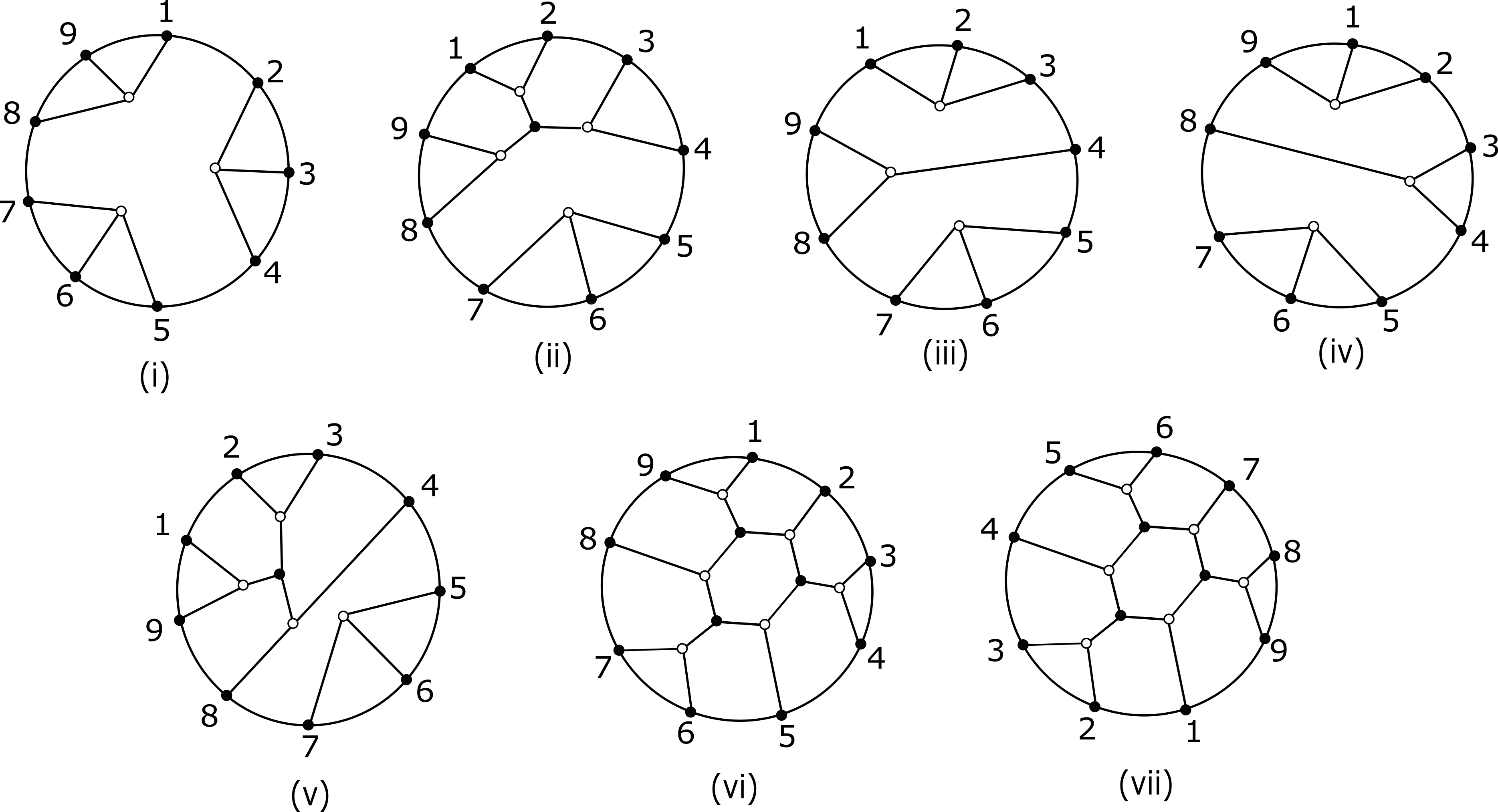}
    \caption{Webs corresponding to the triple products in $Z$.}
    \label{fig:Z1-4}
\end{figure}
\begin{restatable}{lemma}{Zlemma}\label{lemma:Z1-4} Figure \ref{fig:Z1-4} depicts all non-elliptic webs compatible with $Z_1=\textcolor{red}{\Delta_{145}}\textcolor{blue}{\Delta_{278}}\textcolor{forest}{\Delta_{369}}$. Webs (ii) through (v) are the only non-elliptic webs compatible with $Z_2=\textcolor{red}{\Delta_{245}}\textcolor{blue}{\Delta_{178}}\textcolor{forest}{\Delta_{369}}$, web (vi) is the only non-elliptic web compatible with $Z_3=\textcolor{red}{\Delta_{123}}\textcolor{blue}{\Delta_{456}}\textcolor{forest}{\Delta_{789}}$, and web (vii) is the only non-elliptic web compatible with $Z_4=\textcolor{red}{\Delta_{129}}\textcolor{blue}{\Delta_{345}}\textcolor{forest}{\Delta_{678}}$. Additionally, each compatibility is unique.
\end{restatable}

At last, we complete the proofs of our formulas for the twists of $A$, $B$, $C$, and $Z$, as well as their projections and dihedral translates. The arguments are similar at this stage, so we combine them.
\begin{proof}[Proof of Theorems \ref{theorem:connectivityforA}, \ref{theorem:connectivityforB}, \ref{theorem:connectivityforC}, and \ref{theorem:connectivityforZ}]
It follows from Theorem \ref{theorem:facewts} that for any plabic graph $G$ for $\Gr(3,n)$, and for any $I,J,K \subset [n]$ of size 3,
\[\T(\Delta_I\Delta_J\Delta_K) =\T(\Delta_I)\T(\Delta_J)\T(\Delta_K)\]
\[= \left(\sum_{D \in \mathcal{D}_I(G)} \text{wt}_f(D)\right)\left(\sum_{D \in \mathcal{D}_J(G)} \text{wt}_f(D)\right)\left(\sum_{D \in \mathcal{D}_K(G)} \text{wt}_f(D)\right)\]
\[=\sum_{D \in \mathcal{D}_{I,J,K}^3(G)} M_D\text{wt}_f(D),\]
where $\mathcal{D}_{I,J,K}^3(G) \subset \mathcal{D}^3(G)$ is the set of triple dimer configurations $D$ on $G$ formed by overlaying three single dimers with boundary conditions $I$, $J$, and $K$ respectively, and where the multiplicity $M_D$ is the number of triples of single dimers that become $D$ when overlaid.

It follows from \cite[Lemma 4.12]{lam2015} that $\mathcal{D}_{I,J,K}^3(G)$ contains all dimers $D$ such that $W(D)$ is compatible with $\Delta_I\Delta_J\Delta_K$, each with multiplicity
\[M_D=\sum_{W \in \mathcal{W}} C^D_{W} a(I,J,K;W),\]
where $a(I,J,K;W)$ is the number of edge colorings of $W$ compatible with $\Delta_I\Delta_J\Delta_K$, as in Remark \ref{rmk:aIJK}.

Lemma \ref{lemma:A123} asserts that $a(A_1;W)$ is 0 for all nonelliptic webs except those shown in Figure \ref{fig:A123}, for which it is 1; that $a(A_2;W)$ is 0 for all nonelliptic webs except the one shown in Figure \ref{fig:A123}(ii), for which it is 1; and that $a(A_3;W)$ is 0 for all nonelliptic webs except the one shown in Figure \ref{fig:A123}(iii), for which it is 1. It follows that
\[\T(A)=\T(A_1)-\T(A_2)-\T(A_3)=\sum_{D \in \mathcal{D}_{A_1}^3(G)} \text{wt}_f(D)-\sum_{D \in \mathcal{D}_{A_2}^3(G)} \text{wt}_f(D)-\sum_{D \in \mathcal{D}_{A_3}^3(G)} \text{wt}_f(D)\]
\[= \sum_{\substack{D \in \mathcal{D}^3(G)}} \left(\sum_{\substack{W \in \mathcal{W}\\\text{in Figure \ref{fig:A123}}}} C^D_{W}\right) wt_f(D)- \sum_{\substack{D \in \mathcal{D}^3(G)}} \left(\sum_{\substack{W \in \mathcal{W}\\\text{in Figure \ref{fig:A123}(ii)}}} C^D_{W} \right)wt_f(D)\]\[- \sum_{\substack{D \in \mathcal{D}^3(G)}} \left(\sum_{\substack{W \in \mathcal{W}\\\text{in Figure \ref{fig:A123}(iii)}}} C^D_{W}\right) wt_f(D)\]
\[= \sum_{\substack{D \in \mathcal{D}^3(G)}} \sum_{\substack{W \in \mathcal{W}\\\text{in Figure \ref{fig:A123}(i)}}} C^D_{W}wt_f(D) = \sum_{\substack{D \in \mathcal{D}^3(G)}} C^D_{batwing}wt_f(D)\]
since there is only one web, the batwing, in Figure \ref{fig:A123}(i). This is the formula stated in Theorem \ref{theorem:connectivityforA}.

Analogous arguments using Lemmas \ref{lemma:B123}, \ref{lemma:C1-6}, and \ref{lemma:Z1-4} similarly link $B$ to the octopus, $C$ to the hexa-crab, and $Z$ to the tri-crab, completing the specialized statements of each theorem. The generalized statements about dihedral translates and projections follow immediately.
\end{proof}
\end{subsection}

\begin{section}{Comparison to Web Duality} \label{section:duality}
In this section, we describe the results of Theorems \ref{theorem:doubledimersXY}, 
\ref{theorem:connectivityforA}, \ref{theorem:connectivityforB}, \ref{theorem:connectivityforC}, and \ref{theorem:connectivityforZ} in the language of the \textbf{web duality} defined in \cite{fll2019}.
Our results provide rigorous graph-theoretic justification for Observation 8.2 and several of the dualities pictured in Figure 3 of \cite{fll2019}, as well as extending the sense of web duality beyond the case considered in \cite{fll2019} where $k$ divides $n$; these results are depicted in Figure \ref{fig:webdualities}.
Additionally, since our theorems are phrased in terms of face weights rather than the edge weights considered in \cite{fll2019}, they introduce a direct link between web duality and formulas for Laurent expansions of cluster variables.
Finally, we demonstrate calculations with standard Young tableaux that corroborate Observation 8.3 of \cite{fll2019} in the context of our results.

Fix $k<n$, an $r$-dimensional vector space $U$, and a sequence $\lambda=(\lambda_1,\dots,\lambda_n)$ of integers between 1 and $r$ such that $\lambda_1+\lambda_2+\dots+\lambda_n=kr$.
Now $\mathcal{W}_\lambda(U)$ is a certain space of \textbf{tensor invariants} that is in particular spanned by $\text{SL}_r$-webs with $n$ boundary vertices $v_1,\dots,v_n$ such that $v_i$ satisfies a condition determined by $\lambda_i$. We describe some special cases:
\begin{itemize}
    \item If $r=1$, then $\text{SL}_1$-webs are sets of boundary vertices, and $v_i$ should be included in the set if and only if $\lambda_i=1$.
    \item If $r=2$, then $\text{SL}_2$-webs are spanned by non-crossing matchings, and $v_i$ should be included in the matching if and only if $\lambda_i=1$.
    \item If $r=3$, then $\text{SL}_3$-webs are those defined in Section \ref{subsection:webs}. We color $v_i$ black if $\lambda_i=1$ and white if $\lambda_i=2$, and we do not include $v_i$ in the web if $\lambda_i \in \{0,3\}$.
\end{itemize}
We also have a $\mathbb{Z}^n$-grading of $\mathbb{C}[\widehat{\Gr}(k,n)]$ given by the number of times each column appears in a given product of Pl\"ucker coordinates. In particular, $\mathbb{C}[\widehat{\Gr}(k,n)]_\lambda$ consists of linear combinations of $r$-fold products of Pl\"ucker coordinates such that in each product, column $i$ appears $\lambda_i$ times. Note that in \cite{fll2019}, a Pl\"ucker coordinate is shorthand for the sum of dimer weights given by the boundary measurement map of Theorem \ref{theorem:edgewts}. We may identify elements of $\mathbb{C}[\widehat{\Gr}(k,n)]_\lambda$ with webs as in \cite{fominpasha2016}.

Given a plabic graph $G$, the notation $\text{Web}_r(G;\lambda)$ refers to the weighted sum of $r$-weblike subgraphs of $G$ (equivalently, of $r$-fold dimers on $G$) that satisfy the boundary conditions given by $\lambda$. For instance, $\text{Web}_3(G;(2,1,1,1,1,1,1,1)$ would be the weighted sum of all triple dimers $D$ on $G$ such that boundary vertex 1 of $W(D)$ is white and $W(D)$ has 7 boundary black vertices. Finally, we define the \textbf{immanant map} \[\text{Imm}: \mathcal{W}_\lambda(U)^* \rightarrow \mathbb{C}[\widehat{\text{Gr}}(k,n)]_\lambda\]
via
\[\text{Imm}(\varphi)(\tilde{X}(G))=\varphi(\text{Web}_r(G;\lambda))\]
for any edge-weighted plabic graph $\tilde{X}(G)$.
Effectively, the weight of an $r$-fold dimer $D$ in $\text{Web}_r(G;\lambda)$ is included in $\text{Imm}(\varphi)(\tilde{X}(G))$ with multiplicity equal to the value of $\varphi$ on $W(D)$. We say that an $\text{SL}_r$-web $W$ and an $\text{SL}_k$-web $W'$ (viewed as an element of $\mathbb{C}[\widehat{\Gr}(k,n)]$) are \textbf{dual} if the functional $\varphi$ that is 1 on $W$ and 0 on all other independent webs has
$\text{Imm}(\varphi)(\tilde{X}(G))=W'$.

We may now restate Theorems \ref{theorem:doubledimersXY}, \ref{theorem:connectivityforA}, \ref{theorem:connectivityforB}, \ref{theorem:connectivityforC}, and \ref{theorem:connectivityforZ}. For clarity, the following theorem lacks full generality with respect to projections and dihedral translates, but the generalizations should be clear from the original theorem statements. We rely on Proposition \ref{prop:edgetoface} to transition between our face weights and the edge weights of \cite{fll2019}.
\begin{theorem}\label{theorem:dualityrestatement}
Let $G$ be a plabic graph with face weights as in Definition \ref{definition:faceweights}, and assign it edge weights as in Definition \ref{def:MaScEdgeWeights}.
\begin{enumerate}
    \item[(\ref{theorem:doubledimersXY})] Let $U$ be a 2-dimensional vector space, and define $\varphi_{X} \in \mathcal{W}_{[1^6]}(U)^*$ to be 1 on the $\text{SL}_2$ basis web with paths connecting six boundary vertices in pairs $\{1, 6\}, \{2,3\},$ and $\{4,5\}$. Then $\text{Imm}(\varphi_{X})(\tilde{X}(G))=\T(X)$. Similarly, define $\varphi_{Y} \in \mathcal{W}_{[1^6]}(U)^*$ to be 1 on the $\text{SL}_2$ basis web with paths connecting vertices in pairs $\{1, 2\}, \{3,4\},$ and $\{5,6\}$. Then $\text{Imm}(\varphi_{Y})(\tilde{X}(G))=\T(Y)$.
    \item[(\ref{theorem:connectivityforA})] Let $U$ be a 3-dimensional vector space, and define $\varphi_{A} \in \mathcal{W}_{[2,1^7]}(U)^*$ to be the batwing as in Figure \ref{fig:mainresults}. Then $\text{Imm}(\varphi_{A})(\tilde{X}(G))=\T(A)$.
    \item[(\ref{theorem:connectivityforB})] Let $U$ be a 3-dimensional vector space, and define $\varphi_{B} \in \mathcal{W}_{[2,1^7]}(U)^*$ to be the octopus as in Figure \ref{fig:mainresults}. Then $\text{Imm}(\varphi_{B})(\tilde{X}(G))=\T(B)$.
    \item[(\ref{theorem:connectivityforC})] Let $U$ be a 3-dimensional vector space, and define $\varphi_{C} \in \mathcal{W}_{[1^9]}(U)^*$ to be the hexa-crab as in Figure \ref{fig:mainresults}. Then $\text{Imm}(\varphi_{C})(\tilde{X}(G))=\T(C)$.
    \item[(\ref{theorem:connectivityforZ})] Let $U$ be a 3-dimensional vector space, and define $\varphi_{Z} \in \mathcal{W}_{[1^9]}(U)^*$ to be the tri-crab as in Figure \ref{fig:mainresults}. Then $\text{Imm}(\varphi_{Z})(\tilde{X}(G))=\T(Z)$.
\end{enumerate}
\end{theorem}
Note that in particular, removing the stated choice of edge weights and considering $X$, $Y$, $C$, and $Z$ only as arising abstractly from the boundary measurement map recovers several of the dualities in the top and bottom left of Figure 3 of \cite{fll2019} (reproduced in the top left and bottom of Figure \ref{fig:webdualities}). Our work presents the twist map as concretely realizing this duality.

It also adds an interpretation of duality for $A$ and $B$ in $\Gr(3,8)$, which appears pictorially on the top right of Figure \ref{fig:webdualities}. We reference \cite[Figure 22]{fominpasha2016} for the tensor diagrams corresponding to $A$ and $B$, the duals of the batwing and octopus, respectively.
Figure \ref{fig:webdualities} displays elements of $\mathcal{W}_\lambda(U)^*$ on the left, and their dual tensor diagrams in $\mathbb{C}[\widehat{\Gr}(k,n)]$ on the right.\footnote{We only study one direction of duality for $\Gr(3,8)$, since diagrams with white boundary vertices represent more general $SL_3$-invariants than elements of the coordinate ring of the Grassmannian.} Unlike our setting, Fomin and Pylyavskyy \cite{fominpasha2016} notably allow tensor diagrams with boundary vertices of valence higher than one.

Note that an analogue of Observation 8.2 of \cite{fll2019} continues to hold in the listed examples for $\Gr(3,8)$: the dual of a non-elliptic web arises from clasping boundary vertices of another non-elliptic web. It would be interesting to check this statement for the other non-elliptic webs depicted in Lemmas \ref{lemma:pathlesswebs}, \ref{lemma:webswithpathsadj}, and \ref{lemma:webswithpathsnonadj} using the methods we will describe in Section \ref{section:appendixlemmas}.

\begin{figure}
    \centering
    \includegraphics[scale=.1]{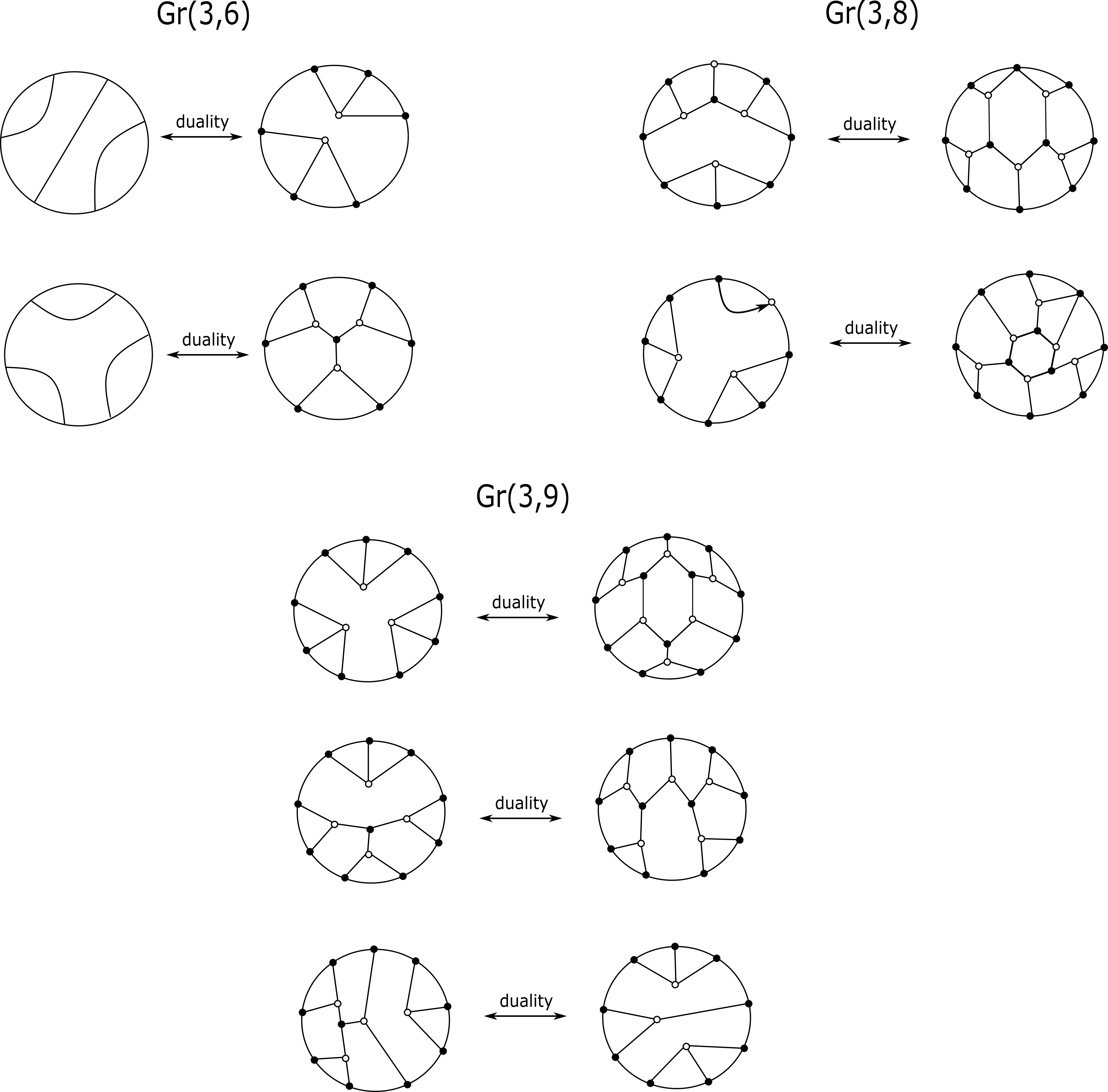}
    \caption{Duality between webs in small cases. Note that the top-left and bottom of the figure depicts webs shown in \cite[Figure 3]{fll2019}; the top-right is new.}
    \label{fig:webdualities}
\end{figure}

\begin{subsection}{Young Tableaux} \label{section:young}
We provide some explicit computations with tableaux associated to \cite[Observation 8.3]{fll2019}, which notes that in small cases, the duality map aligns with a combination of the transpose map applied to a rectangular standard Young tableau and the Khovanov-Kuperberg bijection \cite{KhovKup} between two-row or three-row standard Young tableaux and non-crossing matchings or non-elliptic webs, respectively.

We begin with $\Gr(3,6)$. In what follows, for every standard Young tableau of rectangular shape $[3,3]$, we demonstrate the effect of the Khovanov-Kuperberg bijection: we read the top row from left to right, and match each entry $j$ to the largest unmatched value $i$ on the row below such that $i < j$.

\begin{center}
    $\begin{ytableau}
 4 & 5 & 6 \\ 1 & 2 & 3
 \end{ytableau} \leftrightarrow (3,4), (2,5), (1,6)$,

\vspace{1em}

$\begin{ytableau}
 2 & 5 & 6 \\ 1 & 3 & 4
 \end{ytableau}\leftrightarrow (1,2), (4,5), (3,6) $,
 
 \vspace{1em}

$\begin{ytableau}
 3 & 4 & 6 \\ 1 & 2 & 5
 \end{ytableau}\leftrightarrow (2,3), (1,4), (5,6)$,
 \vspace{1em}

 $\begin{ytableau}
 3 & 5 & 6 \\ 1 & 2 & 4 
 \end{ytableau} \leftrightarrow (2,3), (4,5), (1,6)$, ~~~and~~~
 
 \vspace{1em}

$\begin{ytableau}
 2 & 4 & 6  \\ 1 & 3 & 5
 \end{ytableau} \leftrightarrow (1,2), (3,4), (5,6)$.
\FloatBarrier

\end{center}
Additionally, as computed in \cite{tymo}, the transposes of these five standard Young tableaux (of shape $[2,2,2]$) biject to non-elliptic webs, with 
\begin{center}
$\begin{ytableau}
 3 & 6 \\ 2 & 5 \\ 1 & 4
 \end{ytableau}$ corresponding to the union of the two tripods $[1,2,3]$ and $[4,5,6]$,

\vspace{1em}

$\begin{ytableau}
 4 & 6 \\ 3 & 5 \\ 1 & 2
 \end{ytableau}$ corresponding to the union of the two tripods $[1,5,6]$ and $[2,3,4]$, and

\vspace{1em}

$\begin{ytableau}
 5 & 6 \\ 2 & 4 \\ 1 & 3
 \end{ytableau}$ corresponding to the union of the two tripods $[1,2,6]$ and $[3,4,5]$,
\vspace{1em}
\end{center}
while 
 $\begin{ytableau}
 4 & 6 \\ 2 & 5 \\ 1 & 3
 \end{ytableau}$ and  
  $\begin{ytableau}
 5 & 6 \\ 3 & 4 \\ 1 & 2
 \end{ytableau}$ correspond to two rotations of a hexapod (the second entry in the right column at the top left of Figure \ref{fig:webdualities}), which encode the compound determinants $X =\det \bigg( v_1 \times v_2 \quad v_3 \times v_4 \quad v_5 \times v_6\bigg)$ and $Y=\det \bigg( v_6 \times v_1 \quad v_2 \times v_3 \quad v_4 \times v_5\bigg)$, respectively.
We now observe that, as shown in Figure \ref{fig:webdualities}, duality links a non-crossing matching to the non-elliptic web corresponding to the transpose of its standard Young tableau:
\[(3,4),(2,5),(1,6) \leftrightarrow \text{the union of the two tripods $[1,2,3]$ and $[4,5,6]$}\]
\[(1,2),(4,5),(3,6) \leftrightarrow \text{the union of the two tripods $[1,5,6]$ and $[2,3,4]$}\]
\[(2,3),(1,4),(5,6) \leftrightarrow \text{the union of the two tripods $[1,2,6]$ and $[3,4,5]$}\]
\[(2,3),(4,5),(1,6) \leftrightarrow \text{the hexapod corresponding to }X\]
\[(1,2),(3,4),(5,6) \leftrightarrow \text{the hexapod corresponding to }Y.\]

We perform the analogous computations for the dualities in $\Gr(3,9)$ given by twisting the degree three cluster algebra elements $C$ and $Z$. Under the Khovanov-Kuperberg bijection, the tensor diagram for $C$ (the second entry in the right column at the bottom of Figure \ref{fig:webdualities}) corresponds to the standard Young tableau 
 $$T_C = \begin{ytableau}
6 & 8 & 9 \\ 3& 5& 7 \\ 1 & 2 & 4
 \end{ytableau}.$$
The web corresponding to its transpose
 $$T_C^* = \begin{ytableau}
 4 & 7 & 9 \\ 2 & 5 & 8 \\ 1 & 3 & 6
 \end{ytableau}$$ 
 is the hexa-crab depicted in Figure \ref{fig:mainresults}, which aligns with the duality depicted in the middle row in the bottom of Figure \ref{fig:webdualities}.
Additionally, Example 8.1 and Remark 8.2 of \cite{quantum20} verify that the tensor diagram for $Z$ (the first entry in the right column on the bottom of Figure \ref{fig:webdualities}) corresponds to the standard Young tableau
 $$T_Z = \begin{ytableau}
5 & 8 & 9 \\ 2& 6 & 7 \\ 1 & 3 & 4
 \end{ytableau}.$$
The web corresponding to its transpose
 $$T_Z^* = \begin{ytableau}
 4 & 7 & 9 \\ 3 & 6 & 8 \\ 1 & 2 & 5
 \end{ytableau}$$
is in fact the three tripods $(1,8,9), (2,3,4), (5,6,7)$, i.e. the fourth web illustrated in Figure \ref{fig:pathlesswebs9}, the tri-crab. Again, this result is consistent with the duality depicted in the top row in the bottom left quadrant of \cite[Figure 3]{fll2019}.

It would be interesting to extend these computations to our additional dualities in $\Gr(3,8)$, depicted in Figure \ref{fig:webdualities}.
We expect that they relate to {\bf semi-standard Young tableaux} of shape $[3,3,3]$, using entries only involving $1,2,\dots, 8$; conjecturally, these would be the 24 tableaux provided in \cite[Section 3.1]{clustering} for $\Gr(3,8)$.

\end{subsection}
\end{section}

\begin{section}{Construction of $C$}\label{section:constructionofC}
In this section, we justify our claim from Section \ref{subsection:differences} that the element
\[C=\Delta_{124}\Delta_{357}\Delta_{689} + \Delta_{123}\Delta_{456}\Delta_{789} - \Delta_{124}\Delta_{356}\Delta_{789}-\Delta_{123}\Delta_{457}\Delta_{689} \]
of $\mathbb{C}[\Gr(3,9)]$ is a cluster variable. It will suffice to show that the above expression corresponds to the tensor diagram on the right of the middle row in the bottom left of Figure \ref{fig:webdualities}, since that diagram is a planar tree, and therefore corresponds to a cluster variable by \cite[Corollary 8.10]{fominpasha2016}.

We explained in Section \ref{section:duality} that under the Khovanov-Kuperberg bijection, the cluster algebra element $C$ and its corresponding web invariant $[W_C]$ correspond to a particular standard Young tableau, namely
 $$T_C = \begin{ytableau}
6 & 8 & 9 \\ 3& 5& 7 \\ 1 & 2 & 4
 \end{ytableau}.$$
We algebraically express $[W_C]$ as a polynomial in Pl\"ucker coordinates by producing a tensor diagram from the rows of $T_C$ and then resolving crossings to create a summation in terms of planar webs. This process is illustrated in Figure \ref{fig:computationforC}: we consider the triple product $\Delta_{124}\Delta_{357}\Delta_{689}$ formed from the rows\footnote{On further review of this process, as we apply the Khovanov-Kuperberg bijection described in Section \ref{section:young} to the transpose $T_C^*$, the first step yields the triples $(6,8,9),~(3,5,7),~(1,2,4)$, which are natural to utilize to obtain a leading monomial for an expression of a web invariant in terms of Pl\"ucker coordinates based on the theory of non-crossing tableaux, see \cite{pylyavskyy2009non, PPS}.  However, it is only a coincidence rather than a more general phenomenon that these triples agree with the three rows of $T_C$.} of $T_C$, construct a corresponding tensor diagram by superimposing the three associated tripods, and apply the $SL_3$-web Skein relation shown in Figure \ref{fig:tensordiagramskein} to resolve two of the crossings that appear.

\begin{figure}[H]
\centering

\tikzset{every picture/.style={line width=0.75pt}} 

\begin{tikzpicture}[x=0.75pt,y=0.75pt,yscale=-.75,xscale=.75]

\draw    (108,200) -- (200.5,260) ;
\draw    (108.5,261) -- (200.5,201) ;
\draw  [fill={rgb, 255:red, 255; green, 255; blue, 255 }  ,fill opacity=0.98 ] (103.75,200) .. controls (103.75,197.65) and (105.65,195.75) .. (108,195.75) .. controls (110.35,195.75) and (112.25,197.65) .. (112.25,200) .. controls (112.25,202.35) and (110.35,204.25) .. (108,204.25) .. controls (105.65,204.25) and (103.75,202.35) .. (103.75,200) -- cycle ;
\draw  [fill={rgb, 255:red, 255; green, 255; blue, 255 }  ,fill opacity=0.98 ] (104.25,261) .. controls (104.25,258.65) and (106.15,256.75) .. (108.5,256.75) .. controls (110.85,256.75) and (112.75,258.65) .. (112.75,261) .. controls (112.75,263.35) and (110.85,265.25) .. (108.5,265.25) .. controls (106.15,265.25) and (104.25,263.35) .. (104.25,261) -- cycle ;
\draw  [fill={rgb, 255:red, 13; green, 0; blue, 0 }  ,fill opacity=0.98 ] (196.25,201) .. controls (196.25,198.65) and (198.15,196.75) .. (200.5,196.75) .. controls (202.85,196.75) and (204.75,198.65) .. (204.75,201) .. controls (204.75,203.35) and (202.85,205.25) .. (200.5,205.25) .. controls (198.15,205.25) and (196.25,203.35) .. (196.25,201) -- cycle ;
\draw  [fill={rgb, 255:red, 13; green, 0; blue, 0 }  ,fill opacity=0.98 ] (196.25,260) .. controls (196.25,257.65) and (198.15,255.75) .. (200.5,255.75) .. controls (202.85,255.75) and (204.75,257.65) .. (204.75,260) .. controls (204.75,262.35) and (202.85,264.25) .. (200.5,264.25) .. controls (198.15,264.25) and (196.25,262.35) .. (196.25,260) -- cycle ;
\draw    (269,201) -- (380.5,201) ;
\draw    (270,260) -- (380.5,260) ;
\draw  [fill={rgb, 255:red, 255; green, 255; blue, 255 }  ,fill opacity=0.98 ] (264.75,201) .. controls (264.75,198.65) and (266.65,196.75) .. (269,196.75) .. controls (271.35,196.75) and (273.25,198.65) .. (273.25,201) .. controls (273.25,203.35) and (271.35,205.25) .. (269,205.25) .. controls (266.65,205.25) and (264.75,203.35) .. (264.75,201) -- cycle ;
\draw  [fill={rgb, 255:red, 255; green, 255; blue, 255 }  ,fill opacity=0.98 ] (265.75,260) .. controls (265.75,257.65) and (267.65,255.75) .. (270,255.75) .. controls (272.35,255.75) and (274.25,257.65) .. (274.25,260) .. controls (274.25,262.35) and (272.35,264.25) .. (270,264.25) .. controls (267.65,264.25) and (265.75,262.35) .. (265.75,260) -- cycle ;
\draw  [fill={rgb, 255:red, 13; green, 0; blue, 0 }  ,fill opacity=0.98 ] (376.25,201) .. controls (376.25,198.65) and (378.15,196.75) .. (380.5,196.75) .. controls (382.85,196.75) and (384.75,198.65) .. (384.75,201) .. controls (384.75,203.35) and (382.85,205.25) .. (380.5,205.25) .. controls (378.15,205.25) and (376.25,203.35) .. (376.25,201) -- cycle ;
\draw  [fill={rgb, 255:red, 13; green, 0; blue, 0 }  ,fill opacity=0.98 ] (376.25,260) .. controls (376.25,257.65) and (378.15,255.75) .. (380.5,255.75) .. controls (382.85,255.75) and (384.75,257.65) .. (384.75,260) .. controls (384.75,262.35) and (382.85,264.25) .. (380.5,264.25) .. controls (378.15,264.25) and (376.25,262.35) .. (376.25,260) -- cycle ;
\draw    (429.5,199) -- (469.5,232) ;
\draw    (429.5,261) -- (469.5,232) ;
\draw    (540.5,232) -- (581.5,264) ;
\draw    (578.5,198) -- (540.5,232) ;
\draw    (469.5,232) -- (540.5,232) ;
\draw  [fill={rgb, 255:red, 255; green, 255; blue, 255 }  ,fill opacity=0.98 ] (425.25,199) .. controls (425.25,196.65) and (427.15,194.75) .. (429.5,194.75) .. controls (431.85,194.75) and (433.75,196.65) .. (433.75,199) .. controls (433.75,201.35) and (431.85,203.25) .. (429.5,203.25) .. controls (427.15,203.25) and (425.25,201.35) .. (425.25,199) -- cycle ;
\draw  [fill={rgb, 255:red, 255; green, 255; blue, 255 }  ,fill opacity=0.98 ] (425.25,261) .. controls (425.25,258.65) and (427.15,256.75) .. (429.5,256.75) .. controls (431.85,256.75) and (433.75,258.65) .. (433.75,261) .. controls (433.75,263.35) and (431.85,265.25) .. (429.5,265.25) .. controls (427.15,265.25) and (425.25,263.35) .. (425.25,261) -- cycle ;
\draw  [fill={rgb, 255:red, 255; green, 255; blue, 255 }  ,fill opacity=0.98 ] (536.25,232) .. controls (536.25,229.65) and (538.15,227.75) .. (540.5,227.75) .. controls (542.85,227.75) and (544.75,229.65) .. (544.75,232) .. controls (544.75,234.35) and (542.85,236.25) .. (540.5,236.25) .. controls (538.15,236.25) and (536.25,234.35) .. (536.25,232) -- cycle ;
\draw  [fill={rgb, 255:red, 13; green, 0; blue, 0 }  ,fill opacity=0.98 ] (465.25,232) .. controls (465.25,229.65) and (467.15,227.75) .. (469.5,227.75) .. controls (471.85,227.75) and (473.75,229.65) .. (473.75,232) .. controls (473.75,234.35) and (471.85,236.25) .. (469.5,236.25) .. controls (467.15,236.25) and (465.25,234.35) .. (465.25,232) -- cycle ;
\draw  [fill={rgb, 255:red, 13; green, 0; blue, 0 }  ,fill opacity=0.98 ] (574.25,198) .. controls (574.25,195.65) and (576.15,193.75) .. (578.5,193.75) .. controls (580.85,193.75) and (582.75,195.65) .. (582.75,198) .. controls (582.75,200.35) and (580.85,202.25) .. (578.5,202.25) .. controls (576.15,202.25) and (574.25,200.35) .. (574.25,198) -- cycle ;
\draw  [fill={rgb, 255:red, 13; green, 0; blue, 0 }  ,fill opacity=0.98 ] (577.25,264) .. controls (577.25,261.65) and (579.15,259.75) .. (581.5,259.75) .. controls (583.85,259.75) and (585.75,261.65) .. (585.75,264) .. controls (585.75,266.35) and (583.85,268.25) .. (581.5,268.25) .. controls (579.15,268.25) and (577.25,266.35) .. (577.25,264) -- cycle ;

\draw (223,222.4) node [anchor=north west][inner sep=0.75pt]    {$=$};
\draw (401,221.4) node [anchor=north west][inner sep=0.75pt]    {$+$};

\end{tikzpicture}
\caption{Additional Skein relation for tensor diagrams, which are not necessarily planar.}
\label{fig:tensordiagramskein}
\end{figure}
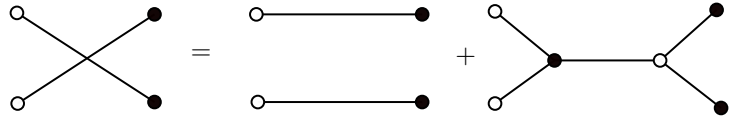

 \begin{figure}[H]
     \centering
     \includegraphics[scale=.15]{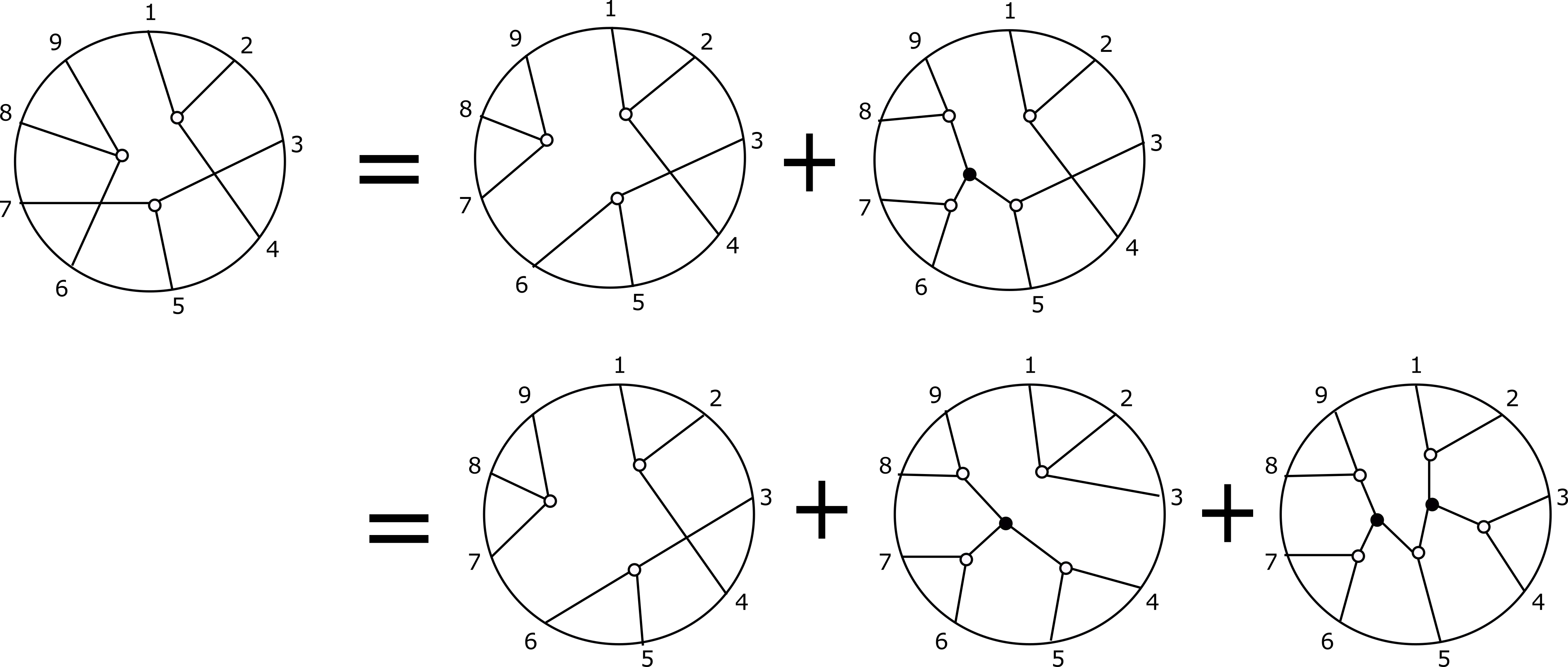}
     \caption{Applying the Skein relation of Figure \ref{fig:tensordiagramskein} to the tensor diagram for $\Delta_{124}\Delta_{357}\Delta_{689}$ to obtain the tensor diagram for $C$.}
     \label{fig:computationforC}
 \end{figure}
 
To translate the tensor diagrams in Figure \ref{fig:computationforC} into an expression in Pl\"ucker coordinates, we note that the tensor diagram on the left side of the equation consists of three tripods, and therefore corresponds to a product of three Pl\"ucker cluster variables, in particular $\Delta_{124}\Delta_{357}\Delta_{689}$. The first summand in the bottom line of Figure \ref{fig:computationforC} similarly corresponds to $\Delta_{124}\Delta_{356}\Delta_{789}$. The next summand has two components: a tripod, which corresponds to the Pl{\"u}cker coordinate $\Delta_{123}$; and a hexapod, which corresponds to the compound determinant $\det(v_4 \times v_5, v_6 \times v_7, v_8 \times v_9)$, i.e. the quadratic difference $X^{456789}=\Delta_{457}\Delta_{689} - \Delta_{456}\Delta_{789}$. Therefore, we may express this summand as $\Delta_{123}(\Delta_{457}\Delta_{689} - \Delta_{456}\Delta_{789})$. Finally, the rightmost summand at the bottom of Figure \ref{fig:computationforC} is the desired web $[W_C]$. We have therefore recovered the relation
\[\Delta_{124}\Delta_{357}\Delta_{689} = \Delta_{124}\Delta_{356}\Delta_{789} +\Delta_{123}(\Delta_{457}\Delta_{689} - \Delta_{456}\Delta_{789})+C,\]
and our expression for $C$ follows. 

We note that similar tensor diagram manipulations are sufficient to calculate expansions of $A$, $B$, and $Z$; however, such expansions are already present in the literature.
\end{section}

\begin{section}{Appendix: Proofs of Lemmas}\label{section:appendixlemmas}
In this Appendix, we prove the lemmas from Section \ref{section:tripledimers}, restating them for convenience.

\Alemma*

\begin{proof}[Proof of Lemma \ref{lemma:A123}]
In Figures \ref{fig:A1proof}, \ref{fig:A2proof}, and \ref{fig:A3proof}, we enumerate all nonelliptic webs with boundary vertices $\{v_1,\dots,v_n\}$ such that $v_1$ is colored white and $v_2,\dots,v_n$ are colored black. Those compatible with $A_1$, $A_2$, or $A_3$ respectively are fully colored and boxed; for all others, the impossibility of a proper edge coloring with the corresponding boundary edge colors is demonstrated. Note that no web compatible with $A_2$ can contain a path $v_2 \to v_1$ or $v_5 \to v_1$, since the edges incident to $v_2$ and $v_5$ cannot be the same color as the edge incident to $v_1$; similarly, no web compatible with $A_3$ can contain a path $v_8 \to v_1$ or $v_5 \to v_1$. We therefore omit such webs in the corresponding figures. Also note that any completed proper coloring is unique given the compatibility conditions.

\begin{figure}
    \centering
    \includegraphics[width=\textwidth]{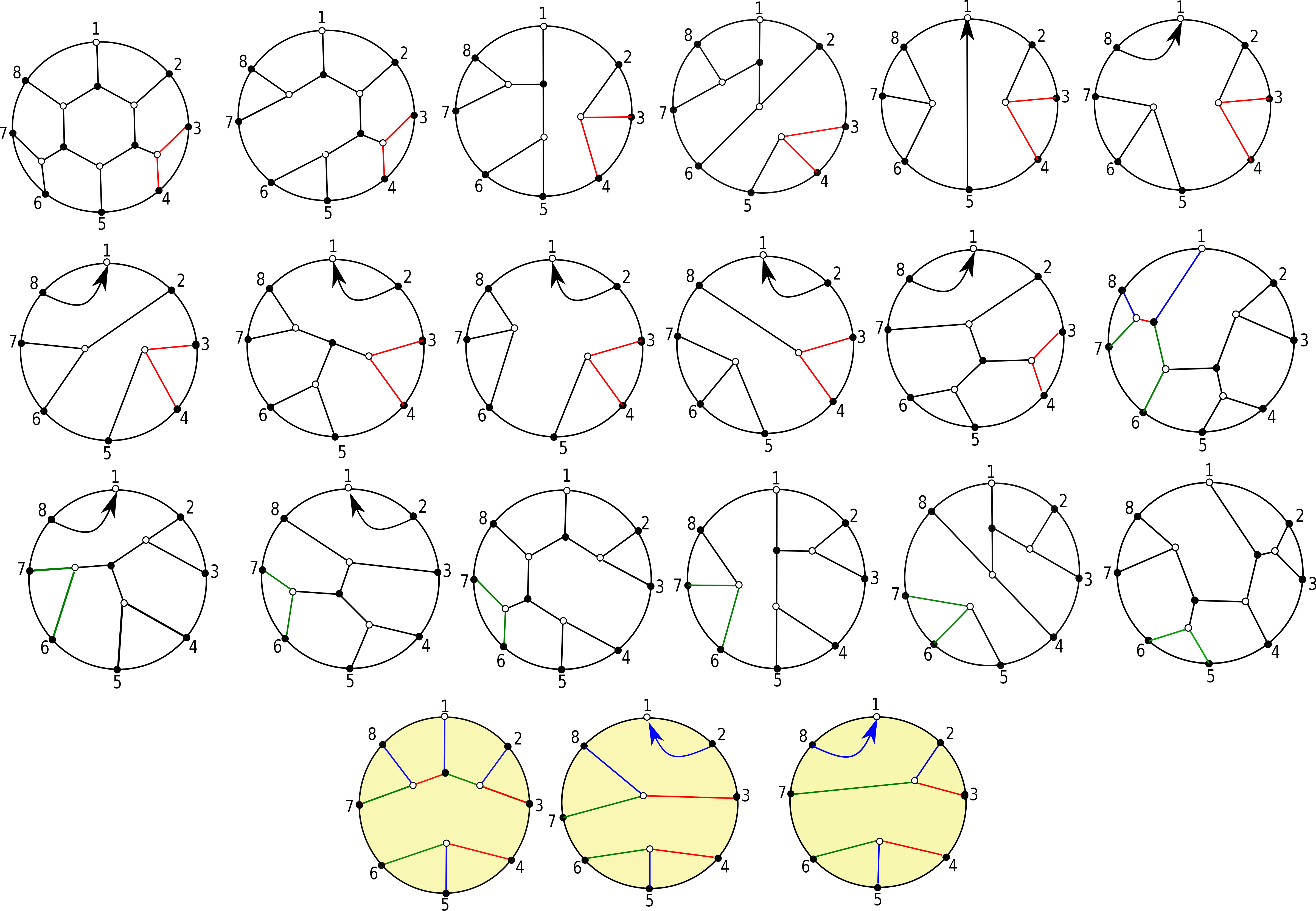}
    \caption{Enumeration of non-elliptic webs compatible (highlighted in yellow for emphasis) and incompatible with $A_1=(134)(258)(167)$.}
    \label{fig:A1proof}
\end{figure}
\begin{figure}
    \centering
    \includegraphics[width=\textwidth]{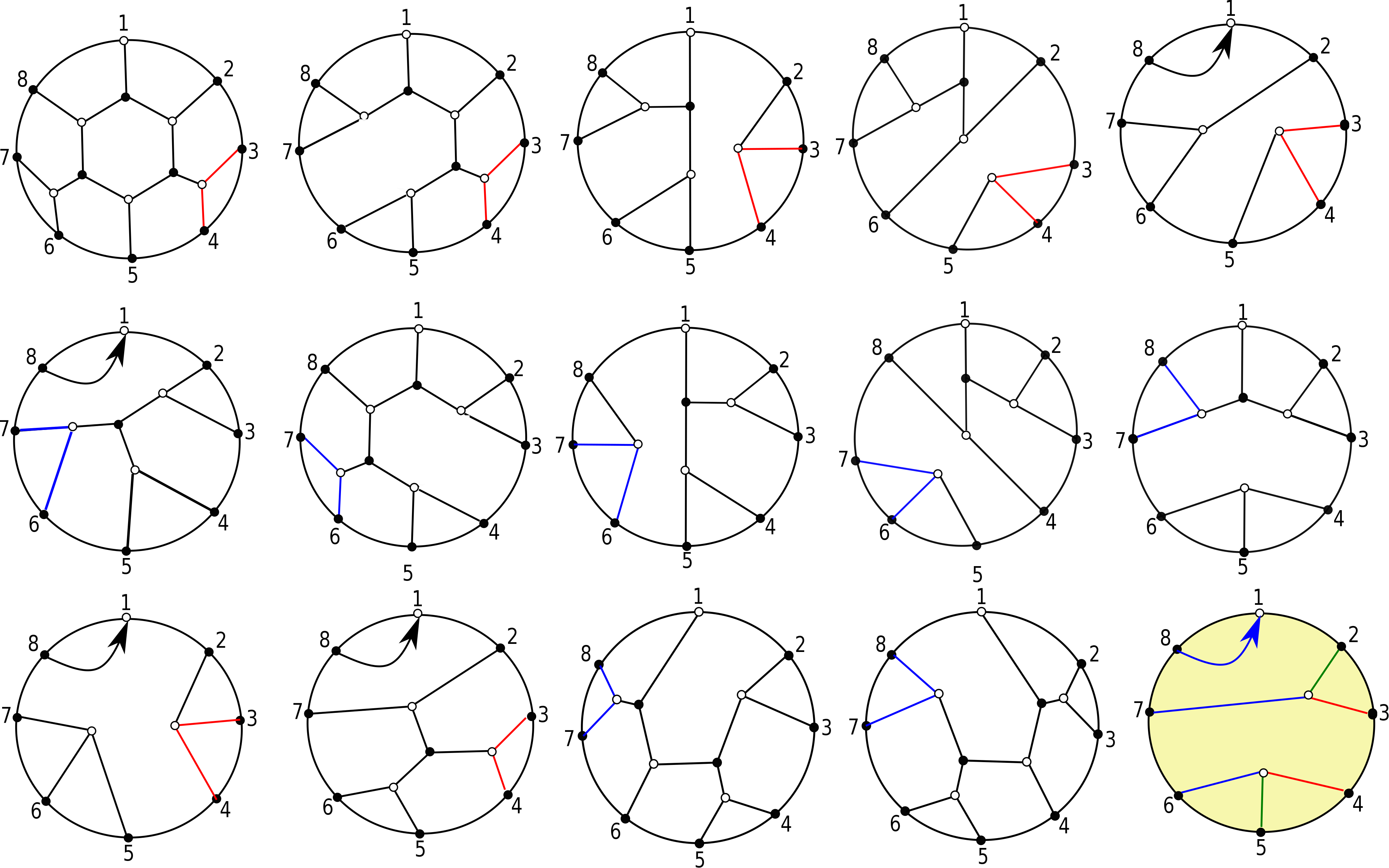}
    \caption{Enumeration of non-elliptic webs compatible (highlighted in yellow for emphasis) and incompatible with $A_2=(134)(125)(678)$.}
    \label{fig:A2proof}
\end{figure}
\begin{figure}
    \centering
    \includegraphics[width=\textwidth]{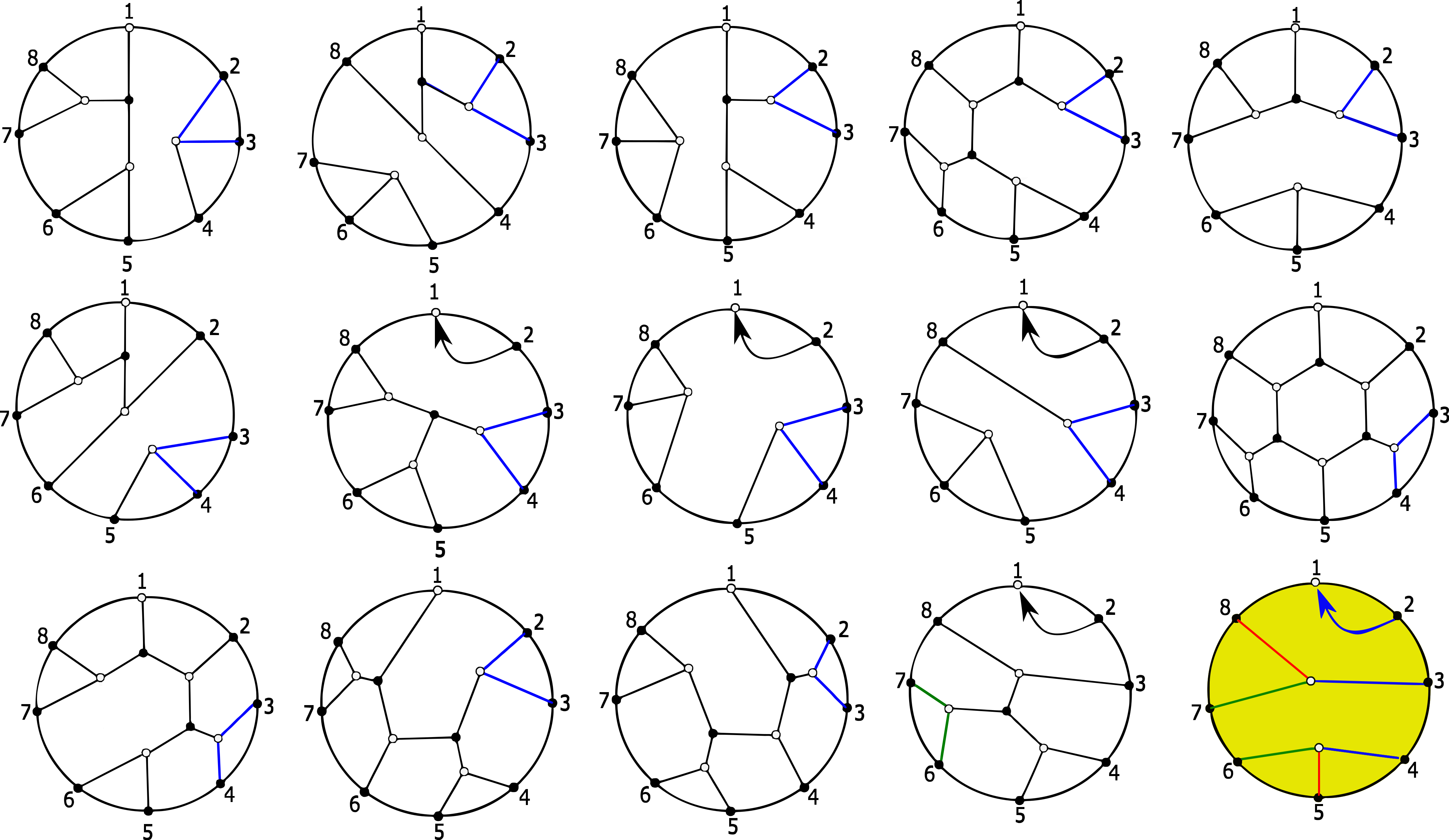}
    \caption{Enumeration of non-elliptic webs compatible (highlighted in yellow for emphasis) and incompatible with $A_3=(158)(234)(167)$.}
    \label{fig:A3proof}
\end{figure}
\end{proof}

\Blemma*

\begin{proof}[Proof of Lemma \ref{lemma:B123}]
Similarly to the previous proof, for each term of $B$, we enumerate all nonelliptic webs with boundary vertices $\{v_1,\dots,v_n\}$ such that $v_2$ is colored white and $v_1,v_3,\dots,v_n$ are colored black. Note that no web compatible with $B_1$ can contain a path $v_6 \to v_2$, no web compatible with $B_2$ can contain any path, and no web compatible with $B_3$ can contain a path $v_3 \to v_2$, so we omit these webs in the corresponding figures. Also note that any completed proper coloring is unique given the compatibility conditions.
\begin{figure}
    \centering
    \includegraphics[width=\textwidth]{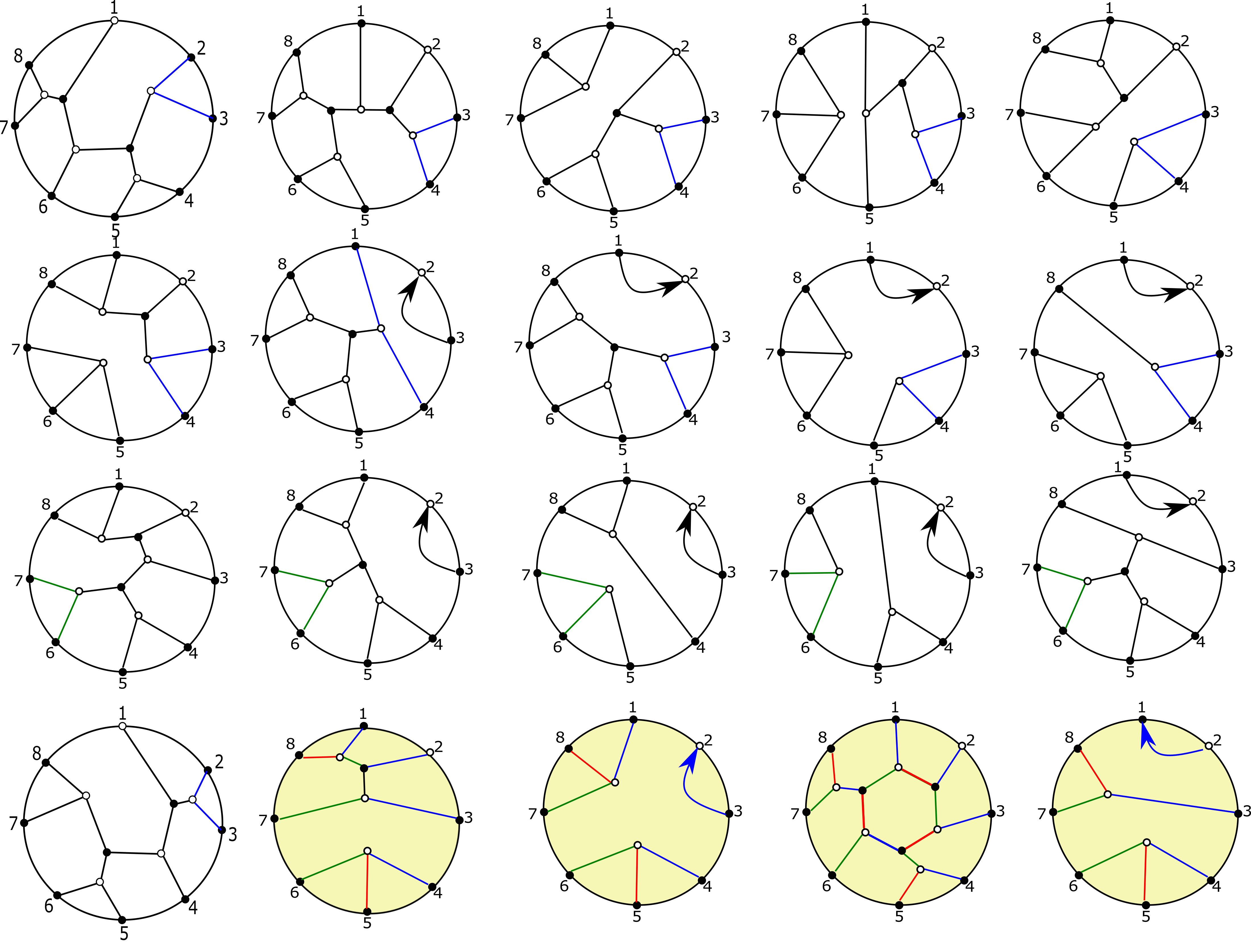}
    \caption{Enumeration of non-elliptic webs compatible (highlighted in yellow for emphasis) and incompatible with $B_1=(258)(134)(267)$.}
    \label{fig:B1proof}
\end{figure}
\begin{figure}
    \centering
    \includegraphics[width=\textwidth]{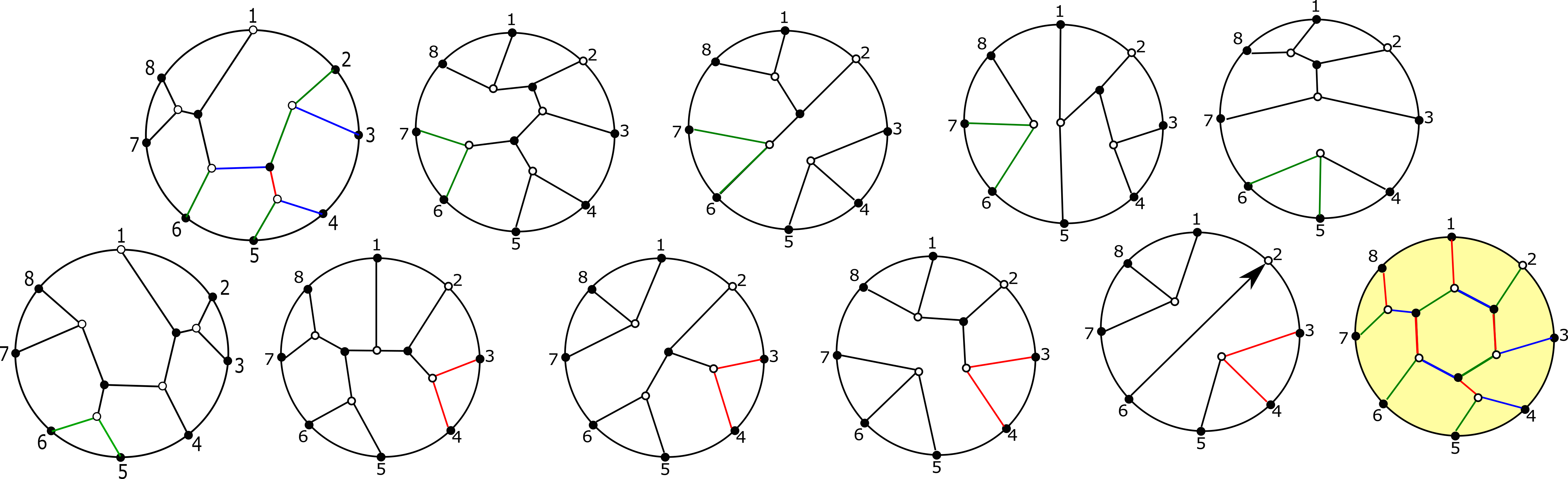}
    \caption{Enumeration of non-elliptic webs compatible (highlighted in yellow for emphasis) and incompatible with $B_2=(234)(128)(567)$.}
    \label{fig:B2proof}
\end{figure}
\begin{figure}
    \centering
    \includegraphics[width=\textwidth]{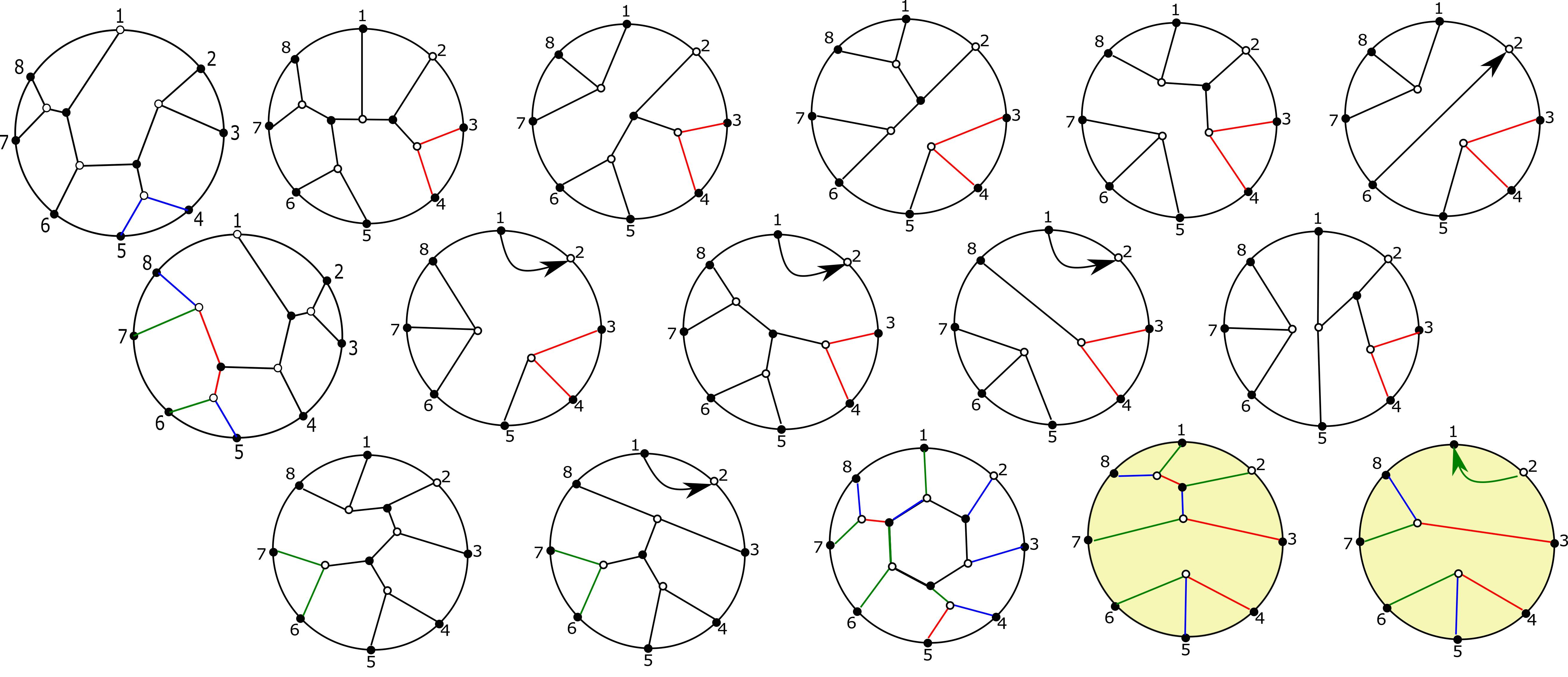}
    \caption{Enumeration of non-elliptic webs compatible (highlighted in yellow for emphasis) and incompatible with $B_3=(234)(258)(167)$.  
    Note that the middle web of the third row is incompatible due to a collision on an internal vertex of the hexagon, where two \textcolor{forest}{green} edges are forced to meet. The rightmost two webs of the third row give the only webs compatible with the boundary condition $B_3$.}
    \label{fig:B3proof}
\end{figure}
\end{proof}

We use the following lemma to limit our lists of webs for terms of $C$.
\begin{lemma}\label{lemma:frozenwebs}
Let $I=\{i-1, i,i+1\}$ mod 9, and $J,K \in \binom{[9]\setminus I}{3}$.
If a non-elliptic web is compatible with $\Delta_I\Delta_J\Delta_K$, it must be one of the webs pictured in Figure \ref{fig:frozenwebs}. 
\begin{figure}
    \centering
    \includegraphics[width=\textwidth]{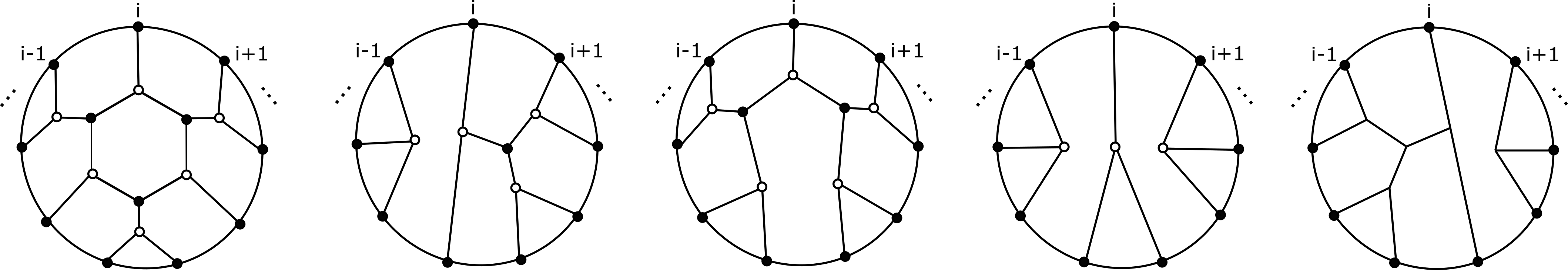}
    \caption{All non-elliptic webs compatible with some $\Delta_I\Delta_J\Delta_K$, where $I=\{i-1,i,i+1\}$ mod 9 and $J,K \in \binom{[9]\setminus I}{3}$.}
    \label{fig:frozenwebs}
\end{figure}
\end{lemma}
\begin{proof}
In Figure \ref{fig:frozenwebspf}, we enumerate all nonelliptic webs listed in Lemma \ref{lemma:pathlesswebs9}, drawing vertex $i$ at the top of each web without loss of generality, and eliminate any that do not admit a proper coloring such that the edges adjacent to vertices $i-1$, $i$, and $i+1$ are the same color.
\begin{figure}
    \centering
    \includegraphics[width=\textwidth]{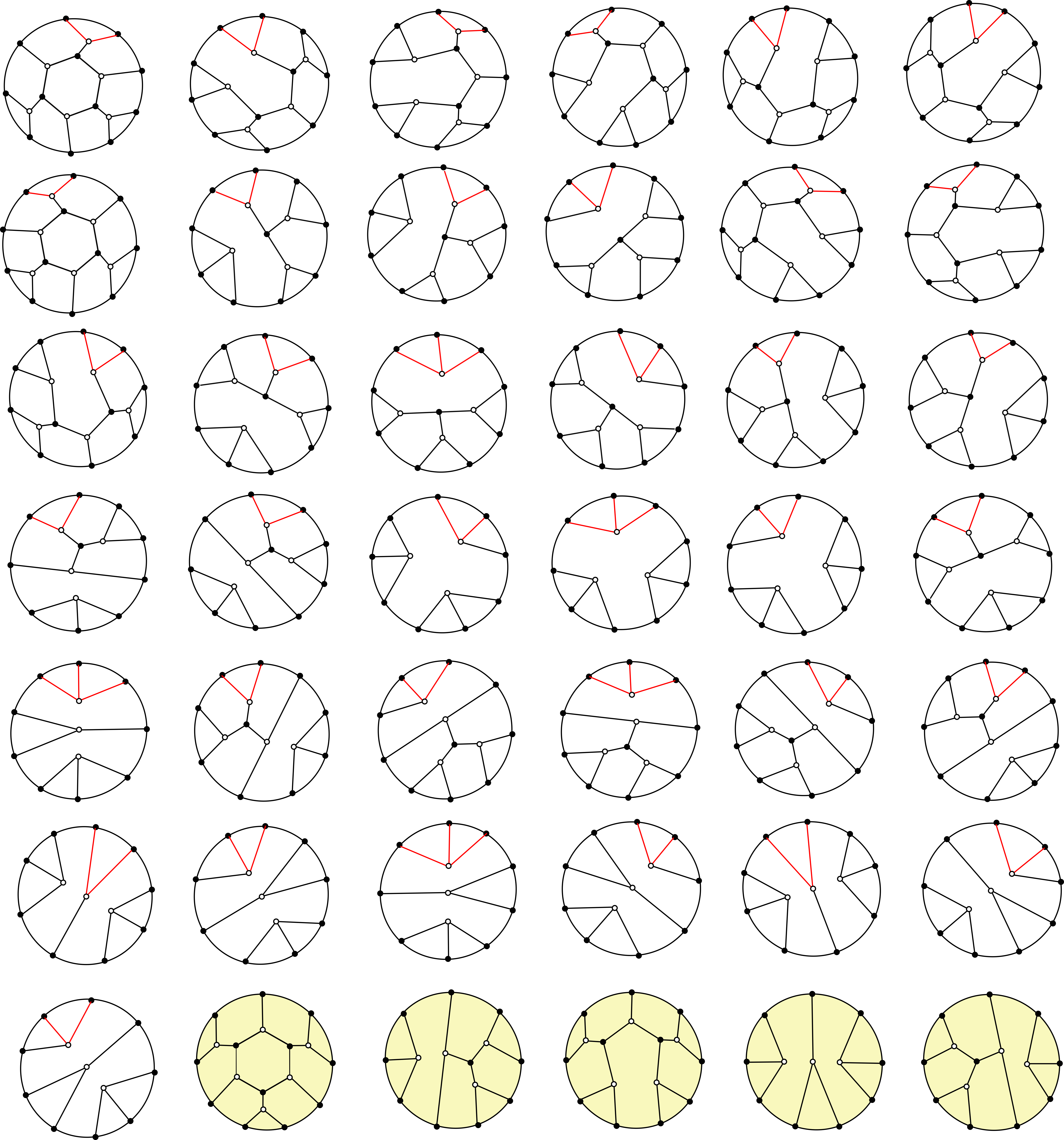}
    \caption{Enumeration of non-elliptic webs compatible (highlighted in yellow for emphasis) and incompatible with some $\Delta_I\Delta_J\Delta_K$, where $I=\{i-1,i,i+1\}$ mod 9 with vertex $i$ drawn at the top of each web, and $J,K \in \binom{[9]\setminus I}{3}$.}
    \label{fig:frozenwebspf}
\end{figure}
\end{proof}

We now prove Lemmas \ref{lemma:C1-6} and \ref{lemma:Z1-4}.

\Clemma*

\begin{proof}[Proof of Lemma \ref{lemma:C1-6}]
For each term of $C$, we enumerate non-elliptic webs with nine black boundary vertices. Since every term of $C$ besides $C_1$ has a factor of $\Delta_I$ where $I=\{i-1,i,i+1\}$ mod 9, for these terms we only check the appropriate rotations of the webs listed in Lemma \ref{lemma:frozenwebs}; see Figure \ref{fig:C2-6proof}. For $C_1$, in Figure \ref{fig:C1proof} we check all webs listed in Lemma \ref{lemma:pathlesswebs9}. Note that any completed proper coloring is unique given the compatibility conditions.
\begin{figure}
    \centering
    \includegraphics[width=\textwidth]{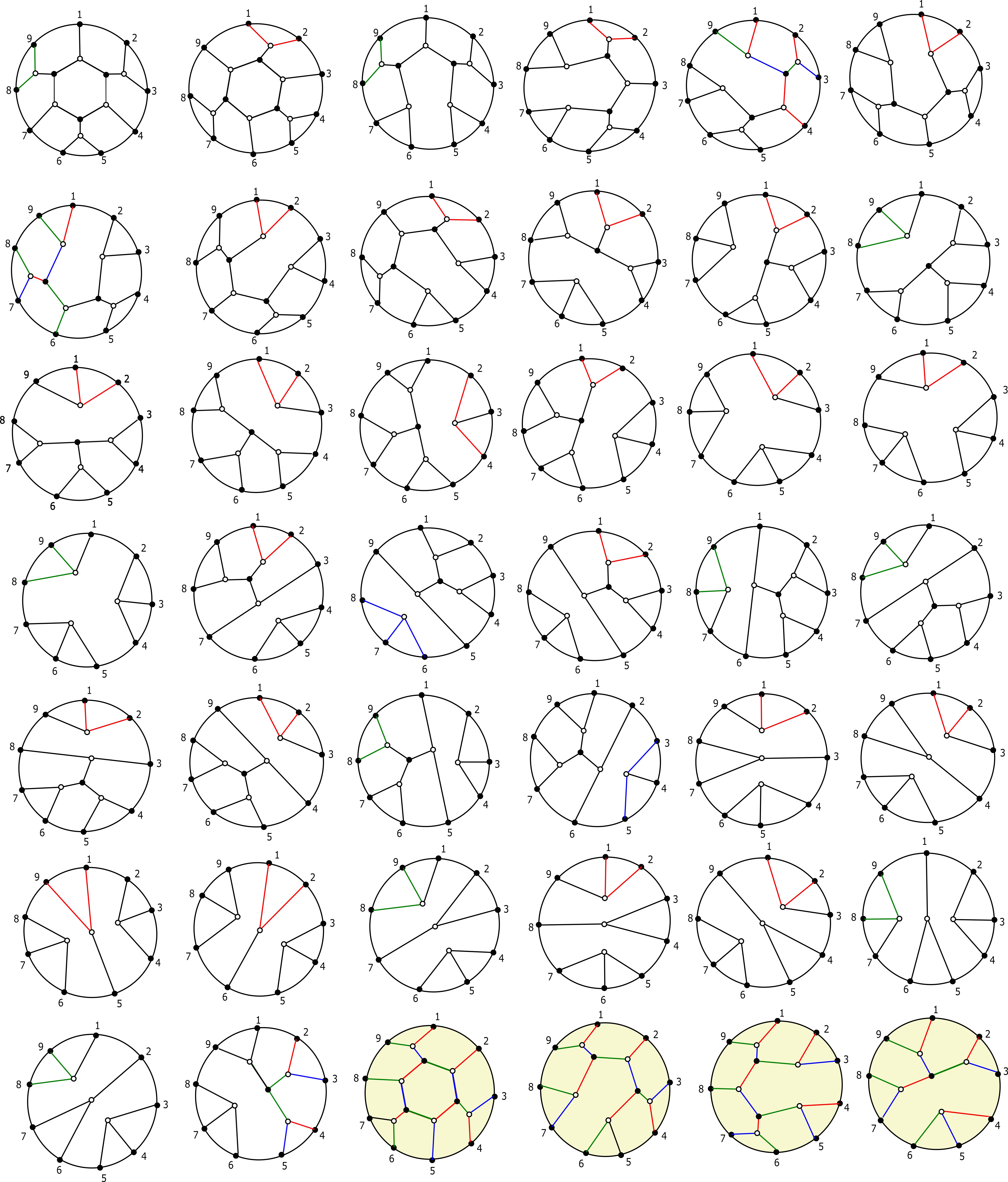}
    \caption{Enumeration of non-elliptic webs compatible (highlighted in yellow for emphasis) and incompatible with $C_1$, checking all webs listed in Lemma \ref{lemma:pathlesswebs9}.}
    \label{fig:C1proof}
\end{figure}
\begin{figure}
    \centering
   \includegraphics[width=\textwidth]{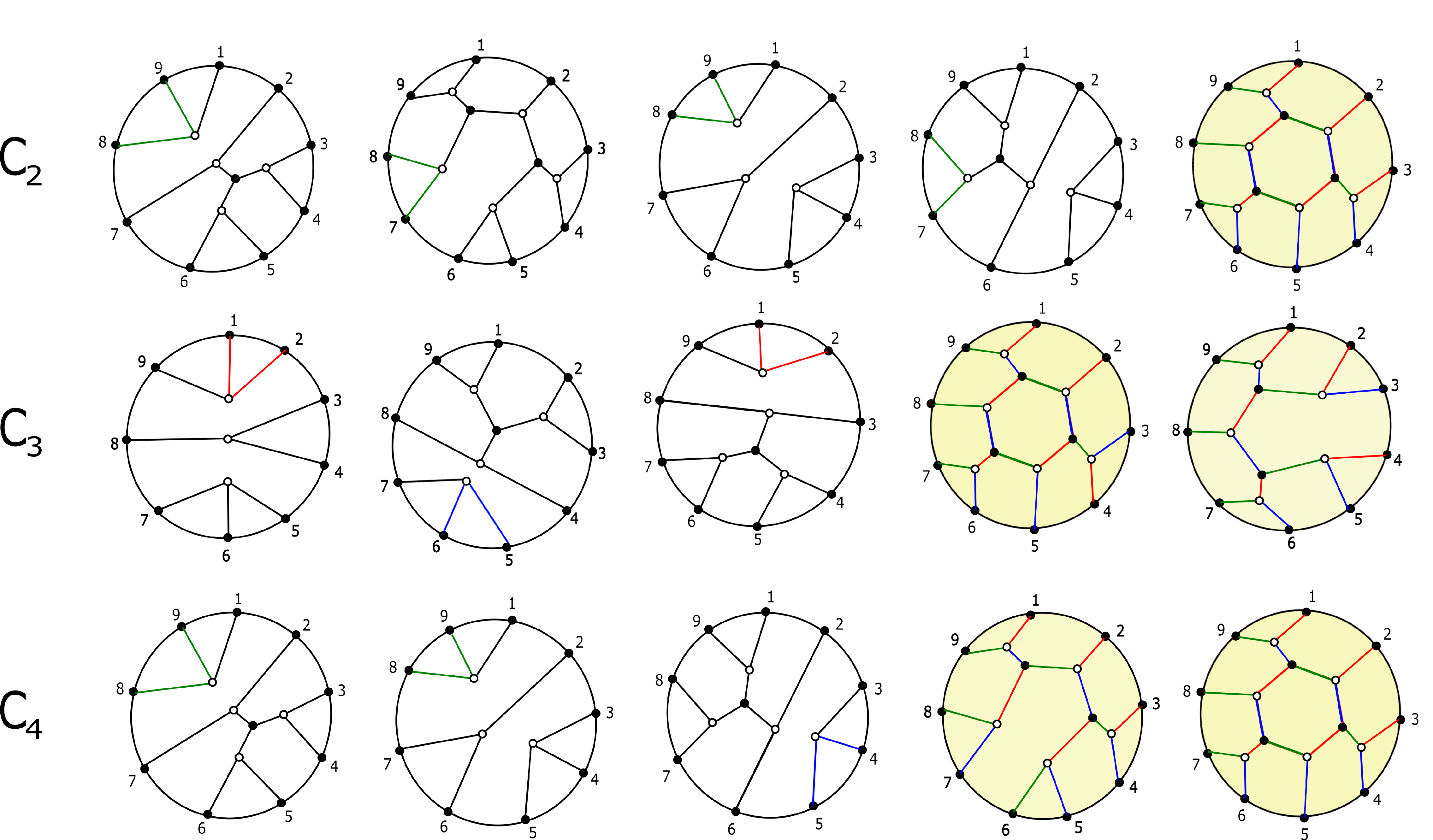}
    \caption{Enumeration of non-elliptic webs compatible (highlighted in yellow for emphasis) and incompatible with $C_2, C_3,$ and $C_4$, checking appropriate rotations of the webs listed in Lemma \ref{lemma:frozenwebs} for $\Delta_I = \Delta_{123}$, $\Delta_{789}$, and $\Delta_{123}$, respectively.}
    \label{fig:C2-6proof}
\end{figure}
\end{proof}

\begin{figure}
    \centering
    \includegraphics[width=\textwidth]{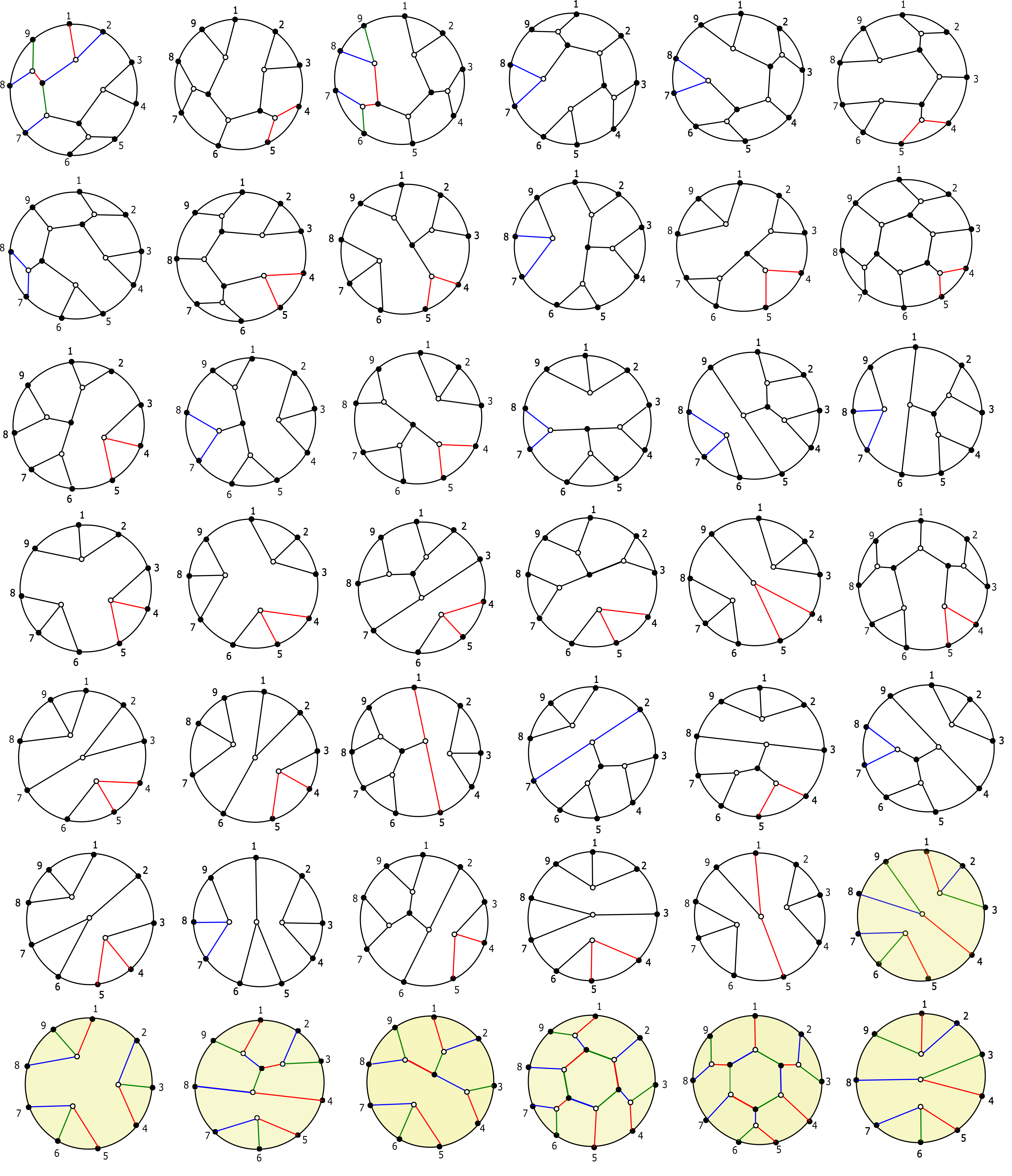}
    \caption{Enumeration of non-elliptic webs compatible (highlighted in yellow for emphasis) and incompatible with $Z_1$.}
    \label{fig:Z1proof}
\end{figure}
\begin{figure}
    \centering
    \includegraphics[width=\textwidth]{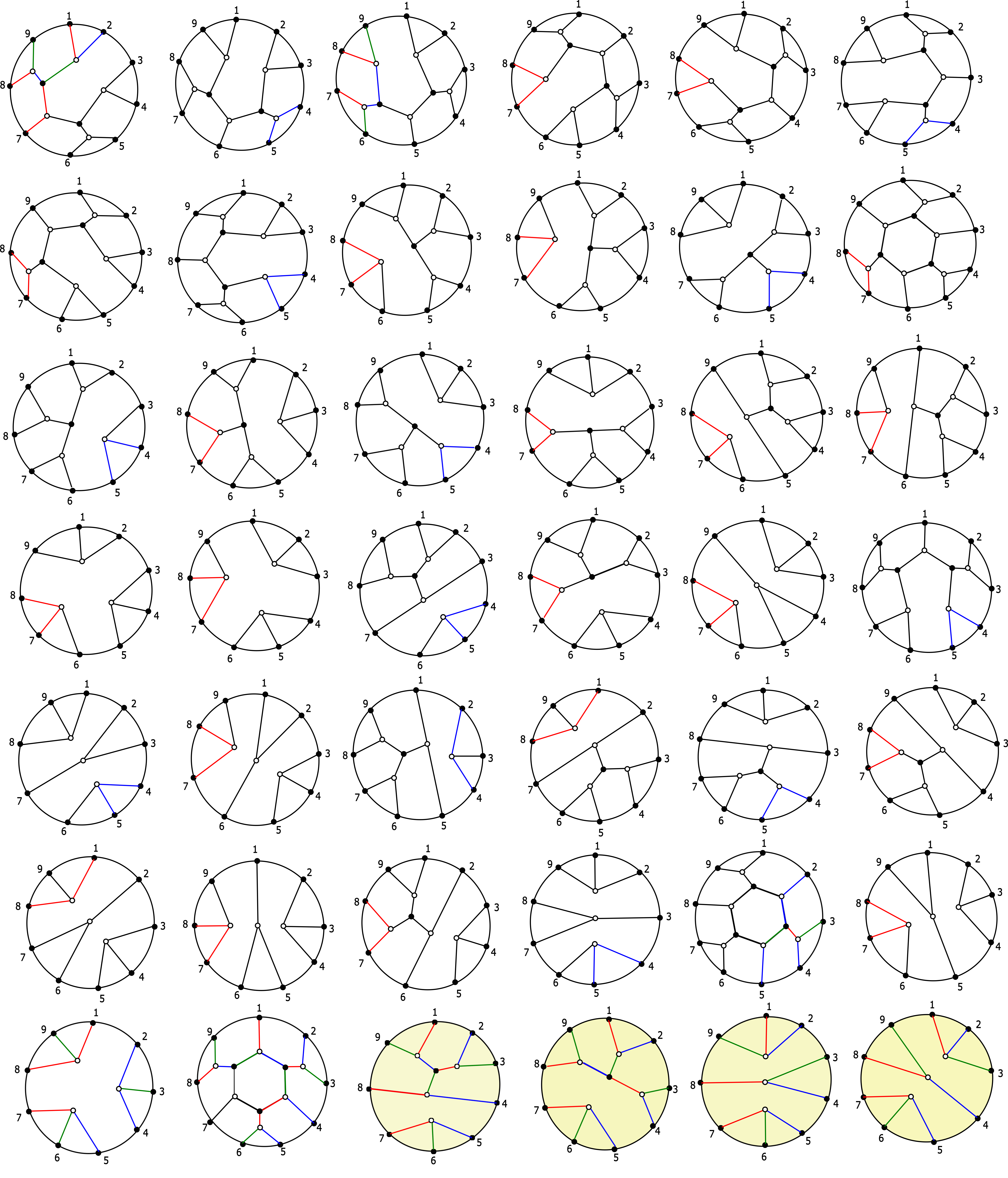}
    \caption{Enumeration of non-elliptic webs compatible (highlighted in yellow for emphasis) and incompatible with $Z_2$.}
    \label{fig:Z2proof}
\end{figure}

\Zlemma*

\begin{proof}[Proof of Lemma \ref{lemma:Z1-4}]
Similarly to the previous proof, in Figures \ref{fig:Z1proof} and \ref{fig:Z2proof} we enumerate non-elliptic webs with nine black boundary vertices for each of $Z_1$ and $Z_2$. Note that $Z_3=C_2=(123)(456)(789)$, and $Z_4$ is a cyclic shift of $Z_3$ corresponding to a counterclockwise rotation by one vertex, so we may refer to Lemma \ref{lemma:C1-6} for the webs compatible with those triple products. Also note that any completed proper coloring is unique given the compatibility conditions.
\end{proof}
\end{section}\FloatBarrier

\begin{section}{Appendix: Computations of Twists}\label{section:appendixAB}
In this Appendix, we demonstrate computations using Theorems \ref{theorem:connectivityforA} and \ref{theorem:connectivityforB}. Since this section contains many long expressions in Pl\"ucker coordinates, in what follows we will use the shorthand $(I)$ to denote the Pl\"ucker coordinate $\Delta_I$. We note that we chose to compute $\T(\sigma^2(A))$ and $\T(\sigma^7(B))$ (where $\sigma \in D_8$ represents clockwise rotation by one vertex) in particular because, for the initial seed we will consider, the numbers of terms in their Laurent expansions are relatively small compared to the corresponding numbers for other dihedral images of $A$ and $B$.

To check the Laurent expressions arising from our computations, we used {\sc SageMath} \cite{sagemath} and Pavel Galashin's applet \cite{applet} to explore the full set of Pl\"ucker cluster variables in a particular quiver for $Gr(3,8)$, via sequences of square moves on the corresponding plabic graph.  This allowed us to identify the resulting cluster variables as Pl\"ucker coordinates, as well as record a sample mutation sequence to get to each such variable. We then used {\sc SageMath}, including the ClusterSeed and ClusterAlgebra packages (thanks to the second author and Christian Stump; Dylan Rupel and Salvatore Stella) to arrive at Laurent polynomial expressions for all such cluster variables.  Computations in {\sc SageMath} subsequently allowed us to compute Laurent polynomial expressions for the entire mutation class of 128 cluster variables. With this list, we were able to identify the Laurent expansions for the remaining $56$ non-Pl\"ucker cluster variables as versions of $X$, $Y$, $A$, and $B$. We will refer to these expansions in the following two examples.

For reference, we list our computations for the twists of $A$ and $B$ as in \cite{marshscott2016} here:
    \begin{equation}\label{eq:twistA}
    \begin{split}
    \T(A) &=(123)(234)(456)(178)[(267)(358) - (235)(678)]\\
     &=(123)(234)(456)(178)Y^{235678}
     \end{split}
    \end{equation}
    \begin{equation}\label{eq:twistB}
    \begin{split}
    \T(B) &= (178)(456)(234)[(348)(367)(125) - (348)(567)(123) - (345)(367)(128)]\\
    &=(178)(456)(234)\sigma^4\rho(B)
     \end{split}
    \end{equation}
where $\rho$ reflects indices via $i \rightarrow 9-i$.
We justify these computations algebraically in Section \ref{subsection:tableoftwists}, along with providing a table of the twists of all $\Gr(3,7)$ cluster variables.

\subsection{Computing $\T(\sigma^2(A))$}\label{subsection:exampleforA}
We give a complete description of the triple dimer partition function given in Theorem \ref{theorem:connectivityforA} for
\[\sigma^2(A) = (356)(247)(138)-(356)(347)(128)-(237)(456)(138),\]
the cyclic rotation of $A$ clockwise by two vertices. In what follows, we will refer to the rotation of the batwing connectivity pattern clockwise by two vertices as the $\text{batwing}^2$.

From the theorem, the only triple dimers $D$ whose weights contribute to $\T(\sigma^2(A))$ are those such that $W(D)$ contains the $\text{batwing}^2$ as a nonelliptic summand. We were able to list these by first drawing all triple dimers such that $W(D)$ is the $\text{batwing}^2$, and then adding squares, bigons, and internal cycles wherever possible.
Tables \ref{tab:dimersforA} and \ref{tab:dimersforAcont} contain a complete list of these triple dimers, and their associated weights; note that only one triple dimer has a bigon in its corresponding web, causing its weight to have a coefficient of $C_{batwing^2}^D=2$ in $\T(\sigma^2(A))$.

According to Theorem \ref{theorem:connectivityforA}, summing the weights listed in Tables \ref{tab:dimersforA} and \ref{tab:dimersforAcont} should yield $\T(\sigma^2(A))$. Indeed, it follows from (\ref{eq:twistA}) that $\T(\sigma^2(A))= (123)(345)(456)(678)Y^{124578}$,
and our code yielded the following Laurent polynomial expression for this product of cluster variables:
\begin{align*}
    \T(\sigma^2(A))&= 
    \frac{(123)(345)(456)(678)}{(246)(568)(268)^2(168)} [(248)^2(256)(268)(168)^2(567) +\\
    &+(248)(245)(268)^2(168)^2(567) + (248)(246)(568)(268)(168)(128)(567)\\
    &+(248)^2(256)^2(168)^2(567) + (248)(246)(568)(268)(168)(128)(567)\\
    &+2 \cdot (248)(246)(256)(568)(168)(128)(678) + (245)(246)(568)(268)(168)(128)(678)\\
    &+(246)^2(568)^2(128)^2(678) + (248)(246)(256)(568)(268)(168)(178)\\
    &+ (246)^2(568)^2(268)(128)(178)].
\end{align*}
The terms of this Laurent expansion corresponding to each triple dimer are also listed in Tables \ref{tab:dimersforA} and \ref{tab:dimersforAcont}, confirming that the sums agree.

\begin{subsection}{Computing $\T(\sigma^7(B))$} \label{subsection:exampleforB}
We similarly give a complete description of the triple dimer partition function given in Theorem \ref{theorem:connectivityforB} for
\[\sigma^7(B) = (147)(156)(238)-(123)(178)(456)-(123)(147)(568),\]
the cyclic rotation of $B$ clockwise by seven vertices. 

From the theorem, the only triple dimers $D$ whose weights contribute to $\T(\sigma^7(B))$ are those such that $W(D)$ contains the appropriate rotation of the octopus as a nonelliptic summand. We were able to list these using the same method as in the previous section; Tables \ref{tab:dimersforB}, \ref{tab:dimersforBcont}, and \ref{tab:dimersforBcontcont} contain a complete list of these triple dimers and their associated weights and coefficients.

To check that the sum of these weights (with multiplicity) indeed yields $\T(\sigma^7(B))$, we have from the expression (\ref{eq:twistB}) for $\T(B)$ that
\begin{align*}
\T(\sigma^7(B))
& =(123)(345)(678)\bigg[(256)[(247)(138) - (347)(128)] - (237)(456)(128)\bigg]\\
&=(123)(345)(678)\left( \sigma^3 \rho(B)\right)\\
\end{align*}
where $\rho$ reflects indices via $i \to 9-i$ and $\sigma$ follows this by clockwise rotation. Our code yielded the following Laurent polynomial for this expression: 

\begin{align*}
    \T(\sigma^7(B))
    &= \frac{(123)(345)(678)}{(248)(124)(568)(268)^2(168)} [(248)^3(256)(268)(168)^2(123)(567)\\
    &+(248)^2(246)(568)(268)(168)(128)(123)(567) + (248)^2(256)(268)(168)^2(128)(234)(567)\\
    &+(248)(246)(568)(268)(168)(128)^2(234)(567) + (248)^3(256)^2(168)^2(123)(678)\\
    &+2 \cdot (248)^2(246)(256)(568)(168)(128)(123)(678) + (248)(246)^2(568)^2(128)^2(123)(678)\\
    &+(248)(124)(256)(568)(268)(168)(128)(234)(678) + (248)^2(256)^2(168)^2(128)(234)(678)\\
    &+(124)(246)(568)^2(268)(128)^2(234)(678) + 2 \cdot (248)(246)(256)(568)(168)(128)^2(234)(678)\\
    &(246)^2(568)^2(128)^3(234)(678) + (248)^2(246)(256)(568)(268)(168)(123)(178)\\
    &+(248)(246)^2(568)^2(268)(128)(123)(178) + (124)(246)(568)^2(268)^2(128)(234)(178)\\
    &+(248)(246)(256)(568)(268)(168)(128)(234)(178) + (246)^2(568)^2(268)(128)^2(234)(178)]
\end{align*}
The terms of this Laurent expansion corresponding to each triple dimer are also listed in Tables \ref{tab:dimersforB}, \ref{tab:dimersforBcont}, and \ref{tab:dimersforBcontcont}, confirming that the sums agree.
\end{subsection}

\subsection{Computing Twists Algebraically}\label{subsection:tableoftwists}
In this section, we list the twists of all cluster variables in $\Gr(3,7)$, and algebraically justify our computations for the twists of $A$ and $B$ in $\Gr(3,8)$. We utilize the following results from Section \ref{subsection:twist} (where indices are taken in increasing order modulo $n$):
\begin{itemize}
    \item We have $\T(\Delta_{a,a+1,a+2}) = \Delta_{a+1,a+2,a+3}\Delta_{a+2,b+1,b+2}$.
    \item When $b\not = a-1, a+2$, we have $\T(\Delta_{a,a+1,b}) = \Delta_{a+1,a+2,a+3}\Delta_{a+2,b+1,b+2}$.
    \item When $J = \{a,b,c\}$ where none of $a,b,c$ are adjacent, we have
    \begin{eqnarray*}\T(\Delta_J) &= \det \bigg( v_{a+1} \times v_{a+2} \quad v_{b+1} \times v_{b+2} \quad v_{c+1} \times v_{c+2}\bigg) \\
&=\begin{cases}
X^{a+1,~a+2,~b+1,~b+2,~c+1,~c+2} & a,b,c \neq n-1\\
Y^{a+1,~a+2,~b+1,~b+2,~c+1,~c+2} & \text{otherwise}
\end{cases}.\end{eqnarray*}
\end{itemize}

Table \ref{tab:gr37twists} lists twists in $\Gr(3,7)$.
In this setting, all calculations of twists of quadratic differences of form $X^S$ are special cases of the computation appearing in the proof of Proposition \ref{prop:untwisted}; for instance, when $(a-2,a-1,b-1,b-2,c-2,c-1) = (1,2,3,4,5,6)$ we obtain $\T(X^{123456}) = \Delta_{167}\Delta_{456}\Delta_{234}[\Delta_{357}]$. The computations for those of form $Y^S$ are similar.
%
%

We note that the non-crossing matching corresponding to a given quadratic difference provides a convenient heuristic for computing frozen factors of twists of quadratic differences. In particular, these frozen factors are indexed by all face labels appearing between two boundary vertices that are included in the corresponding matching, but not connected to each other. We also observe overall that in $\Gr(3,7)$, we either have $$\T(X^{s_1,s_2,s_3,s_4,s_5,s_6}) = \Delta_{s_2,s_2+1,s_2+2}\Delta_{s_4,s_4+1,s_4+2}\Delta_{s_6,s_6+1,s_6+2}\Delta_{s_2+1,s_4+1,s_6+1}$$ or that $\T(X^{s_1,s_2,s_3,s_4,s_5,s_6})$ is the product of two frozen variables and a quadratic difference. 

\begin{table}[H]
\centering
\begin{tabular}{|c|c||c|c|}
\hline
\textbf{Variable} & \textbf{Twist} & \textbf{Variable} & \textbf{Twist} \\
        \hline
       $\Delta_{124}$  & $\Delta_{234}\Delta_{356}$ & $\Delta_{126}$ & $\Delta_{234}\Delta_{137}$\\
       \hline
       $\Delta_{235}$  & $\Delta_{345}\Delta_{467}$ & $\Delta_{237}$ & $\Delta_{345}\Delta_{124}$\\
       \hline
       $\Delta_{346}$  & $\Delta_{456}\Delta_{157}$ & $\Delta_{134}$ & $\Delta_{456}\Delta_{235}$\\
       \hline
       $\Delta_{457}$  & $\Delta_{567}\Delta_{126}$ & $\Delta_{245}$ & $\Delta_{567}\Delta_{346}$\\
       \hline
       $\Delta_{156}$  & $\Delta_{167}\Delta_{237}$ & $\Delta_{356}$ & $\Delta_{167}\Delta_{457}$\\
       \hline
       $\Delta_{267}$  & $\Delta_{127}\Delta_{134}$ & $\Delta_{467}$ & $\Delta_{127}\Delta_{156}$\\
       \hline
       $\Delta_{137}$  & $\Delta_{123}\Delta_{245}$ & $\Delta_{157}$ & $\Delta_{123}\Delta_{267}$\\
       \hline\hline
       $\Delta_{125}$  & $\Delta_{234}\Delta_{367}$ & $\Delta_{135}$ & $X^{234567}$\\
       \hline
       $\Delta_{236}$ & $\Delta_{345}\Delta_{147}$ & $\Delta_{246}$ & $Y^{134567}$\\
       \hline
       $\Delta_{347}$ & $\Delta_{456}\Delta_{125}$ & $\Delta_{357}$ & $X^{124567}$\\
       \hline
       $\Delta_{145}$ & $\Delta_{567}\Delta_{236}$ & $\Delta_{146}$ & $Y^{123567}$\\
       \hline
       $\Delta_{256}$ & $\Delta_{167}\Delta_{347}$ & $\Delta_{257}$ & $X^{123467}$\\
       \hline
       $\Delta_{367}$ & $\Delta_{127}\Delta_{145}$ & $\Delta_{136}$ & $Y^{123457}$\\
       \hline
       $\Delta_{147}$ & $\Delta_{123}\Delta_{256}$ & $\Delta_{247}$ & $X^{123456}$\\
       \hline\hline
       $X^{123456}$ & $\Delta_{167}\Delta_{234}\Delta_{456}\Delta_{357}$ & $Y^{123456}$ & $\Delta_{345}\Delta_{567}Y^{123467}$\\
       \hline
       $X^{234567}$& $\Delta_{127}\Delta_{345}\Delta_{567}\Delta_{146}$ & $Y^{234567}$ & $\Delta_{167}\Delta_{456}X^{123457}$\\
       \hline
       $X^{134567}$ & $\Delta_{127}\Delta_{567}Y^{123456}$ & $Y^{134567}$ & $\Delta_{123}\Delta_{167}\Delta_{456}\Delta_{257}$\\
       \hline
       $X^{124567}$ & $\Delta_{127}\Delta_{234}\Delta_{567}\Delta_{136}$ & $Y^{124567}$ & $\Delta_{123}\Delta_{167}Y^{234567}$\\
       \hline
       $X^{123567}$ & $\Delta_{127}\Delta_{234}X^{134567}$ & $Y^{123567}$ & $\Delta_{123}\Delta_{167}\Delta_{345}\Delta_{247}$\\
       \hline
       $X^{123467}$ & $\Delta_{127}\Delta_{234}\Delta_{456}\Delta_{135}$ & $Y^{123467}$ & $\Delta_{123}\Delta_{345}Y^{124567}$\\
       \hline
       $X^{123457}$ & $\Delta_{234}\Delta_{456}X^{123567}$ & $Y^{123457}$ & $\Delta_{123}\Delta_{345}\Delta_{567}\Delta_{246}$\\
       \hline
    \end{tabular}
    \caption{Twists of Gr(3,7) cluster variables, arranged by cyclic orbits.}
    \label{tab:gr37twists}
\end{table}

We now compute the twists of $A$ and $B$ in $\Gr(3,8)$. Applying the twist to the expression
$$A = (134)(258)(167) - (134)(125)(678) - (158)(234)(167)$$ 
yields
\begin{align*}
    \T(A) &= (235)(456)[(347)(126)-(346)(127)](238)(178)\\
    &-(235)(456)(234)(367)(178)(128) - (123)(267)(345)(456)(238)(178)\\
    &=(178)(456)[(235)(347)(126)(238) - (235)(346)(127)(238)\\
    &-(235)(234)(367)(128) - (123)(267)(345)(238).]
\end{align*}
We use the Pl\"ucker relation $(123)(467)-(124)(367)+(126)(347) - (127)(346) = 0$ to arrive at the expression
\begin{align*}
    \T(A) &= (178)(456)[(235)(238)(124)(367) - (235)(238)(123)(467)\\
    &-(235)(234)(367)(128) - (123)(267)(345)(238)],
\end{align*}
and further simplify using the Pl\"ucker relation $(124)(238) - (234)(128) = (123)(248)$ to arrive at
\begin{align*}
    \T(A) &= (178)(456)(123)[(235)(367)(248) - (235)(238)(467) - (267)(238)(345)].
\end{align*}
The Pl\"ucker relation $(267)(348)-(367)(248) + (467)(238) - (678)(234) = 0$ implies that $(235)(367)(248) - (235)(238)(467) = (235)[(267)(348) - (234)(678)]$ and $(267)(238)(345) = (267)[(235)(348) - (234)(358)]$. Substituting yields
\begin{align*}
    \T(A) &= (178)(456)(123)[(267)(234)(358) - (234)(235)(678)]\\
    &=(123)(234)(456)(178)[(267)(358) - (235)(678)]\\
     &=(123)(234)(456)(178)Y^{235678}.
    \end{align*}

We now apply the twist to the expression
$$B = (258)(134)(267) - (234)(158)(267) - (234)(125)(678),$$
which yields
\begin{align*}
    \T(B) &= [(134)(267) - (234)(167)](235)(456)(348)(178)\\
    &- (345)(456)(267)(123)(348)(178) - (345)(456)(234)(367)(178)(128)\\
    &= (456)(178)[(134)(267)(235)(348) - (234)(167)(235)(348)\\
    &-(345)(267)(123)(348) - (345)(234)(367)(128).]
\end{align*}
Using the Pl\"ucker relation $(134)(235) = (123)(345) + (135)(234)$, we arrive at the expression
\begin{align*}
    \T(B) &= (178)(456)(234)[(135)(348)(267) - (235)(167)(348) - (345)(367)(128)],
\end{align*}
and using the Pl\"ucker relation $(167)(235) - (267)(135) + (367)(125) - (567)(123) = 0$ yields
\begin{align*}
    \T(B) &= (178)(456)(234)[(348)(367)(125) - (348)(567)(123) - (345)(367)(128)]\\
    &=(178)(456)(234)\sigma^5\rho(B),
\end{align*}
where $\rho$ reflects indices via $i \to 9-i$ and $\sigma$ follows this by clockwise rotation.
\begin{table}[!htbp]
\caption{Triple Dimers for $\T(\sigma^2(A))$.} \label{tab:dimersforA}
  \begin{tabular}
      {lcc} \hline Triple Dimer Configuration & Weight of Dimer & Term in Laurent Expansion\\
      \hline
      \parbox[c]{1em}{
      \includegraphics[width=1.5in]{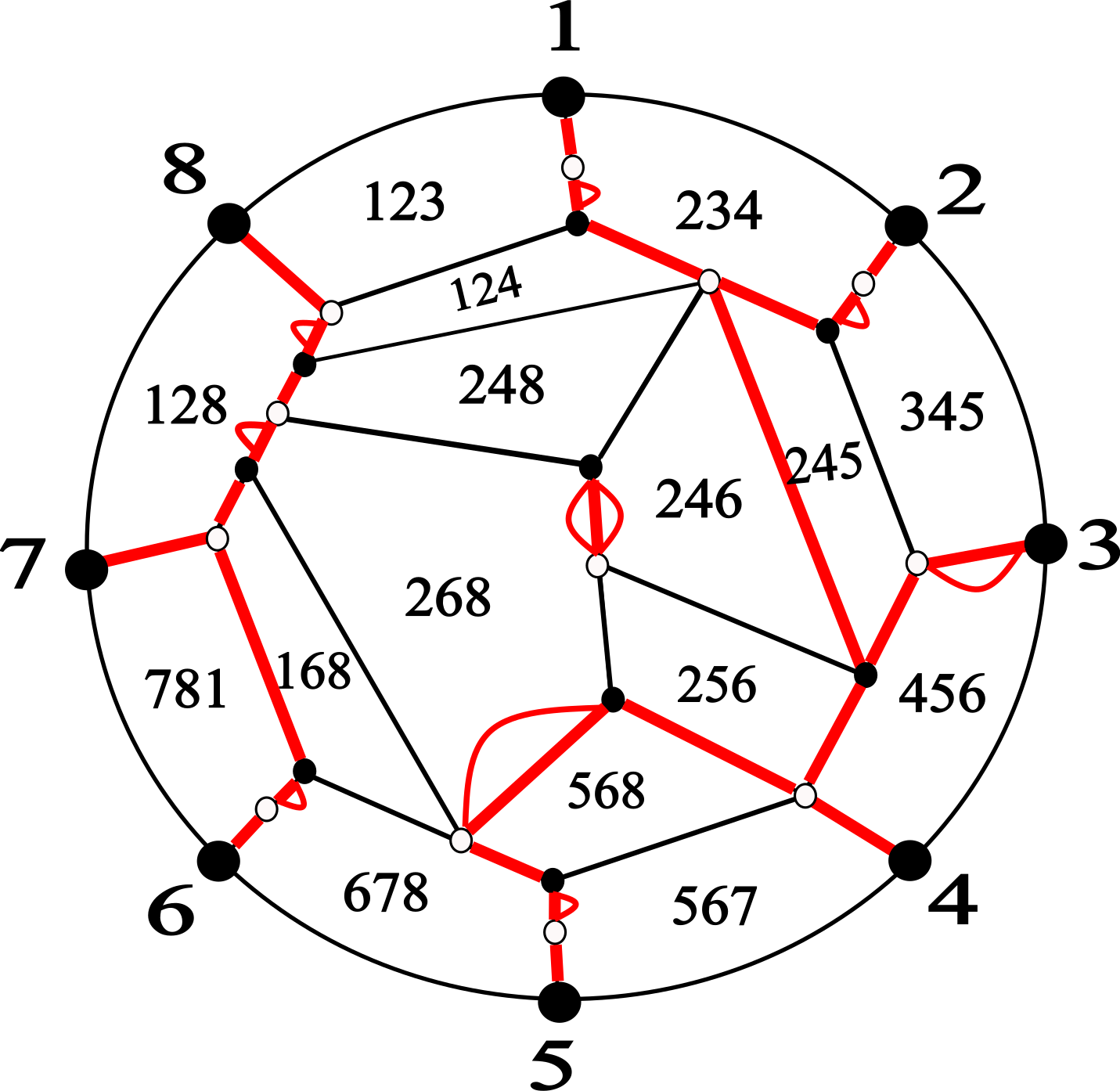}} & $\frac{(123)(345)(456)(567)(678)(248)^2(168)(256)}{(268)(246)}$ & $(248)^2(256)(268)(168)^2(567)$  \\
      \parbox[c]{1em}{
      \includegraphics[width=1.5in]{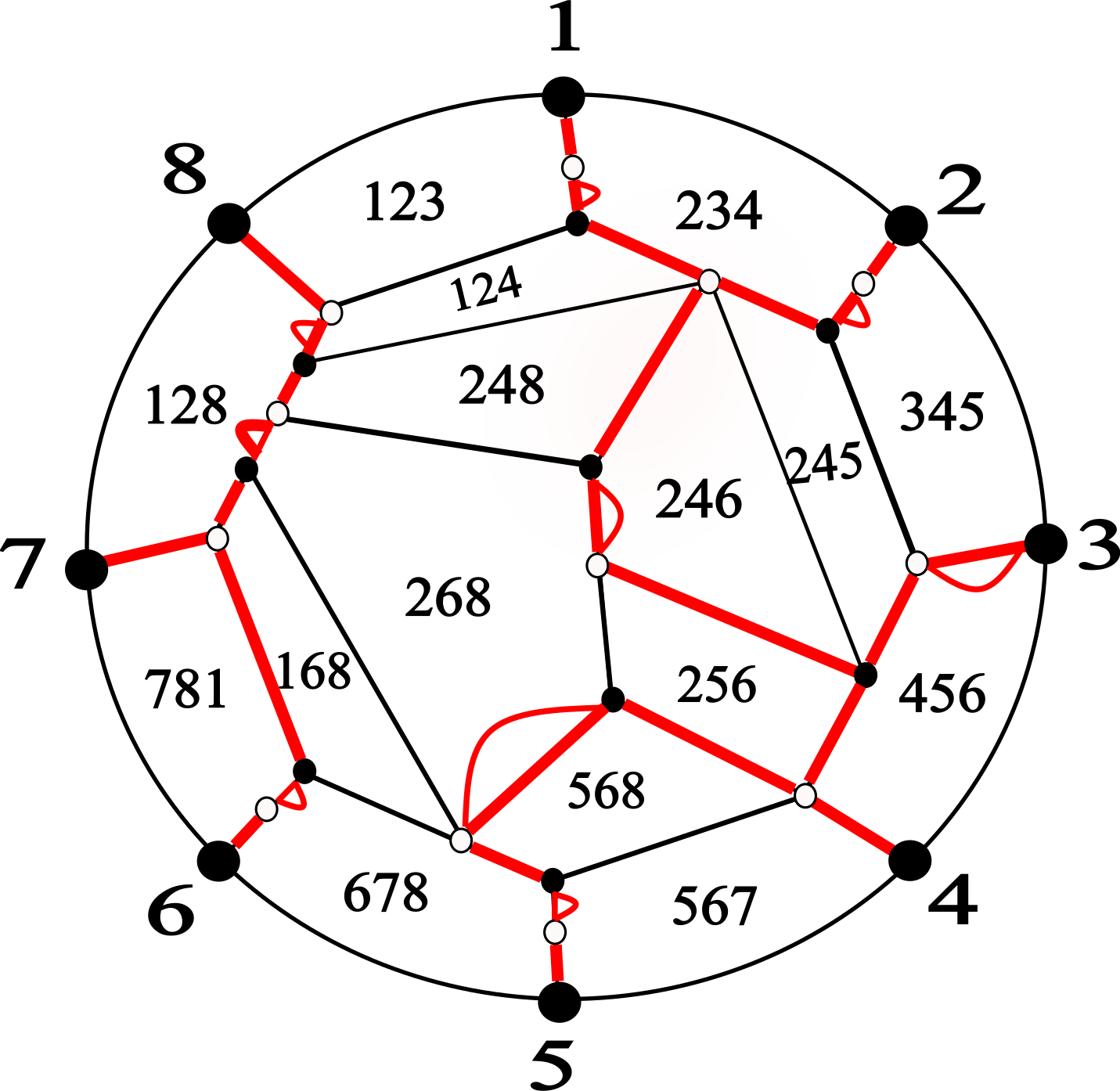}} & $\frac{(123)(345)(456)(567)(678)(248)(168)(245)}{(568)(246)}$ & $(248)(245)(268)^2(168)^2(567)$\\
       \parbox[c]{1em}{
      \includegraphics[width=1.5in]{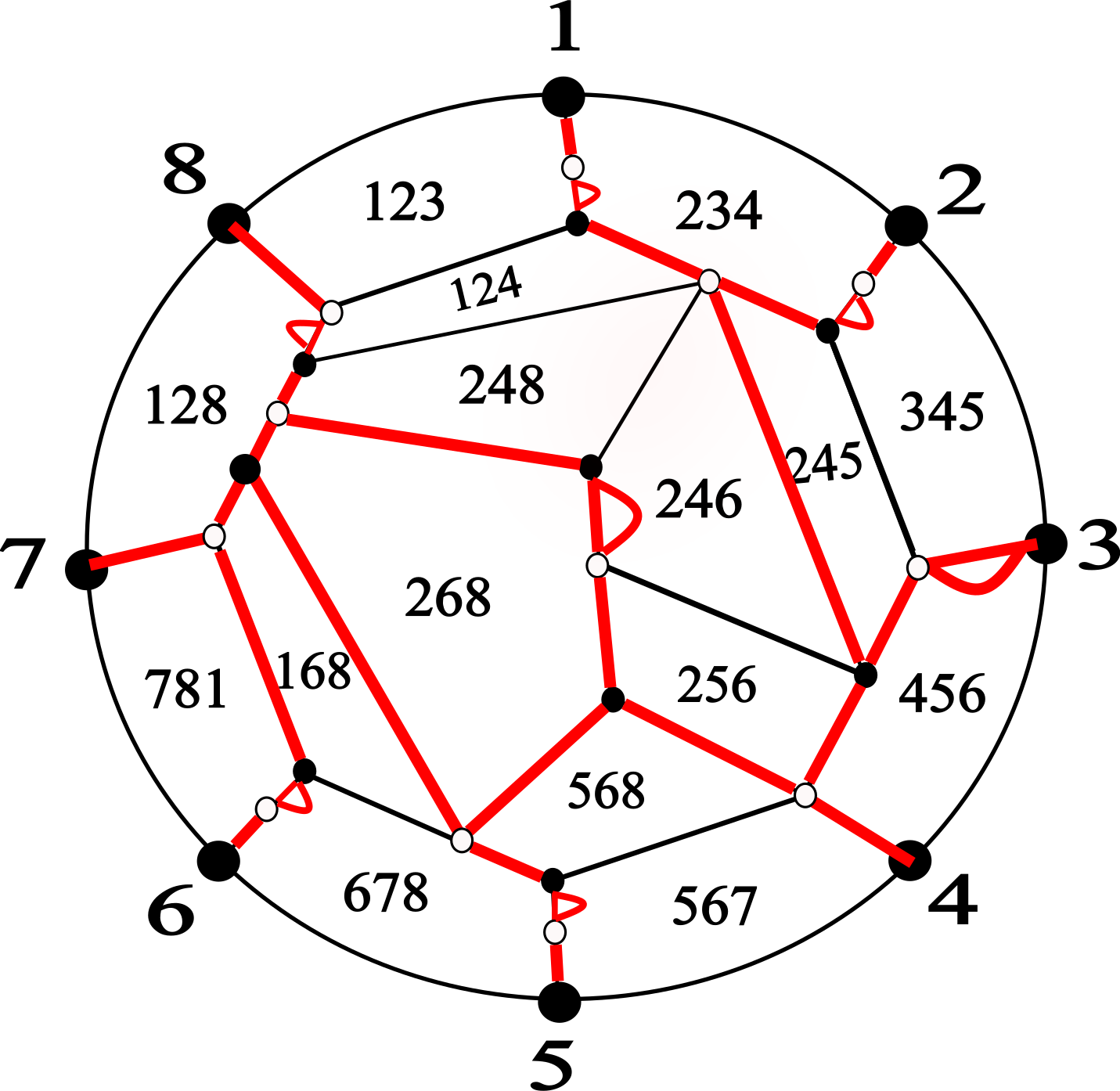}} & $\frac{(123)(345)(456)(567)(678)(128)(248)}{(268)}$ & $(248)(246)(568)(268)(168)(128)(567)$\\
       \parbox[c]{1em}{
      \includegraphics[width=1.5in]{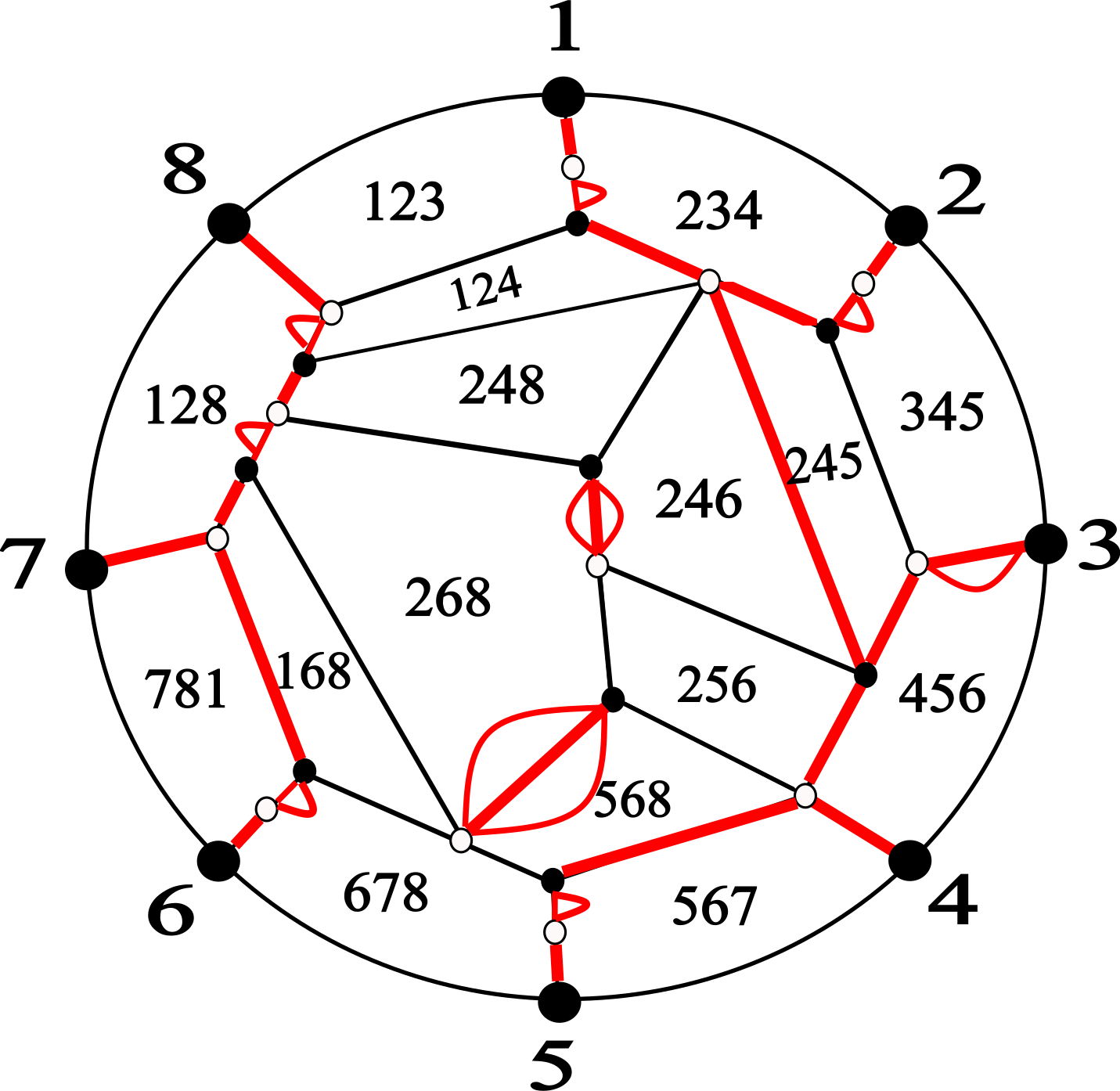}} & $\frac{(123)(345)(456)(678)^2(248)^2(168)(256)^2}{(568)(246)(268)^2}$ & $(248)^2(256)^2(168)^2(567)$ \\
       \parbox[c]{1em}{
      \includegraphics[width=1.5in]{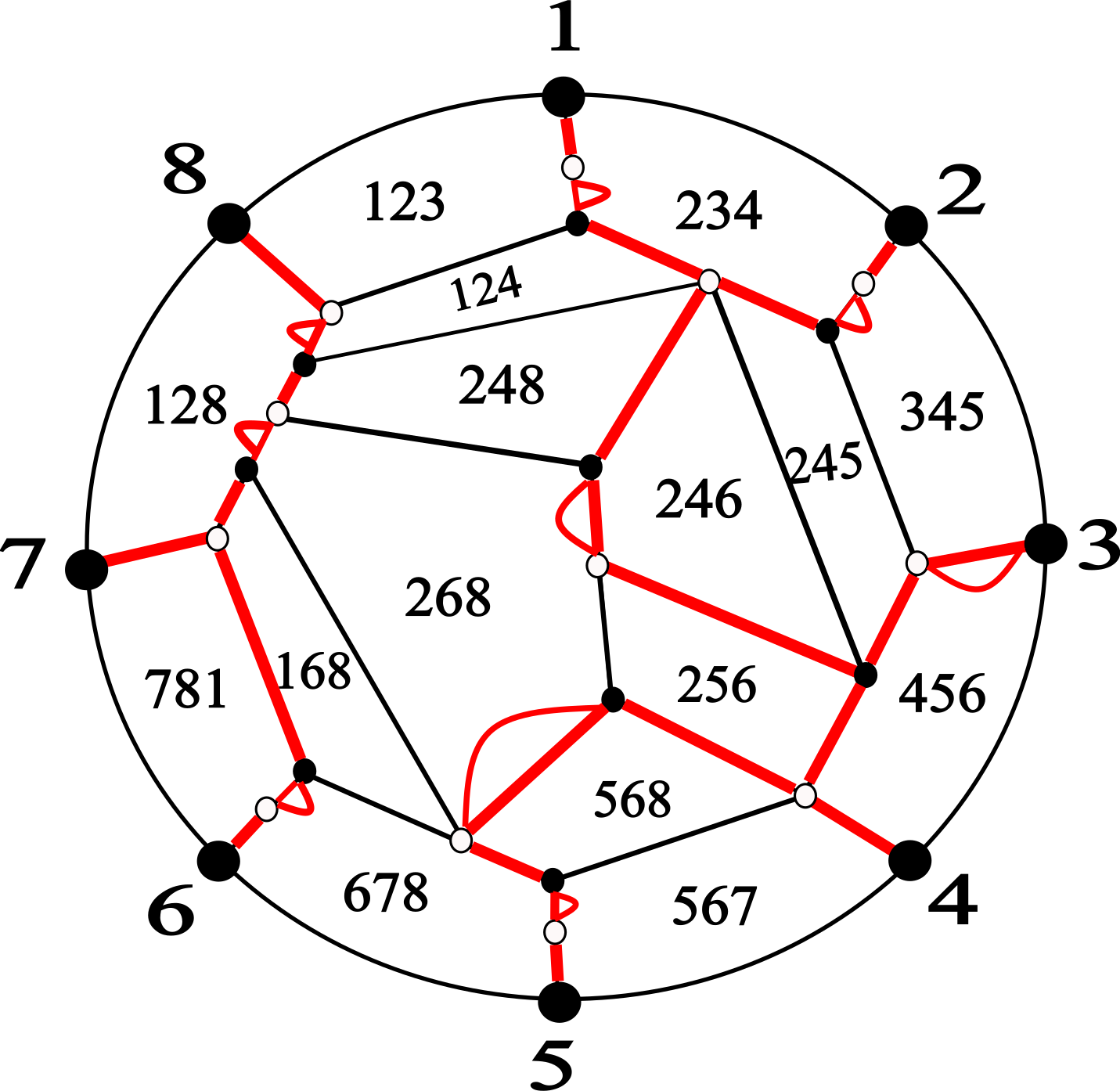}} & $\frac{(123)(345)(456)(678)^2(248)(168)(256)(245)}{(268)(246)(568)}$ & $(248)(246)(568)(268)(168)(128)(567)$  \\
  \end{tabular}
\end{table}

\begin{table}[htbp] 
\caption{Triple Dimers for $\T(\sigma^2(A))$ Continued.} \label{tab:dimersforAcont}
  \begin{tabular}
      {lcc} \hline Triple Dimer Configuration & Weight of Dimer & Term in Laurent Expansion\\
      \hline
       \parbox[c]{1em}{
      \includegraphics[width=1.5in]{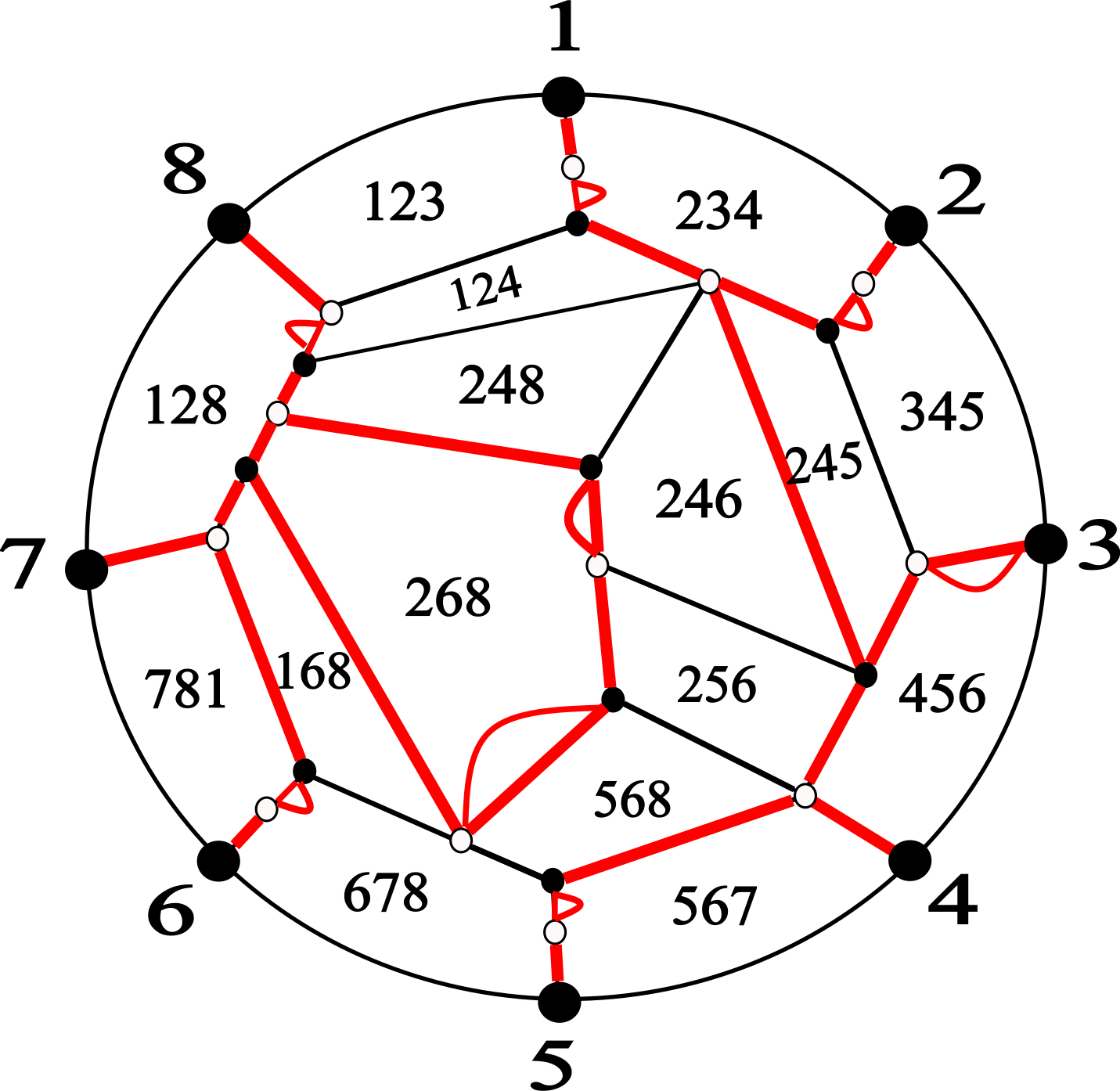}} & $2 \cdot \frac{(123)(345)(456)(678)^2(128)(248)(256)}{(268)^2}$ & $2 \cdot (248)(246)(256)(568)(168)(128)(678)$\\
       \parbox[c]{1em}{
      \includegraphics[width=1.5in]{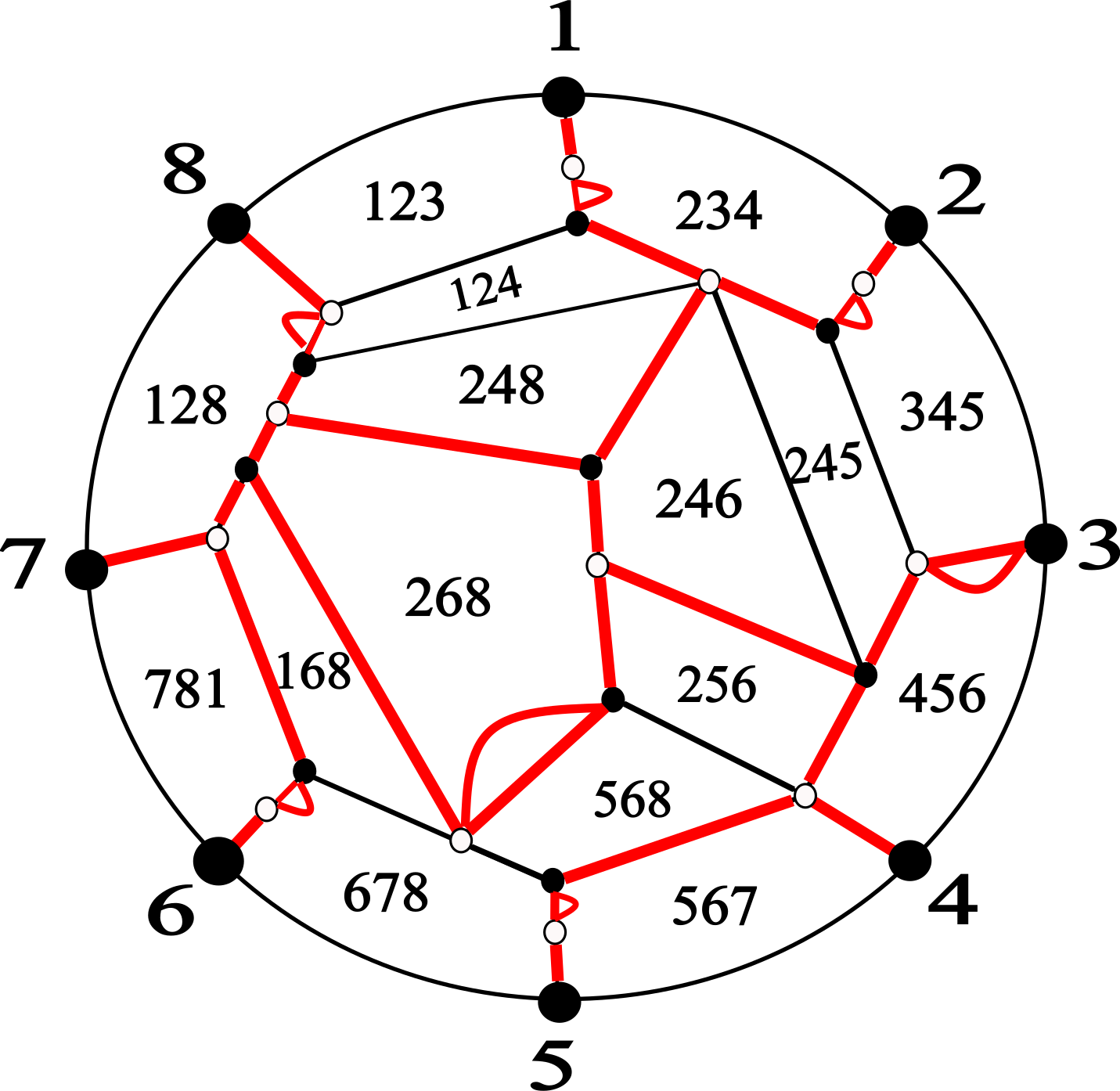}} & $\frac{(123)(345)(456)(678)^2(128)(245)}{(268)}$ & $(245)(246)(568)(268)(168)(128)(678)$ \\
       \parbox[c]{1em}{
      \includegraphics[width=1.5in]{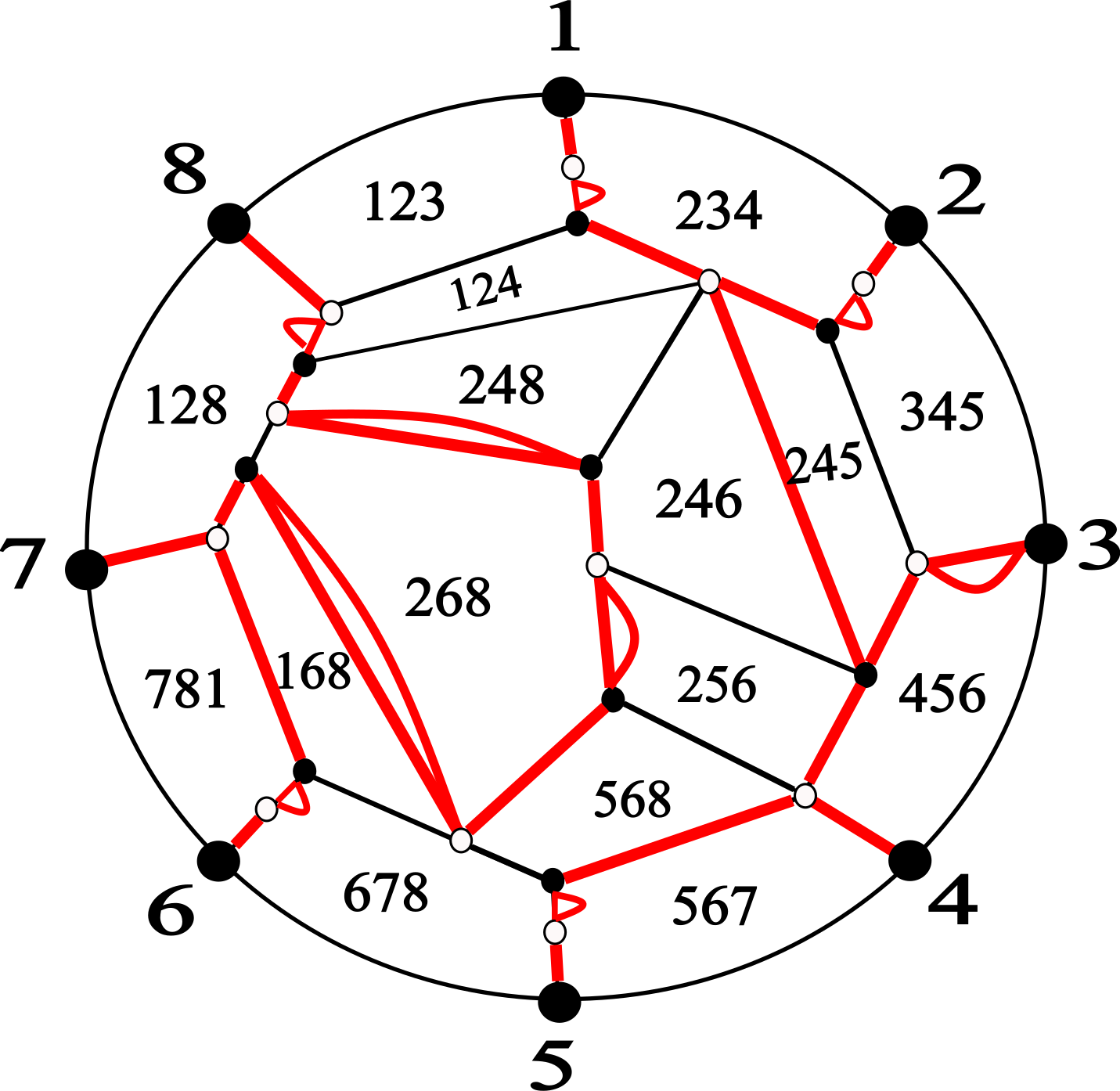}} & $\frac{(123)(345)(456)(678)^2(128)^2(568)(246)}{(268)^2(168)}$ & $(246)^2(568)^2(128)^2(678)$ \\
       \parbox[c]{1em}{
      \includegraphics[width=1.5in]{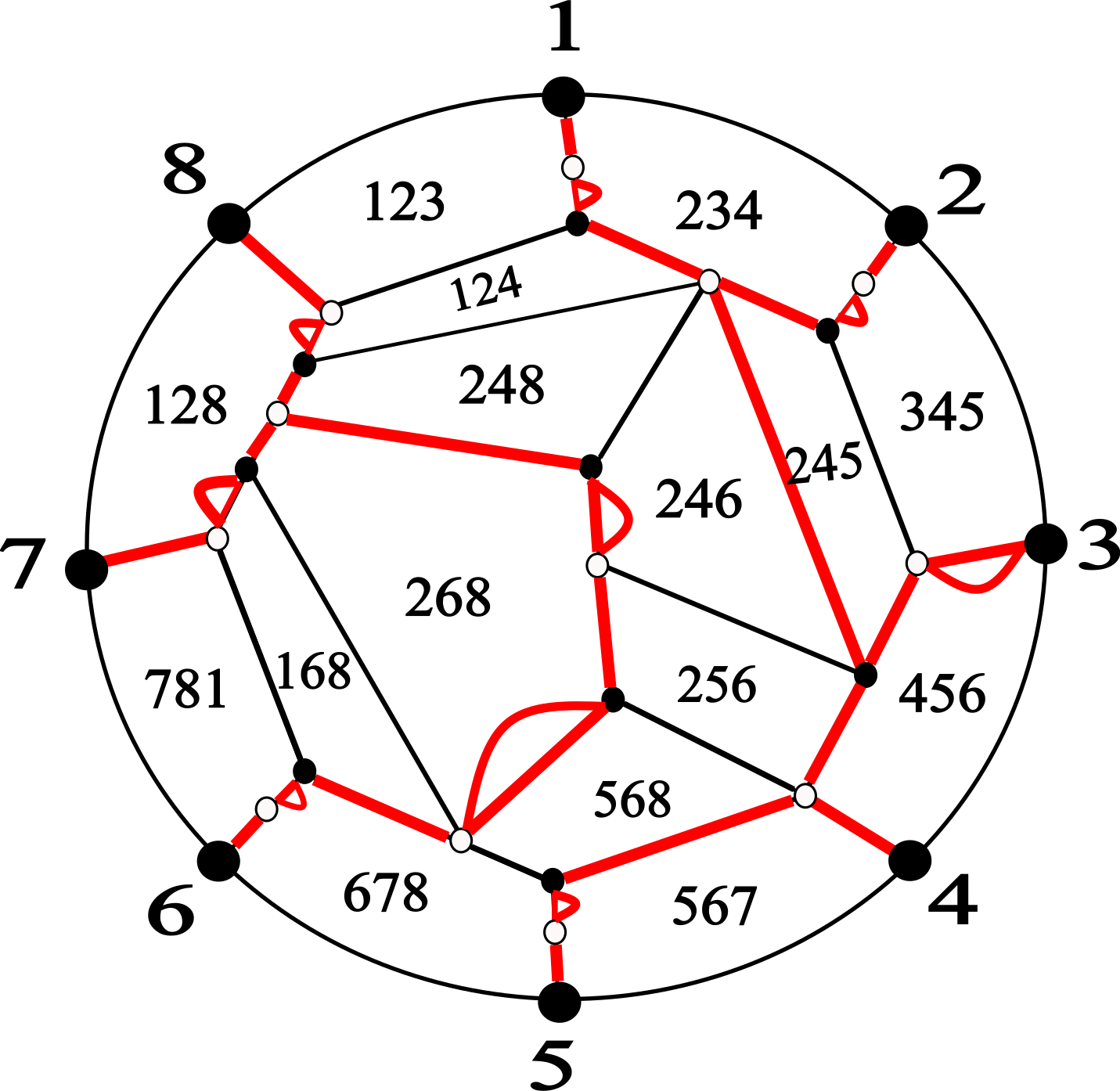}} & $\frac{(123)(345)(456)(678)(178)(248)(256)}{(268)}$ & $(248)(246)(256)(568)(268)(168)(178)$\\
       \parbox[c]{1em}{
      \includegraphics[width=1.5in]{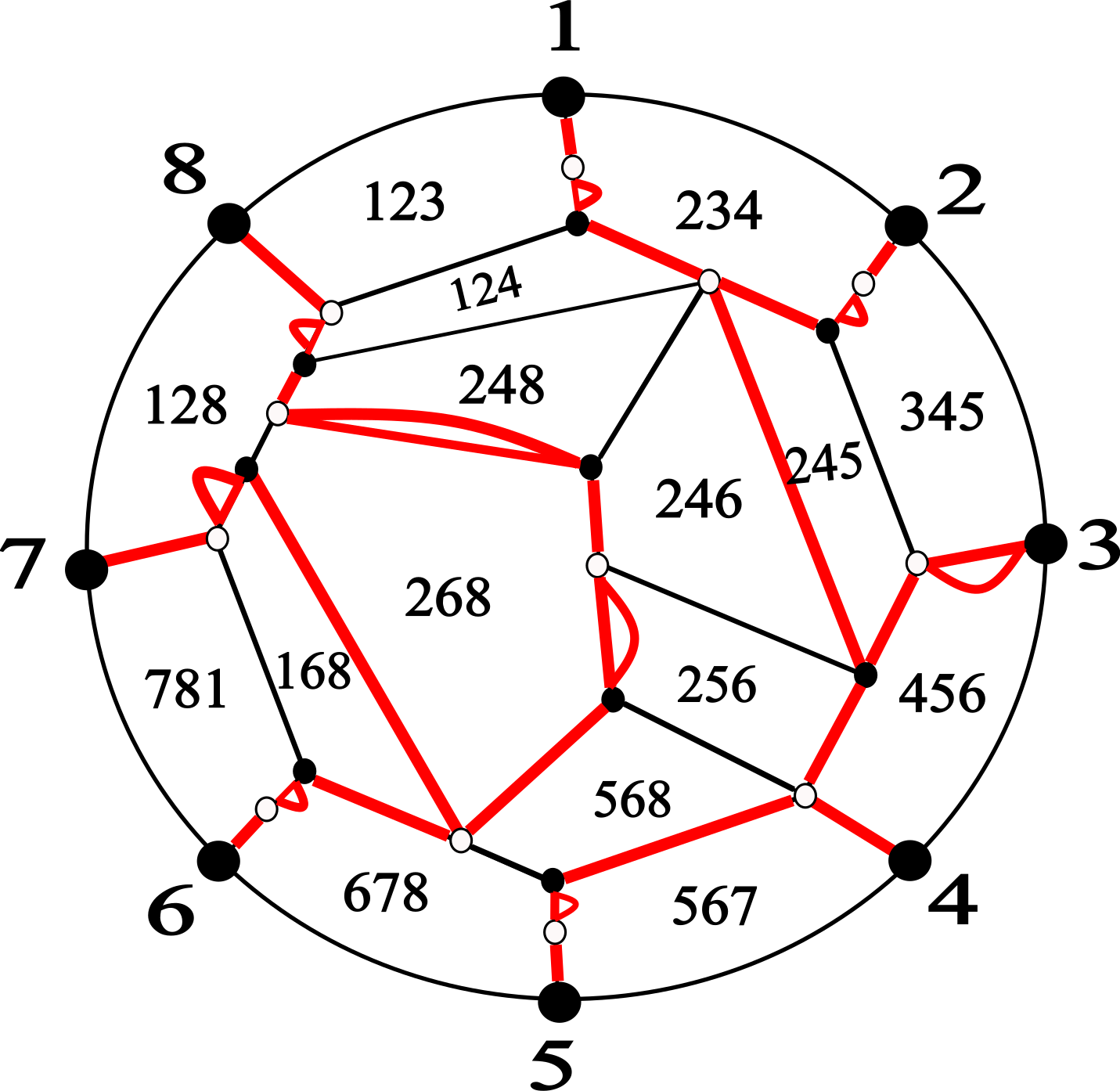}} & $\frac{(123)(345)(456)(678)(178)(128)(568)(246)}{(268)(168)}$ & $(246)^2(568)^2(268)(128)(178)$ \\
  \end{tabular}
\end{table}

\begin{table}
  [htbp] \caption{Triple Dimers for $\T(\sigma^7(B))$.} \label{tab:dimersforB}
  \begin{tabular}
      {lcc} \hline Triple Dimer Configuration & Weight of Dimer & Term in Laurent Expansion\\
      \hline
       \parbox[c]{1em}{
      \includegraphics[width=1.5in]{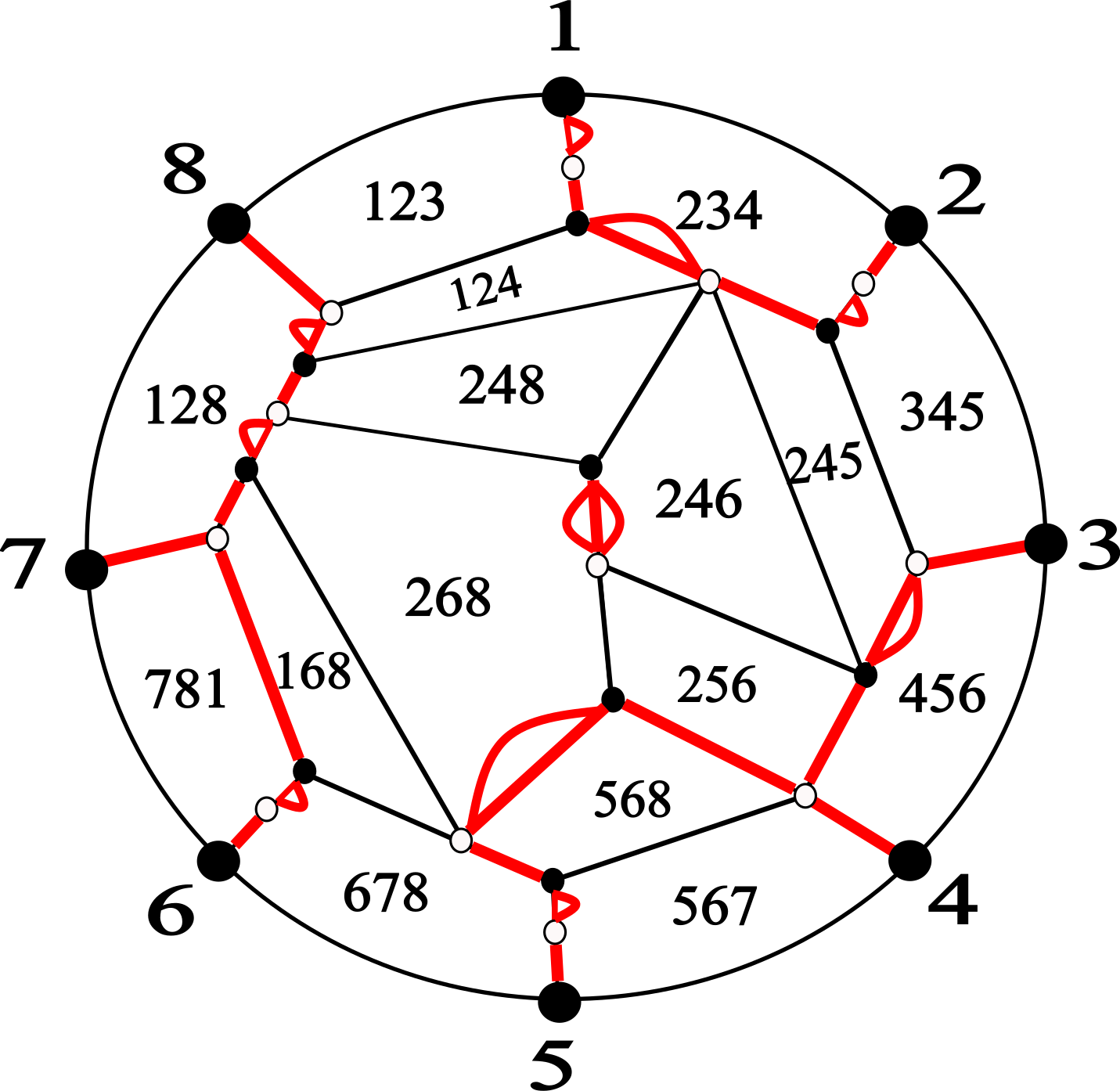}} & $\frac{(248)^2(168)(256)(123)^2(345)(567)(678)}{(124)(268)(568)}$ & $(248)^3(256)(268)(168)^2(123)(567)$\\
       \parbox[c]{1em}{
      \includegraphics[width=1.5in]{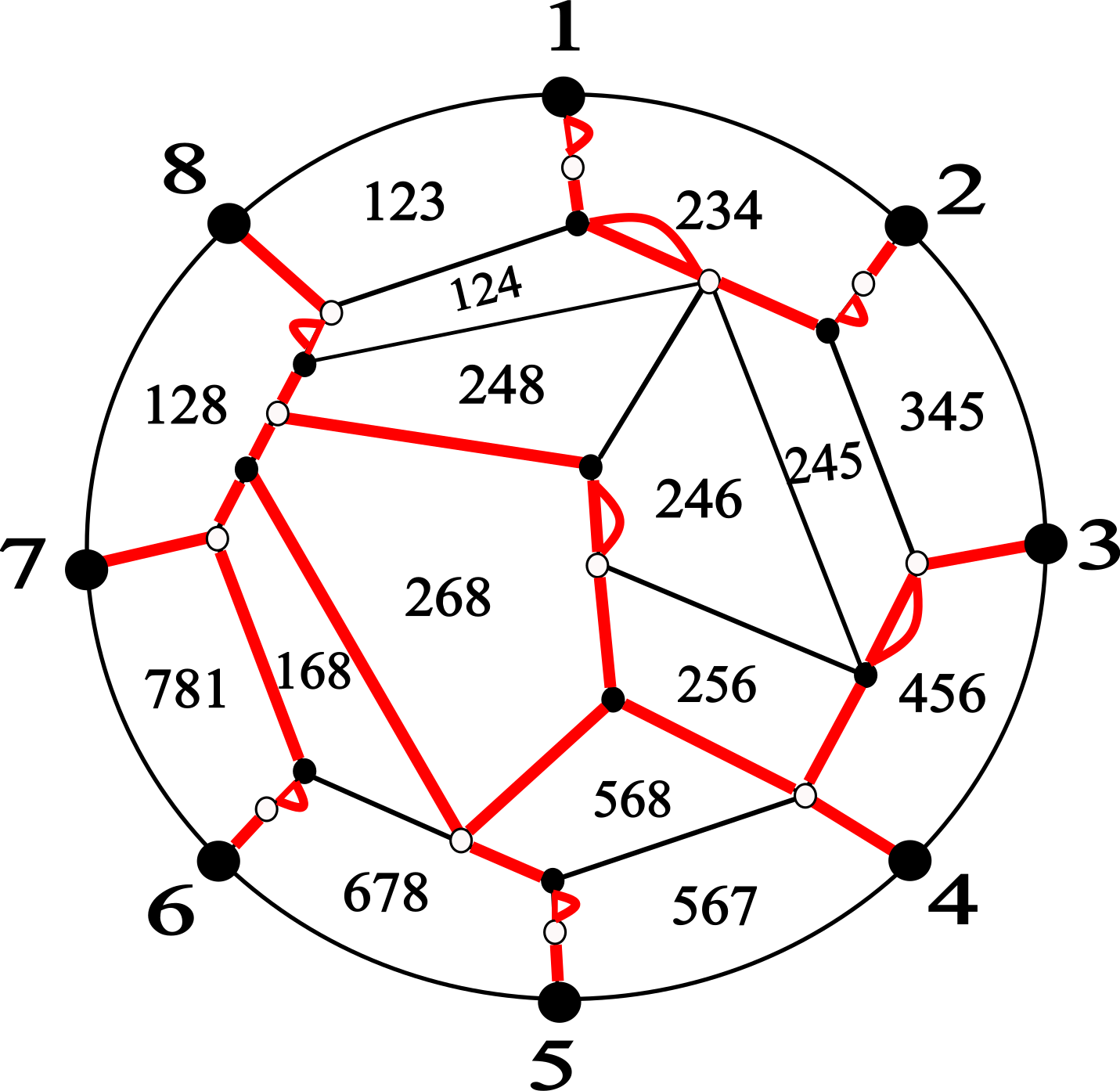}} & $\frac{(248)(246)(123)^2(345)(567)(678)(128)}{(124)(268)}$ & $(248)^2(246)(568)(268)(168)(128)(123)(567)$ \\
       \parbox[c]{1em}{
      \includegraphics[width=1.5in]{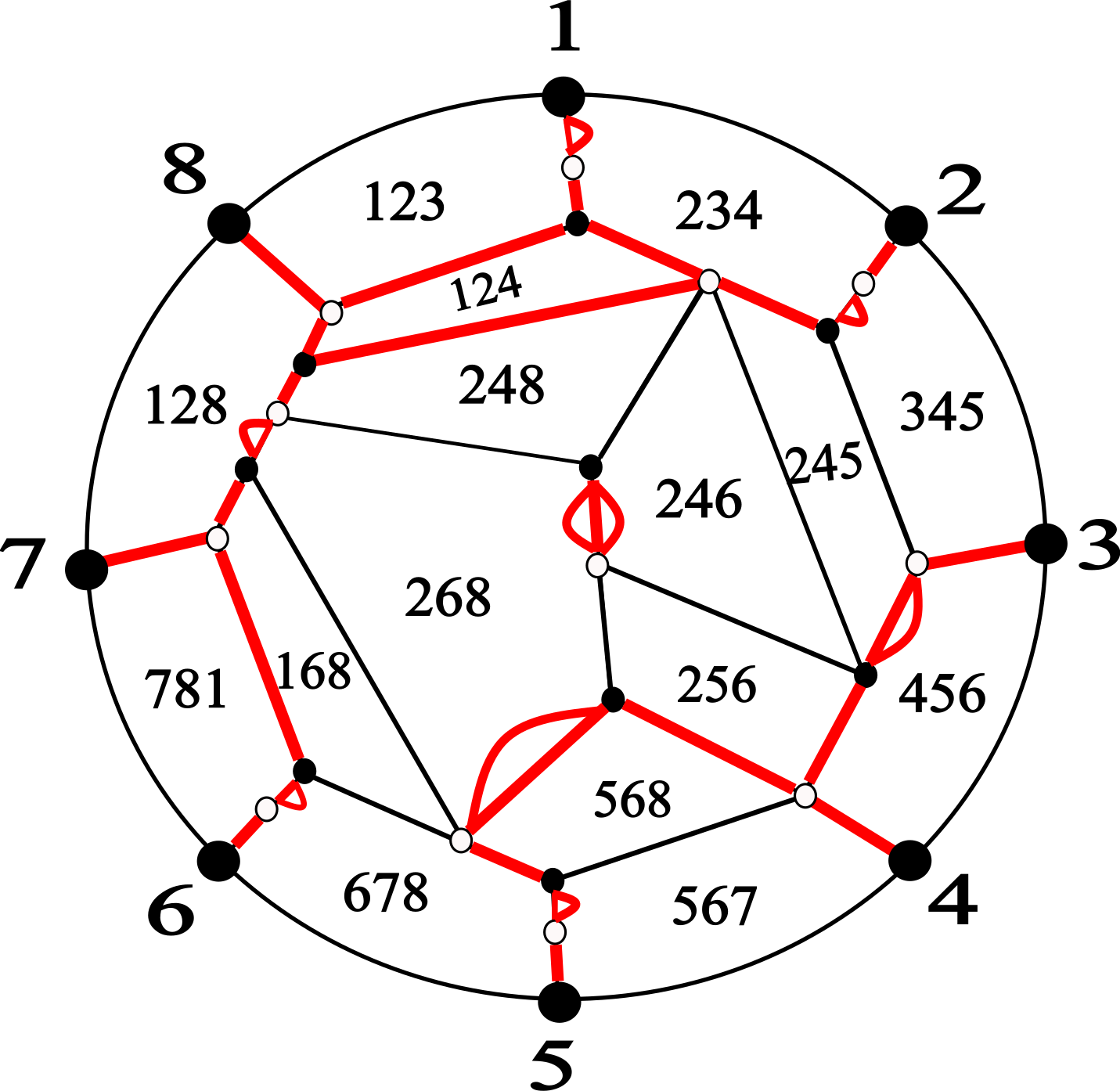}} & $\frac{(248)(168)(256)(123)(234)(345)(567)(678)(128)}{(124)(268)(568)}$ & $(248)^2(256)(268)(168)^2(128)(234)(567)$ \\
       \parbox[c]{1em}{
      \includegraphics[width=1.5in]{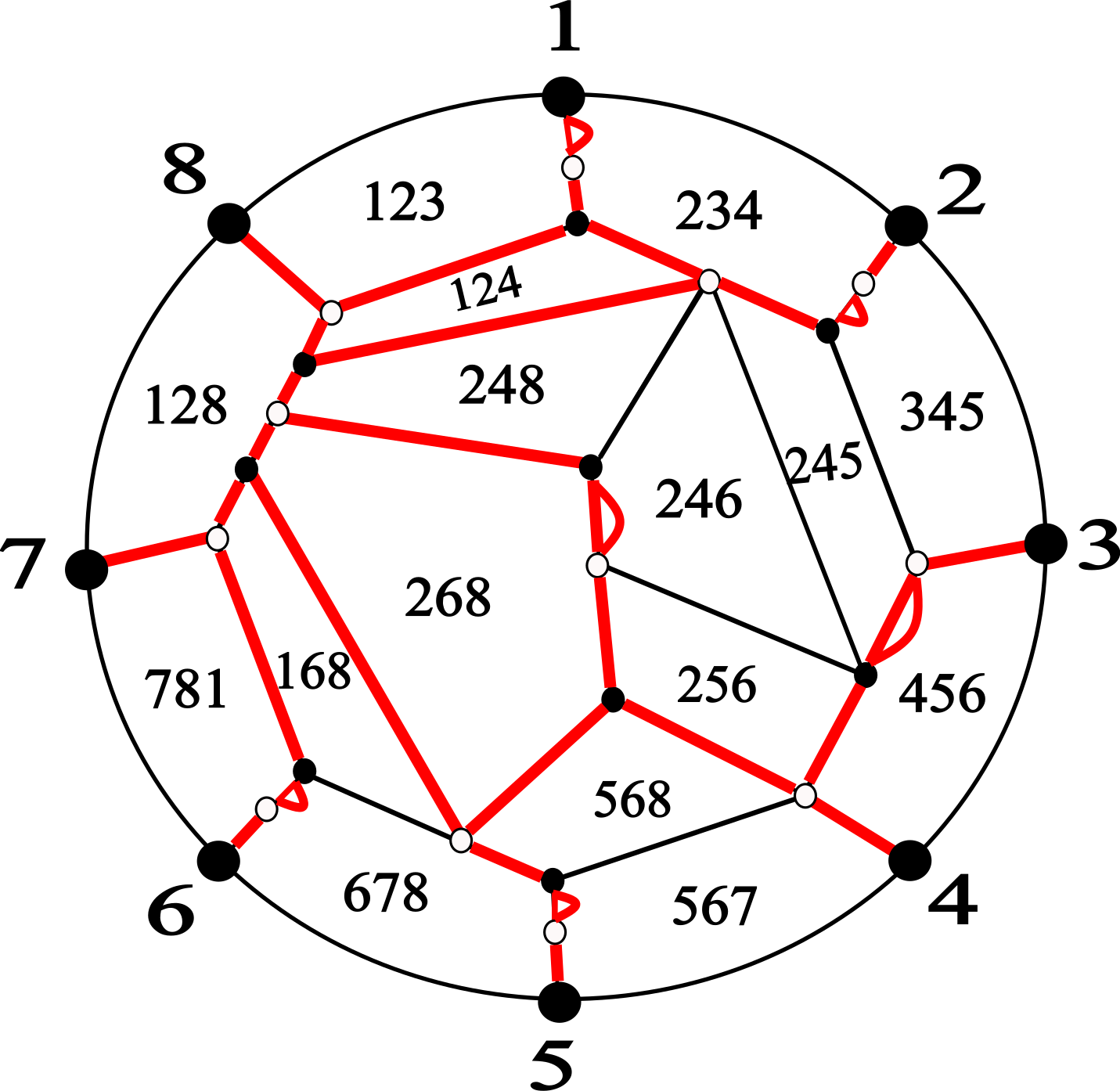}} & $\frac{(246)(123)(234)(345)(567)(678)(128)^2}{(124)(268)}$ & $(248)(246)(568)(268)(168)(128)^2(234)(567)$\\
       \parbox[c]{1em}{
      \includegraphics[width=1.5in]{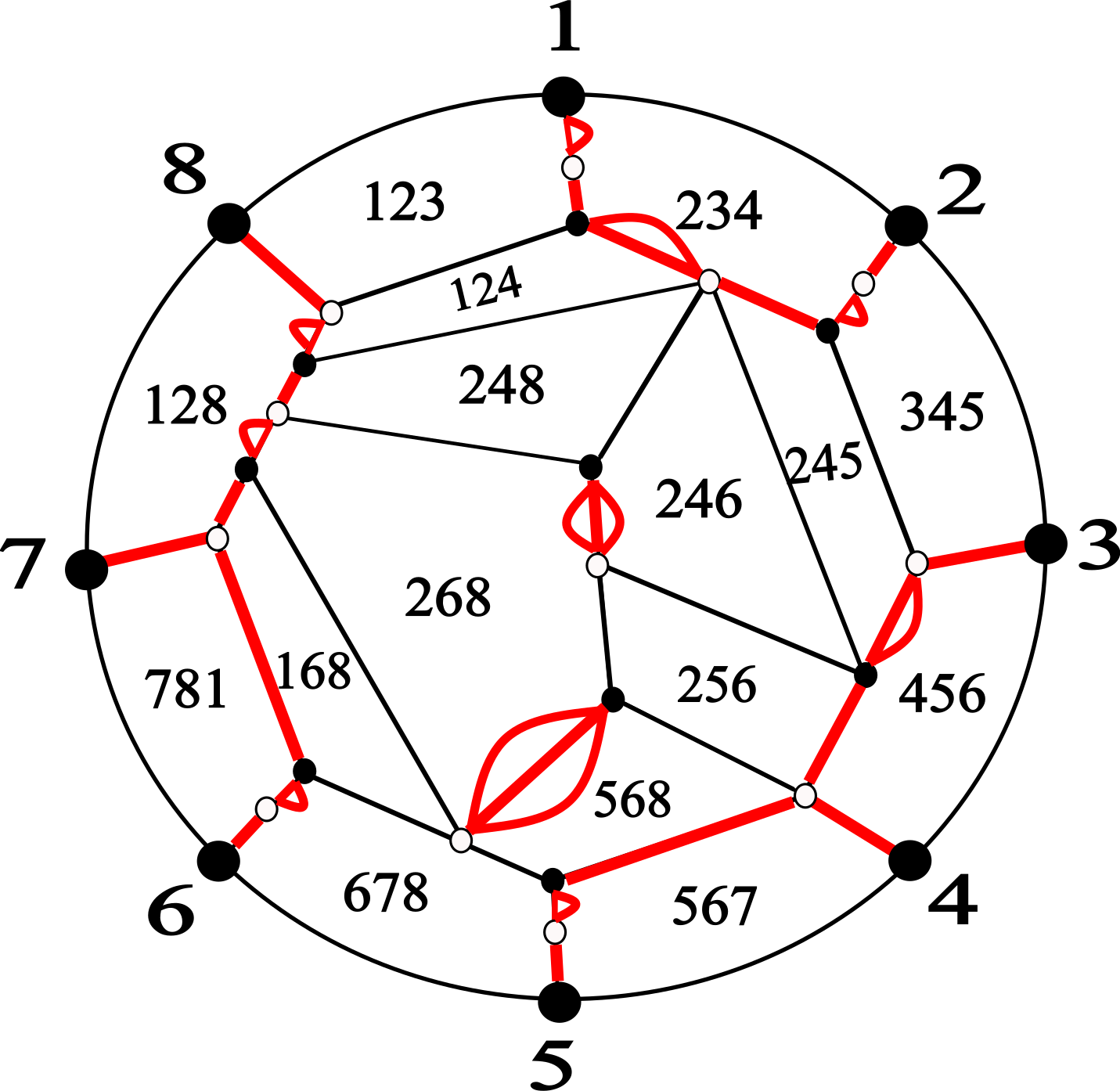}} & $\frac{(248)^2(168)(256)^2(123)^3(345)(678)^2}{(124)(268)^2(568)}$ & $(248)^3(256)^2(168)^2(123)(678)$ \\
      \parbox[c]{1em}{
      \includegraphics[width=1.5in]{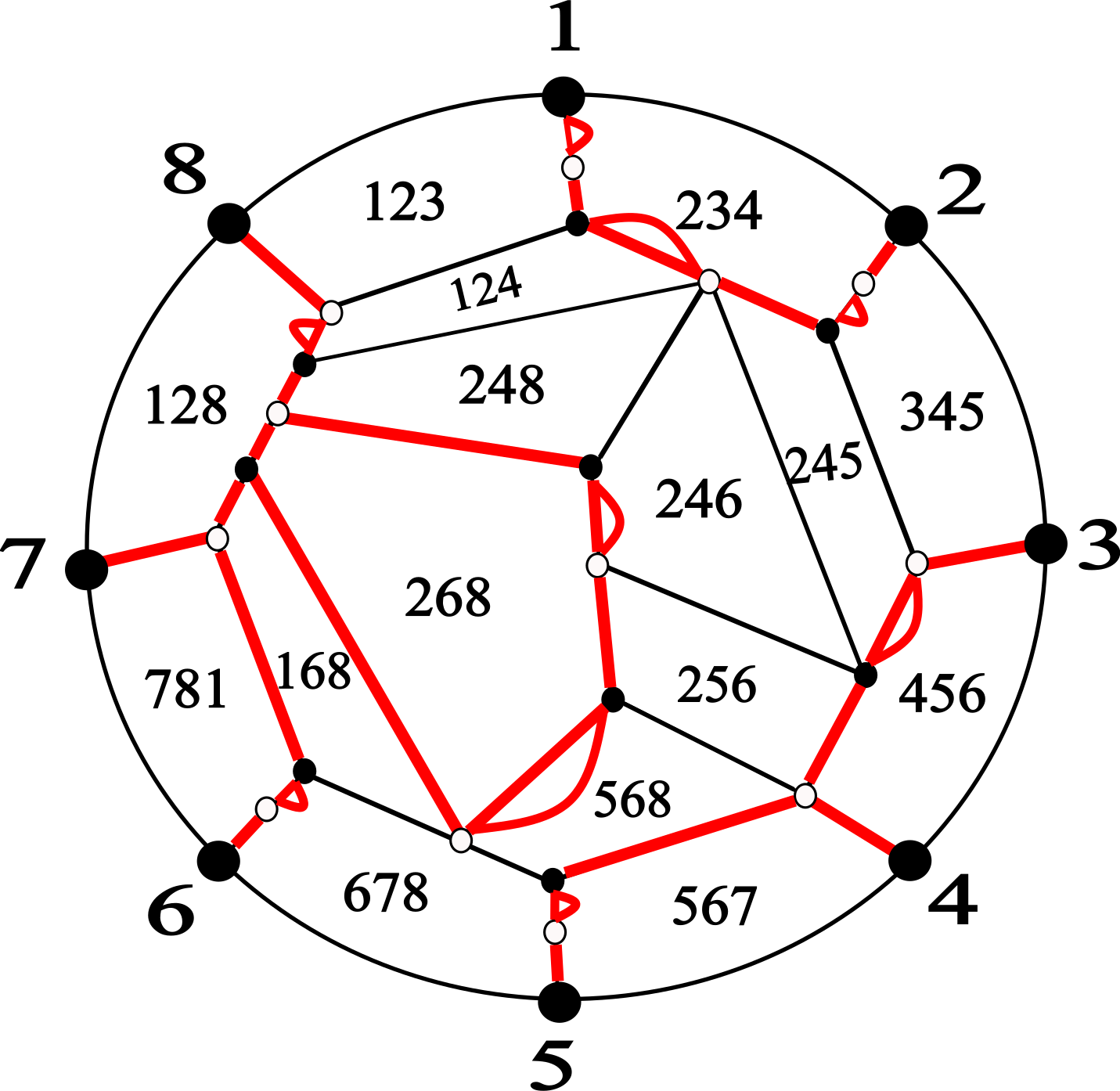}} & $2 \cdot \frac{(248)(246)(256)(123)^2(345)(678)^2(128)}{(124)(268)^2}$ & $2 \cdot (248)^2(246)(256)(568)(168)(128)(123)(678)$ \\
  \end{tabular}
\end{table}

\begin{table}[htbp] 
\caption{Triple Dimers for $\T(\sigma^7(B))$ Continued.} \label{tab:dimersforBcont}
  \begin{tabular}
      {lcc} \hline Triple Dimer Configuration & Weight of Dimer & Term in Laurent Expansion\\
      \hline
       \parbox[c]{1em}{
      \includegraphics[width=1.5in]{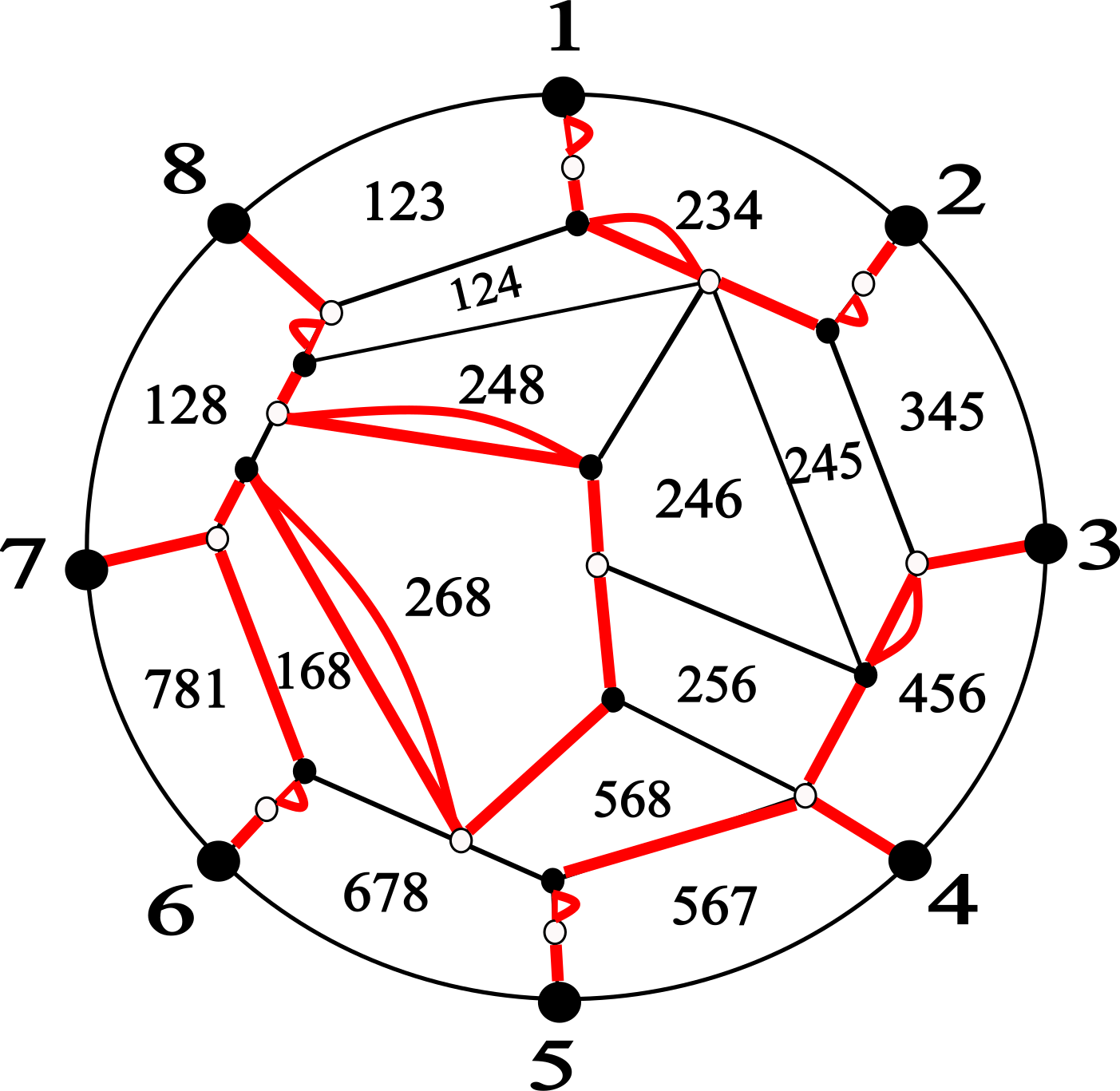}} & $\frac{(246)^2(568)(123)^2(345)(678)^2(128)^2}{(124)(268)^2(168)}$ & $(248)(246)^2(568)^2(128)^2(123)(678)$\\
       \parbox[c]{1em}{
      \includegraphics[width=1.5in]{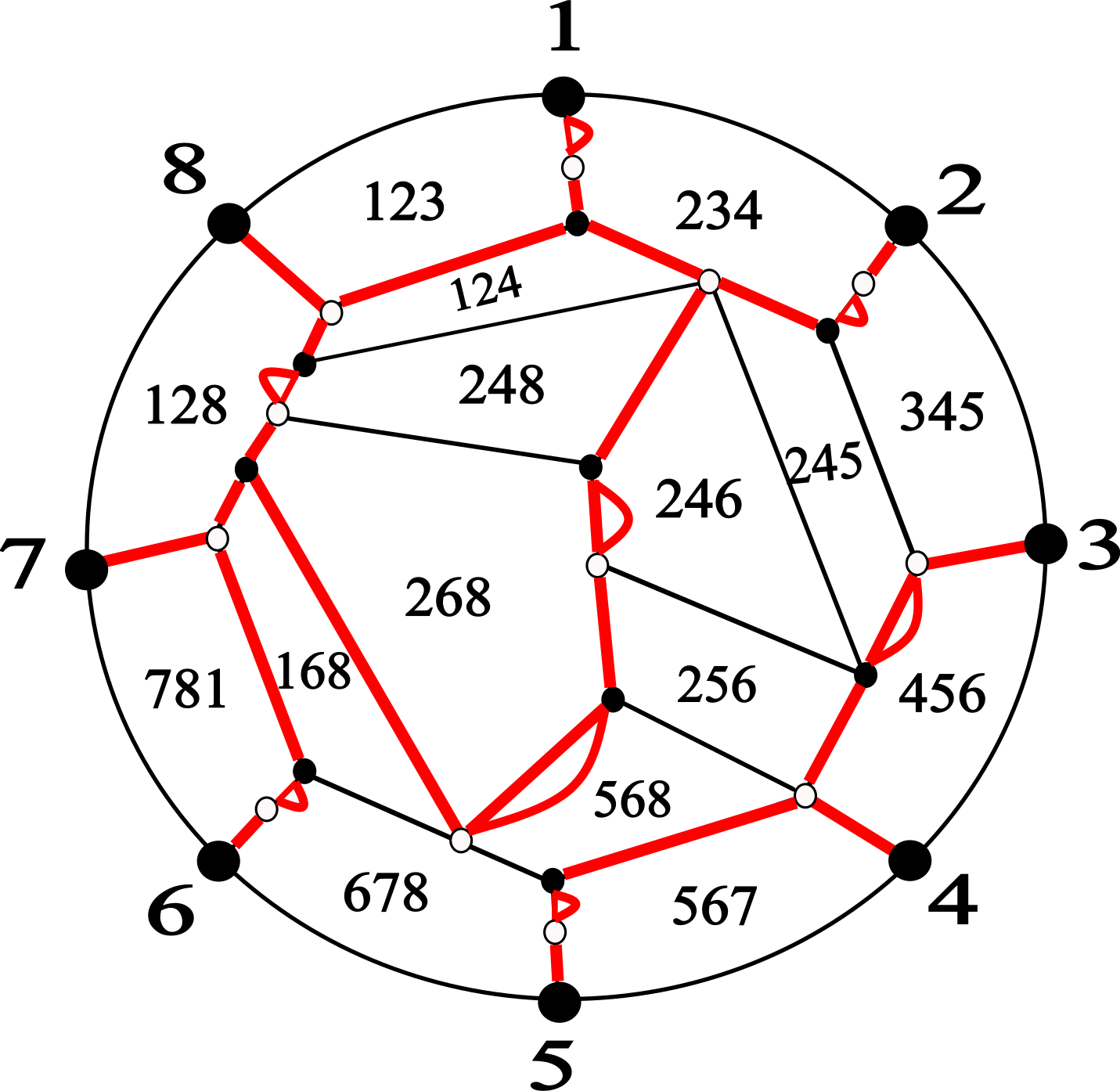}} & $\frac{(256)(123)(234)(345)(678)^2(128)}{(268)}$ & $(248)(124)(256)(568)(268)(168)(128)(234)(678)$ \\
       \parbox[c]{1em}{
      \includegraphics[width=1.5in]{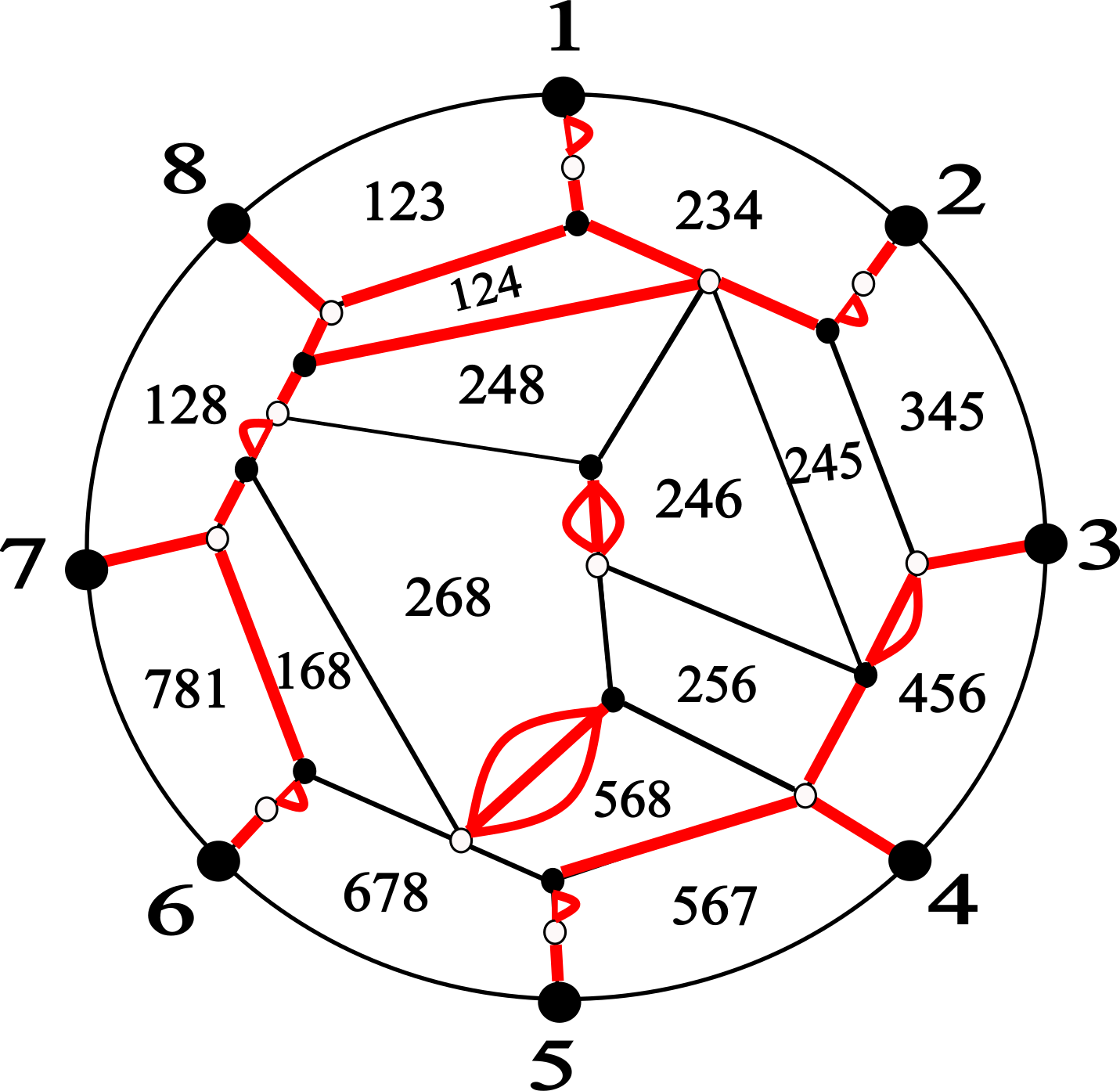}} & $\frac{9248)(168)(256)^2(123)(234)(345)(678)^2}{(124)(268)^2(568)}$ & $(248)^2(256)^2(168)^2(128)(234)(678)$ \\
       \parbox[c]{1em}{
      \includegraphics[width=1.5in]{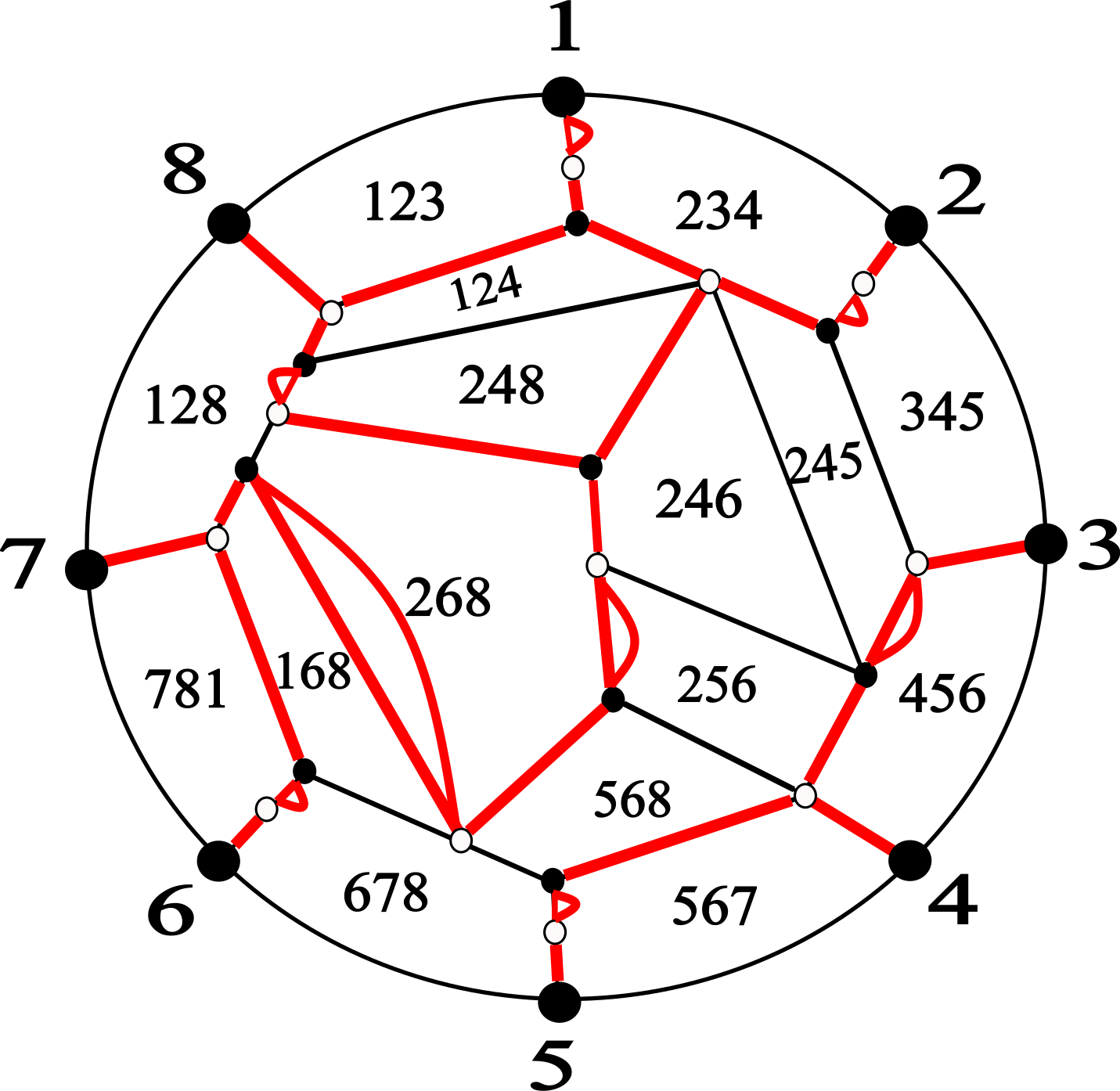}} & $\frac{(246)(568)(123)(234)(345)(678)^2(128)^2}{(248)(268)(168)}$ & $(124)(246)(568)^2(268)(128)^2(234)(678)$\\
       \parbox[c]{1em}{
      \includegraphics[width=1.5in]{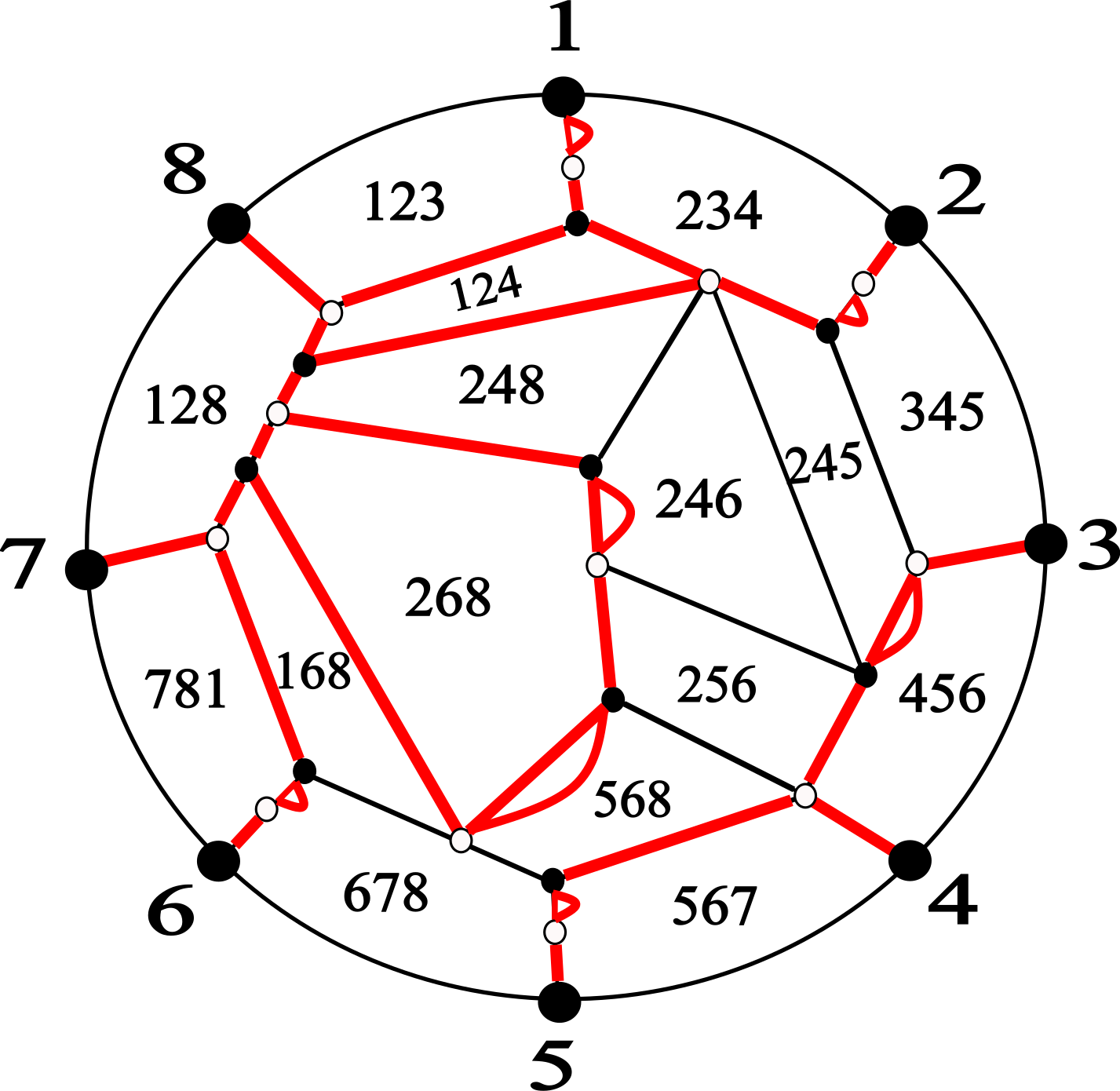}} & $2 \cdot \frac{(246)(256)(123)(234)(345)(678)^2(128)^2}{(124)(268)^2}$ & $2 \cdot (248)(246)(256)(568)(168)(128)^2(234)(678))$ \\
      \parbox[c]{1em}{
      \includegraphics[width=1.5in]{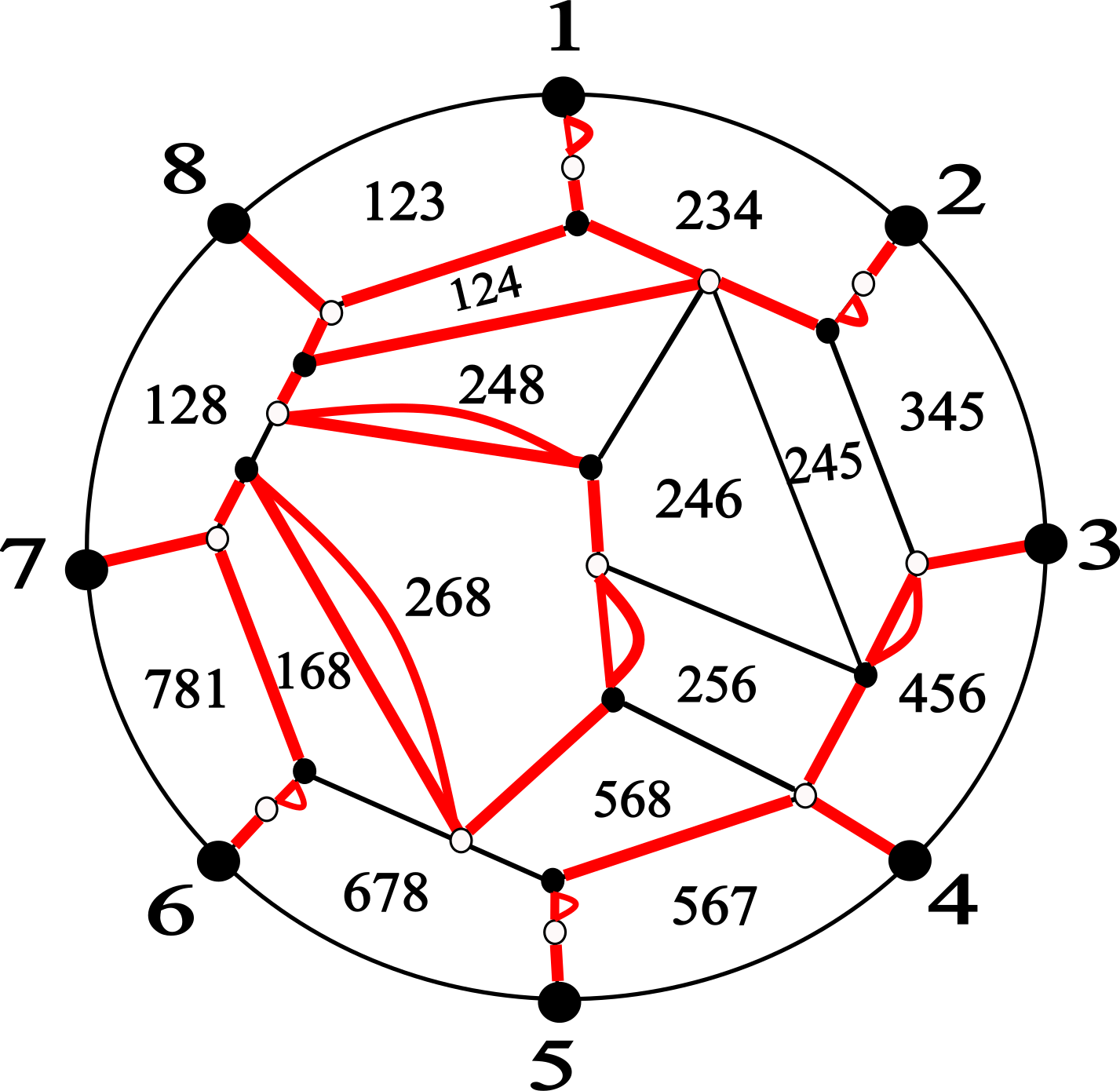}} & $\frac{(246)^2(568)(123)(234)(345)(678)^2(128)^3}{(124)(248)(268)^2(168)}$ & $(246)^2(568)^2(128)^3(234)(678)$ \\
  \end{tabular}
\end{table}

\begin{table}
  [htbp] \caption{Triple Dimers for $\T(\sigma^7(B))$ Continued.} \label{tab:dimersforBcontcont}
  \begin{tabular}
      {lcc} \hline Triple Dimer Configuration & Weight of Dimer & Term in Laurent Expansion\\
      \hline
       \parbox[c]{1em}{
      \includegraphics[width=1.5in]{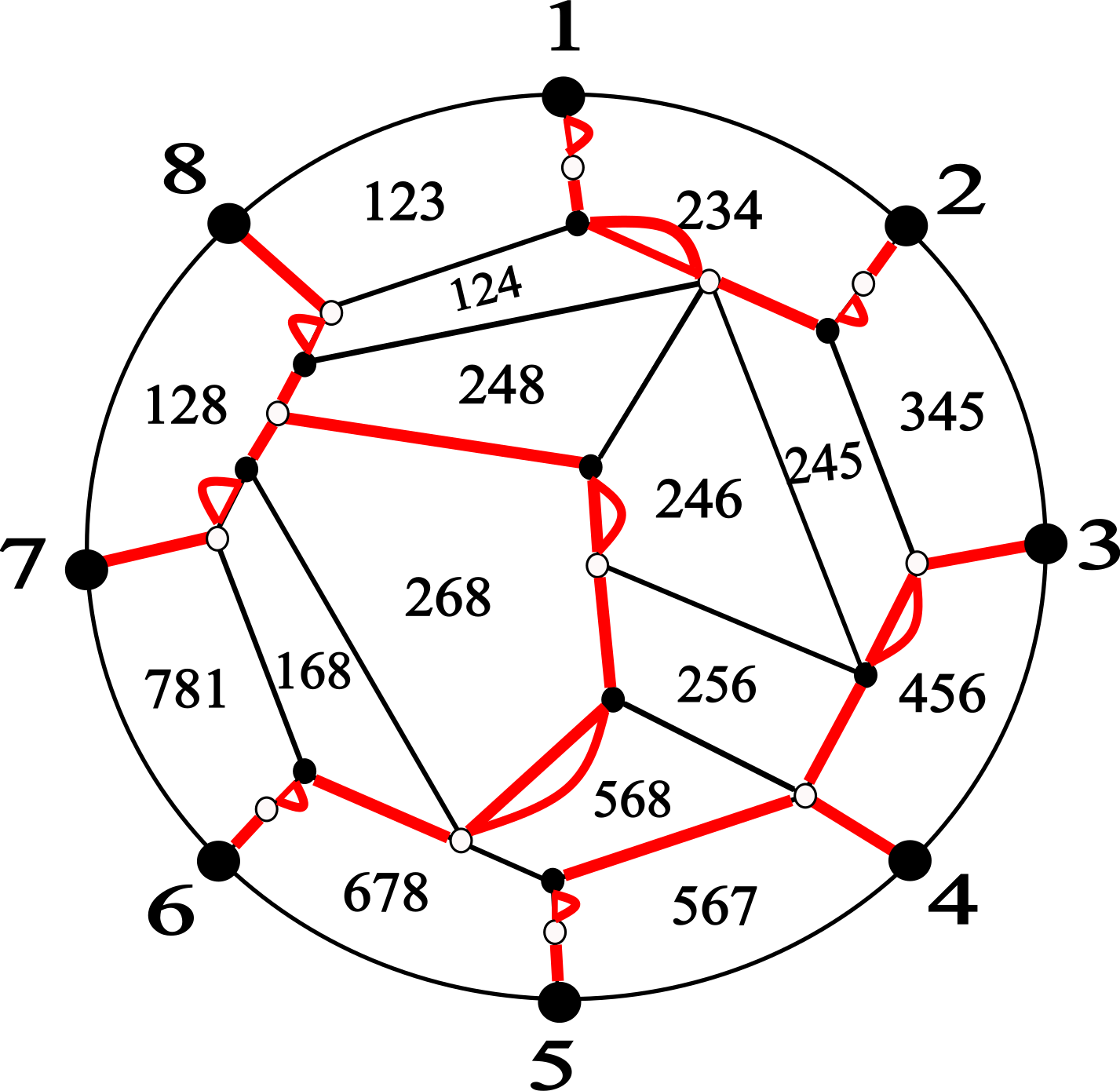}} & $\frac{(248)(246)(256)(123)^2(345)(678)(178)}{(124)(268)}$ & $(248)^2(246)(256)(568)(268)(168)(123)(178)$\\
       \parbox[c]{1em}{
      \includegraphics[width=1.5in]{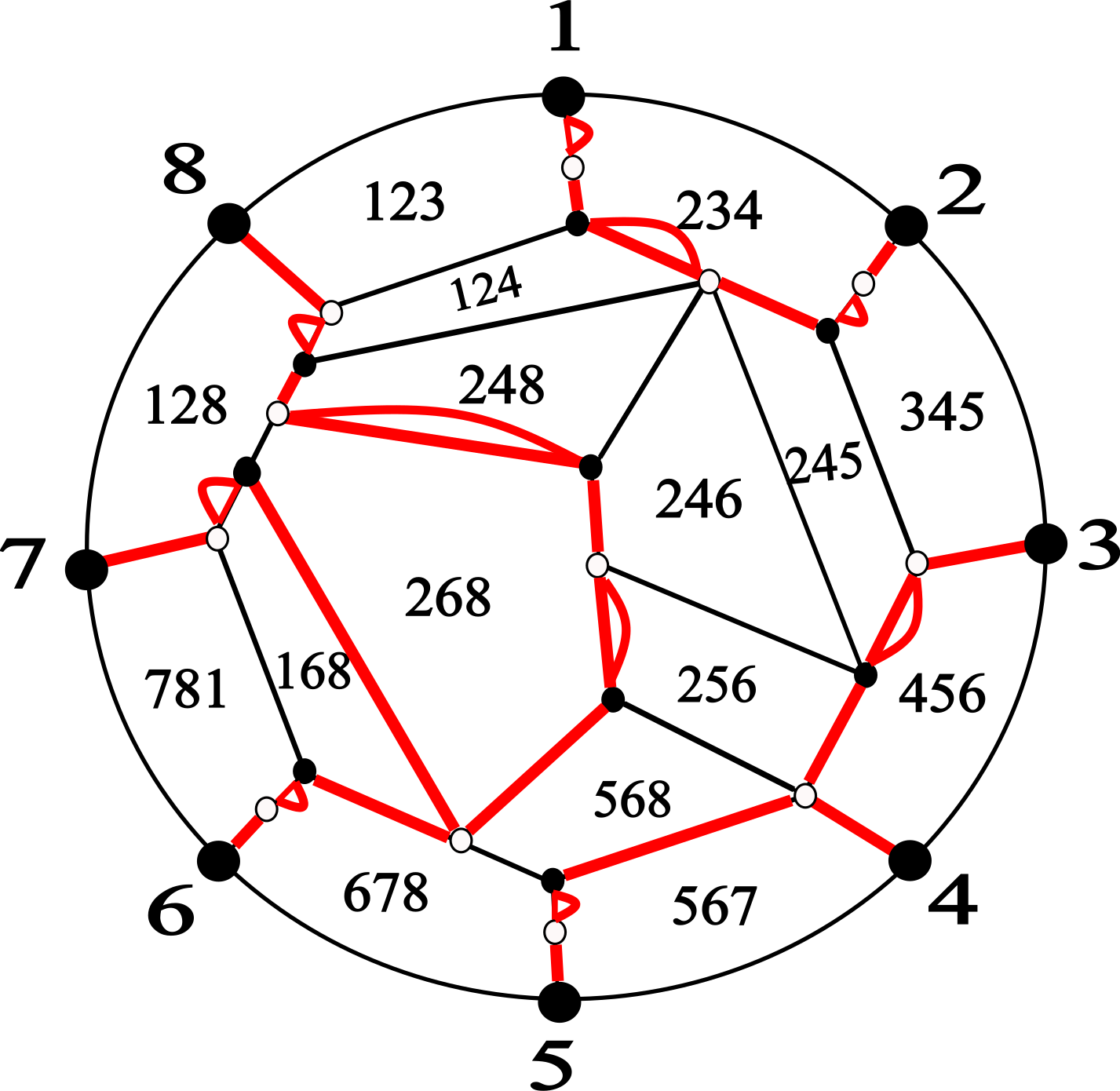}} & $\frac{(246)^2(568)(123)^2(345)(678)(178)(128)}{(124)(268)(168)}$ & $(248)(246)^2(568)^2(268)(128)(123)(178)$ \\
       \parbox[c]{1em}{
      \includegraphics[width=1.5in]{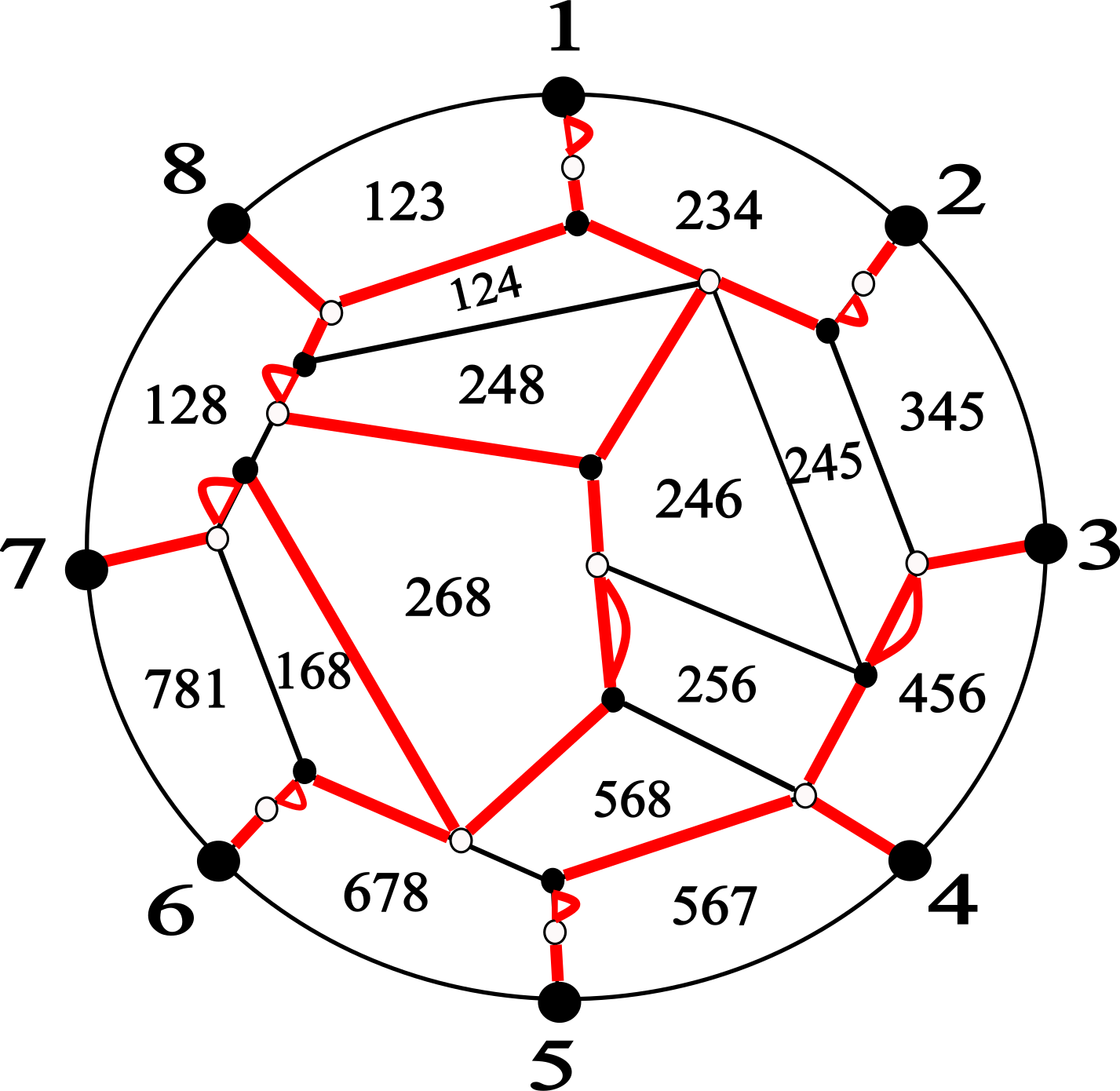}} & $\frac{(246)(568)(123)(234)(345)(678)(178)(128)}{(248)(168)}$ & $(124)(246)(568)^2(268)^2(128)(234)(178)$ \\
       \parbox[c]{1em}{
      \includegraphics[width=1.5in]{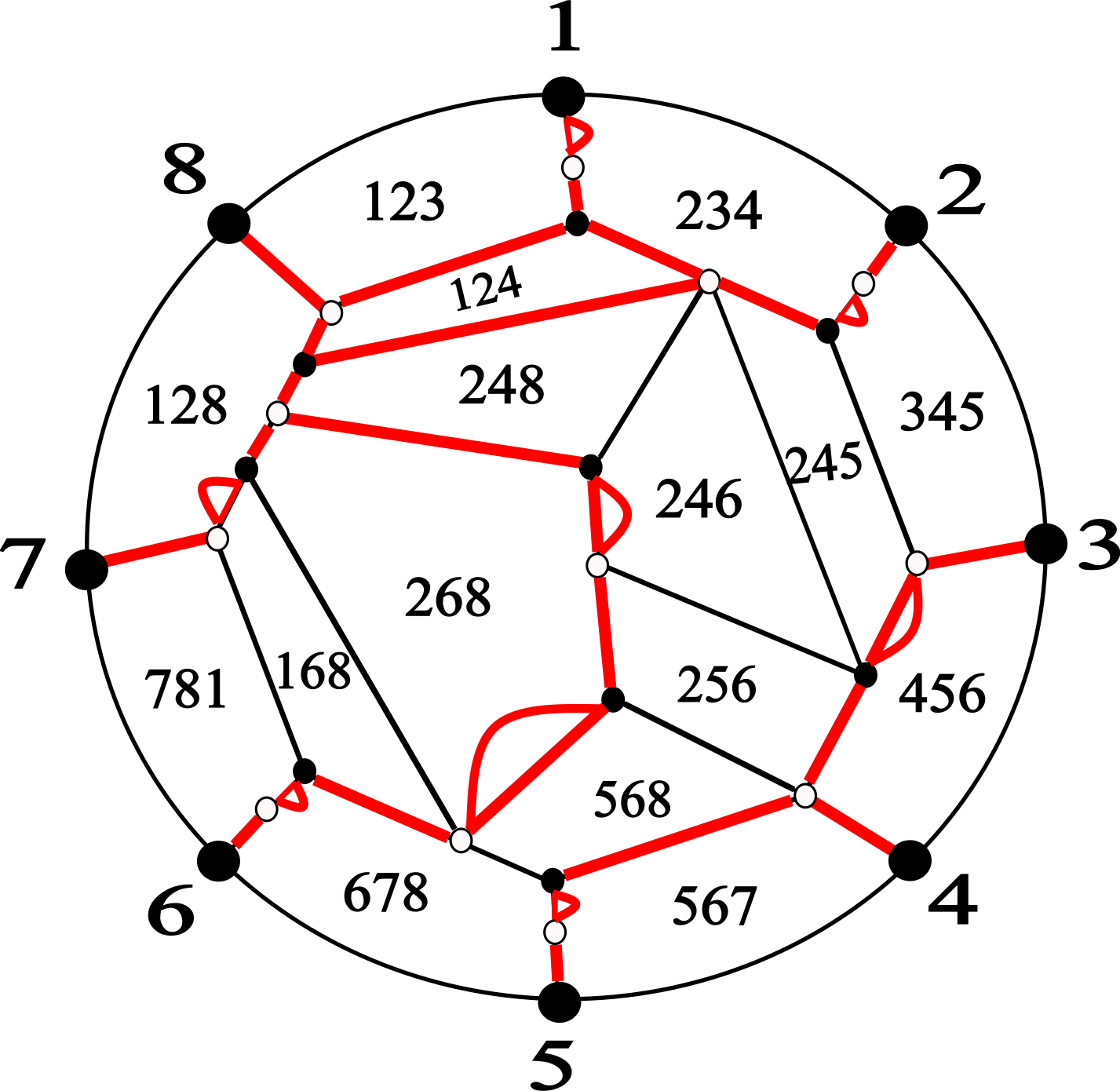}} & $\frac{(246)(256)(123)(234)(245)(678)(178)(128)}{(124)(268)}$ & $(248)(246)(256)(568)(268)(168)(128)(234)(178)$\\
       \parbox[c]{1em}{
      \includegraphics[width=1.5in]{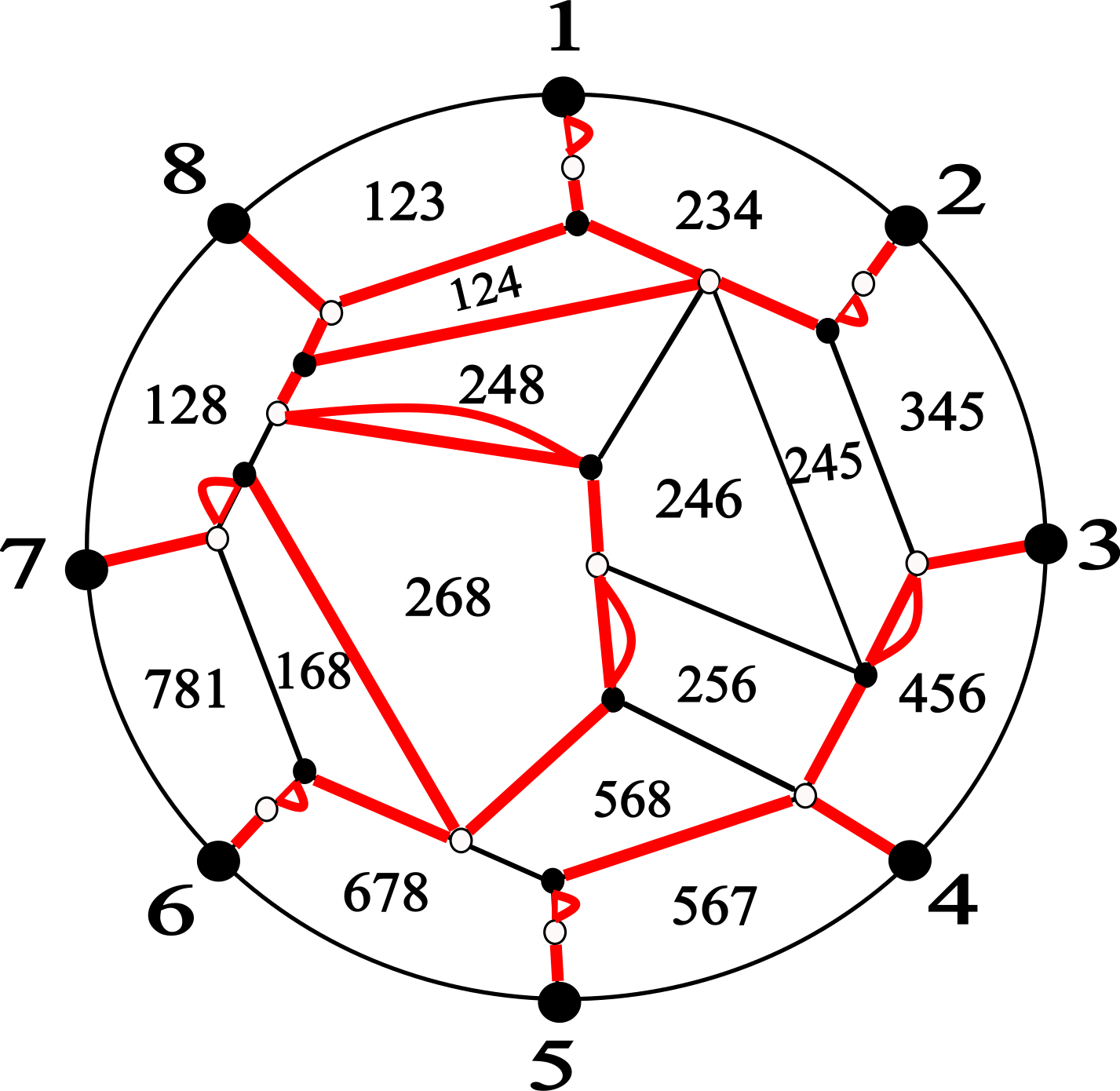}} & $\frac{(246)(568)(123)(234)(345)(678)(178)(128)^2}{(124)(248)(268)(168)}$ & $(246)^2(568)^2(268)(128)^2(234)(178)$ \\
  \end{tabular}
\end{table}
\end{section}

\newpage 

\FloatBarrier
\bibliographystyle{alpha}
\bibliography{main}
\end{document}